\documentclass[12pt]{amsart}

\usepackage{amsmath,amsthm,amscd,amsfonts,amssymb,epic,eepic,bbm,tikz-cd,mathtools,mathrsfs,longtable}
\usepackage[pagebackref,colorlinks=true,linkcolor=blue,citecolor=blue]{hyperref}

\allowdisplaybreaks
\newtheorem{thm}{Theorem}[section]
\newtheorem{cor}[thm]{Corollary}
\newtheorem{conj}[thm]{Conjecture}
\newtheorem{lem}[thm]{Lemma}

\newtheorem{prop}[thm]{Proposition}
\theoremstyle{remark}
\newtheorem*{rem}{Remark}
\newtheorem*{example}{Example}

\newcounter{remarkscounter}
\newenvironment{remarks}
{\medskip\noindent{\it
Remarks.}\begin{list}{{\rm(\arabic{remarkscounter})}
}{\usecounter{remarkscounter}

\setlength{\labelsep}{\fill} \setlength{\leftmargin}{0pt}
\setlength{\itemindent}{\fill}
\setlength{\labelwidth}{\fill}\setlength{\topsep}{0pt}
\setlength{\listparindent}{0pt}}} {\end{list}}

\numberwithin{equation}{section}
\newcommand{\A}{\mathbb{A}}
\newcommand{\GL}{\mathrm{GL}}

\newcommand{\SL}{\mathrm{SL}}
\newcommand{\ZZ}{\mathbb{Z}}

\newcommand{\FF}{\mathbb{F}}

\newcommand{\QQ}{\mathbb{Q}}

\newcommand{\lto}{\longrightarrow}

\newcommand{\hooklto}{\lhook\joinrel\longrightarrow}
\newcommand{\OO}{\mathcal{O}}
\newcommand{\CC}{\mathbb{C}}
\newcommand{\RR}{\mathbb{R}}
\newcommand{\GG}{\mathbb{G}}

\newcommand{\quash}[1]{}

\theoremstyle{definition}

\newenvironment{psmatrix}
  {\left(\begin{smallmatrix}}
  {\end{smallmatrix}\right)}

\renewcommand{\bar}{\overline}

\numberwithin{equation}{subsection}

\newcommand{\one}{\mathbbm{1}}

\title[Schubert Eisenstein series and Poisson summation]{Schubert Eisenstein series and Poisson summation for Schubert varieties}

\author{YoungJu Choie}
\address{Department of Mathematics\\
Pohang University of Science and Technology  \\
Pohang, South Korea 37673}
\email{yjc@postech.ac.kr}

\author{Jayce R. Getz}
\address{Department of Mathematics\\
Duke University\\
Durham, NC 27708}
\email{jgetz@math.duke.edu}

\subjclass[2010]{Primary 11F70; Secondary 11F55, 11F85}
\keywords{Schubert Eisenstein series, Poisson summation conjecture, Schubert varieties}

\thanks{The first author was partially supported by
 2022R1A2B5B0100187113, NRF 2017R1A2B2001807 and BSRI-NRF 2021R1A6A1A10042944. The second author is partially supported by NSF grant DMS 1901883, and received support from Y.~Choie's NRF 2017R1A2B2001807 grant.  
}

\begin{document}

\begin{abstract}
Schubert Eisenstein series are defined by restricting the summation in a degenerate Eisenstein series to a particular Schubert variety.  In the case of $\GL_3$ over $\QQ$ Bump and the first author proved that these Schubert Eisenstein series have meromorphic continuations in all parameters and conjectured the same is true in general.  We revisit their conjecture and relate it to the program of Braverman, Kazhdan, Lafforgue, Ng\^o, and Sakellaridis aimed at establishing generalizations of the Poisson summation formula.  We prove the Poisson summation formula for certain schemes closely related to Schubert varieties and use it to refine and establish the conjecture of Bump and the first author  in many cases.
\end{abstract}

\maketitle

\tableofcontents

\section{Introduction}

In this paper we
prove the Poisson summation conjecture of Braverman-Kazhdan, Lafforgue, Ng\^o, and Sakellaridis for a particular family of varieties related to Schubert varieties (see Theorem \ref{thm:PS:intro}).  
We were motivated to prove this case of the conjectures due to the fact that it implies
  functional equations of Schubert Eisenstein series conjectured to exist  by Bump and the first author in \cite{Bump:Choie}, see Theorem \ref{thm:ES:intro}. We begin by recalling the definition of Schubert Eisenstein series and then move to a discussion of the Poisson summation formula.

\subsection{Generalized Schubert Eisenstein series} Let $F$ be a global field with ring of adeles $\A_F,$ and let $G$ be a (connected) split semisimple group over $F.$  We let
$$
T \leq B \leq P \leq G
$$
be a maximal split torus, a Borel subgroup, and a parabolic subgroup of $G.$  Moreover we let
$$
T \leq M \leq P
$$
be the Levi subgroup of $P$ containing $T$.  
Let $N_G(T)$ be the normalizer of $T$ in $G$ and let $
W(G,T):=N_G(T)(F)/T(F)
$
be the  Weyl group of $T$ in $G.$  Finally we let 
$$
\omega_P :M^{\mathrm{ab}} \lto  \GG_m^{k+1}
$$
be the isomorphism of \eqref{omegaP}.   Here $M^{\mathrm{ab}}=M/M^{\mathrm{der}}$ is the abelianization of $M$, where $Q^{\mathrm{der}}$ is the derived group of an algebraic group $Q$.   

Let $A_{\GG_m} <F_\infty^\times$ be the usual subgroup (see \eqref{AGm}) and let
 $\chi:  
(A_{\GG_m} F^\times \backslash \A_F^\times)^{k+1} \to \CC^\times$ be a character.
For $s=(s_i) \in \CC^{k+1}$ we define
\begin{eqnarray}\label{cha}
    \chi_s(a_0,\dots,a_k):=\chi(a_0,\dots,a_k)\prod_{i=0}^k|a_i|^{s_i}.
\end{eqnarray}

We form the induced representation
$$
I_P(\chi_s):=\mathrm{Ind}_{P}^G(\chi_s \circ \omega_P),
$$
normalized so that it is unitary when $s \in (i\RR)^{k+1}.$
 Using the  Bruhat decomposition of $G$
 one has a decomposition of the generalized flag variety
 $$
 P \backslash G=\coprod_{w \in W(M,T) \backslash W(G,T)}P \backslash PwB.
 $$
 Here we have used the same symbol $w$ for a class in $W(M,T) \backslash W(G,T)$ and for a representative of that class.  Let $X_w$ be the (Zariski) closure of the Schubert cell $P \backslash PwB$ in $
 P \backslash G.$  It is a Schubert variety.  
 
The Schubert Eisenstein series attached to a section $\Phi^{\chi_s} \in I_P(\chi_s),$ $g \in G(\A_F)$ and $w \in W(G,T)$ is defined as
\begin{align} \label{SE}
SE_w(g,\Phi^{\chi_s}):=\sum_{\gamma\in X_w(F)} \Phi^{\chi_s}(\gamma g).
\end{align}
 It converges absolutely provided that $\mathrm{Re}(s)$ lies in a suitable cone.  The function $SE_w(g,\Phi^{\chi_s})$ is no longer left $G(F)$-invariant.  However, it is invariant under the stabilizer of the Schubert cell $X_w$ under the natural action of $G$ on $P \backslash G.$  Since this stabilizer contains $B,$ it is a parabolic subgroup, and it is often larger than $B.$

In \cite{Bump:Choie} Bump and the first author posed the following questions:

\begin{enumerate}
    \item[(a)] Do Schubert Eisenstein series admit a meromorphic continuation to all $s$?
    \item[(b)]
    Do a subset of the functional equations for the full Eisenstein series continue to hold for Schubert Eisenstein series?
    \item[(c)] Is it possible to find a linear combination of Schubert Eisenstein series which is entire?
\end{enumerate}

Let us pause and point out how striking the functions \eqref{SE} and these questions from \cite{Bump:Choie} are. Schubert Eisenstein series are a novel modification of degenerate Eisenstein series that enjoys many of the properties of the full degenerate Eisenstein series. Prior to the work of Bump and Choie, this modification seems to have been missed, despite the fact that degenerate Eisenstein series have been intensely studied for over a century.  

We further generalize the definition of Schubert Eisenstein series in higher rank in a manner that can be specialized to the definition in the $G=\SL_3$ case treated in \cite{Bump:Choie}. 
In this much greater generality we answer questions (a), (b) and (c) affirmatively
when we regard $SE_w(g,\Phi^{\chi_s})$ as a function of $s_0$ and assume the $s_i$ with $i \neq 0$ are fixed with large real part.  We refer to Theorem \ref{thm:ES:intro} and the subsequent remarks for details.   

 Before proceeding, let
 us examine the setting considered in \cite{Bump:Choie} and isolate 
 the change in viewpoint  that is the germ of our work.  Technically speaking in loc.~cit.~the authors work with $\GL_3$ whereas we work with $\mathrm{SL}_3$, but the translation between the two is straightforward.  Let $G=\mathrm{SL}_3$ and let $B<\SL_3$ be the Borel subgroup of upper triangular matrices.  From the point of view of this paper the most interesting Schubert variety in $B \backslash G$ is $B \backslash \overline{B\sigma_1\sigma_2B}$ where
\begin{align} \label{sig1sig2}
\sigma_1=\begin{psmatrix} & 1&   \\ -1 & & \\ &   & & 1\end{psmatrix}, \quad 
\sigma_2= \begin{psmatrix} 1 &  &   \\   & & -1 \\   &1& \end{psmatrix}.
\end{align}
In this case using standard facts on the Bruhat ordering \cite[\S 8.5.4-8.5.5]{Springer:LAG} one has
\begin{align} \label{Bruhat:decomp} 
\overline{B\sigma_1\sigma_2B}=B\sigma_1\sigma_2B \amalg B\sigma_1B \amalg B\sigma_2B \amalg  B.
\end{align}
Let $R$ be an $F$-algebra.  Then 
\begin{align} \label{closure}
    \overline{B\sigma_1\sigma_2B}(R)=\left\{\begin{psmatrix} a & b & c\\ d & e & f\\ & g & h \end{psmatrix}  \in \mathrm{SL}_3(R)\right\}.
\end{align}
Indeed, the set on the right is the $R$-points of an irreducible closed subscheme of $\mathrm{SL}_3$ that contains the irreducible closed subscheme $\overline{B\sigma_1\sigma_2B},$ hence we deduce equality by considering dimensions.  
If we let
\begin{align*}
    P_{2,1}(R):=\left\{\begin{psmatrix} a & b & c\\ d & e & f\\ &  & h \end{psmatrix}  \in \mathrm{SL}_3(R)\right\}, \quad
    P_{1,2}(R):=\left\{\begin{psmatrix} a & b &c \\  & e & f\\ & g & h \end{psmatrix}  \in \mathrm{SL}_3(R)\right\}
\end{align*}
then
\begin{align} \label{new:realiz}
\overline{B\sigma_1\sigma_2B}=P_{2,1}\sigma_1\sigma_2P_{1,2}.
\end{align}
Indeed, it is easy to see that each Bruhat cell on the right of \eqref{Bruhat:decomp} is contained in $P_{2,1}\sigma_1\sigma_2P_{1,2},$ and it is clear on the other hand from \eqref{closure} that $P_{2,1}\sigma_1\sigma_2P_{1,2} \subseteq \overline{B\sigma_1\sigma_2B}.$
The natural map
$$
P_{2,1} \times P_{1,2} \lto \overline{B\sigma_1\sigma_2B}
$$
is a lift of the Bott-Samelson resolution of the image $X_{\sigma_1\sigma_2}$ of $\overline{B\sigma_1\sigma_2B}$ in $B \backslash \SL_3.$  This was the point of departure for the arguments in \cite{Bump:Choie}.    
In \cite{Bump:Choie} the first author and Bump suggest employing these methods to prove the meromorphy of Schubert Eisenstein series in higher rank, but we do not know how to execute this suggestion.

Instead of 
 pursuing Schubert varieties and the Bott-Samelson resolution to study higher rank analogues of the left hand side of \eqref{new:realiz}, we  generalize and study the right hand side directly.
Let us step back to consider the situation for a general split semisimple group $G.$ Consider the closure $\overline{PwB}$ in $G.$  It is the union of Bruhat cells $\cup_{w'\leq w} Pw'B,$ where $\leq$ denotes the Bruhat order \cite[\S 8.5.4-8.5.5]{Springer:LAG}.
 The group $G$ acts on itself on the left, and we let $Q$ be the stabilizer of the closed subscheme $\overline{PwB} \subset G$ (i.e. the algebraic subgroup of $G$ sending $\overline{PwB}$ to itself).  This is a closed algebraic subgroup of $G$ \cite[Corollary 1.81]{Milne:AGbook}.  The group $Q$ contains $P,$ and hence is a parabolic subgroup. 

To prove analytic properties of series indexed by $P(F) \backslash \overline{PwB}(F),$ we could consider any
  parabolic subgroup $P'$ with $P \leq P' \leq Q$ 
  and study series indexed by
 $$
 P'wB.
 $$
Since $\overline{PwB}(F)$ is a finite union of $Pw'B(F)$ as recalled above it is also a finite union of sets of the form $P'w'B(F).$   Thus if series indexed by sets of the form $P'wB(F)$ already admit meromorphic continuations the same is true of the Schubert Eisenstein series.  This is indeed the case, as we will show in Theorem \ref{thm:ES:intro} below.

In fact, our main result also applies to subschemes of $G$ of the form
  \begin{align} \label{P'gammaH}
  P' \gamma H
  \end{align}
 where $P \leq P' \leq G$ are a pair of parabolic subgroups, $\gamma \in G(F)$ and $H$ is an arbitrary algebraic subgroup of $G.$  
 Not all Schubert cells are the image in $P \backslash G$ of a set of the form \eqref{P'gammaH}.  Indeed, Schubert cells are often nonsmooth, whereas the image of any set of the form $P'\gamma H$ in $P \backslash G$ is a smooth subscheme (this follows from Lemma \ref{lem:CP:smooth}).

 In order to treat series indexed by sets of the form $P^{\mathrm{der}}(F) \backslash P\gamma H(F)$ and $P^{\mathrm{der}}(F) \backslash \bar{PwB}(F)$ simultaneously we work with an arbitrary reduced subscheme $Y \subseteq G$ that is stable under left multiplication by $P'.$ 
 Let $X_P^\circ:=P^{\mathrm{der}} \backslash G$ be the Braverman-Kazhdan space associated to $P$ and $G.$   Let
\begin{align} \label{Cgamma}
 Y_{P}=\mathrm{Im}(Y \lto X_{P}^\circ).
\end{align} 
To be more precise, the set theoretic image of $Y \to X_P^{\circ}$ is locally closed by Lemma \ref{lem:locally:closed}.  This set is the underlying topological space of a subscheme $Y_P$ of $X_P^{\circ}.$
The subscheme $Y_P \subseteq X_P^\circ$ is quasi-affine. We emphasize that the schemes $Y_P$ generalize both schemes of the form $P^{\mathrm{der}} \backslash P \gamma H$ and Schubert varieties.  In particular $Y_P$ need not be smooth. For some examples when $G=\mathrm{SL}_n$ we refer to \S \ref{sec:example}.

Let $X_P$ be the affine closure of $X_P^\circ$ and let
\begin{align} \label{Xgammacirc}
Y_{P,P'}\subseteq X_{P}
\end{align}
be the partial closure of $Y_{P}$ in $X_P$ constructed in \eqref{partial:closure} below.
We observe that there is a natural action of $M^{\mathrm{ab}}$ on $Y_{P,P'}$ that preserves $Y_{P}$ (see \eqref{geo:action}). 

Without essential loss of generality we assume $G$ is simple and simply connected.  We additionally make the following technical assumption:
\begin{align} \label{P'assump}
    P \textrm{ is maximal in }P'.
\end{align} 
We will explain the motivation for this assumption after stating Theorem \ref{thm:ES:intro} below.
Under \eqref{P'assump} there is a unique parabolic subgroup $P^{*} <P'$ with Levi subgroup $M$ that is not equal to $P.$ 
Thus one might think of $P^{*}$ as the opposite parabolic of $P$ with respect to $P'.$

In \S \ref{sec:Ss} we define Schwartz spaces 
$$
\mathcal{S}(Y_{Q,P'}(\A_F))
$$
for $Q \in \{P,P^*\}$
together with a Fourier transform
\begin{align}
    \mathcal{F}_{P|P^*}:\mathcal{S}(Y_{P,P'}(\A_F)) \tilde{\lto} \mathcal{S}(Y_{P^*,P'}(\A_F)).
\end{align} 
The Schwartz space $\mathcal{S}(Y_{Q,P'}(\A_F))$ is contained in the set of restrictions to $Y_{P}(\A_F)$ of functions in $C^\infty(X_P^\circ(\A_F))$.  
Let $H \leq G$ be a subgroup, and consider the action of $H$ on $G$ by right multiplication.  Assume that $Y$ is stable under the action of $H.$
  Then the Schwartz spaces $\mathcal{S}(Y_{P,P'}(\A_F))$ and $\mathcal{S}(Y_{P^*,P'}(\A_F))$ are preserved under the action of $M^{\mathrm{ab}}(\A_F) \times H(\A_F)$ of \eqref{Sch:act} and the Fourier transform satisfies a twisted equivariance property by Lemma \ref{lem:equiv}.

Let $I_{P^*}^*(\chi_s):=\mathrm{Ind}_{P^*}^{G}(\chi_s \circ \omega_P).$  
The $*$ indicates that we are inducing  $\chi_s \circ \omega_P,$ not $\chi_s \circ \omega_{P^*}.$
The group $M^{\mathrm{ab}} $ acts on $Y_{P}$ and $Y_{P^*}$ on the left, and hence  we obtain Mellin transforms
\begin{align} 
\label{Mellin} \begin{split}
\mathcal{S}(Y_{P,P'}(\A_F)) &\lto I_P(\chi_s)|_{Y_{P}(\A_F)} \\
f &\longmapsto f_{\chi_s}(\cdot):=f_{\chi_s,P}(\cdot):=\int_{M^{\mathrm{ab}}(\mathrm{A}_F)}\delta_{P}^{1/2}(m)\chi_s(\omega_P(m))f(m^{-1}\cdot)dm, \\\mathcal{S}(Y_{P^*,P'}(\A_F)) &\lto I_{P^*}^*(\chi_s)|_{Y_{P^*}(\A_F)} \\
f &\longmapsto f_{\chi_s}^*(\cdot):=f_{\chi_s,P^*}^*(\cdot):=\int_{M^{\mathrm{ab}}(\mathbb{A}_F)}\delta_{P^*}^{1/2}(m)\chi_s(\omega_P(m))f(m^{-1}\cdot)dm.
\end{split}
\end{align}  
Here $\delta_Q$ is the modular quasi-character of an algebraic group $Q$.  The fact that the Mellin transform $f_{\chi_s}$ (resp.~$f_{\chi_s}^*$) is absolutely convergent for $\mathrm{Re}(s_0)$   large (resp. ~$\mathrm{Re}(s_0)$ small) is built into the definition of the Schwartz space.  

\begin{rem}
We will write $\Phi^{\chi_s}$ for a section of $I_P(\chi_s)$ that is not necessarily a Mellin transform of an element $f \in \mathcal{S}(Y_{P,P'}(\A_F)).$  We take the analogous convention in the local setting and when $I_P(\chi_s)$ is replaced by $I_{P^*}^*(\chi_s).$  
\end{rem}

For $f_1 \in \mathcal{S}(Y_{P,P'}(\A_F)),$ $f_2 \in \mathcal{S}(Y_{P^*,P'}(\A_F))$
we define \textbf{generalized Schubert Eisenstein series}
\begin{align} 
\label{SEis} \begin{split}
E_{Y_P}(f_{1\chi_s}):&=\sum_{y\in M^{\mathrm{ab}}(F) \backslash Y_{P}(F)}f_{1\chi_s}(y),\\
E_{Y_{P^*}}^*(f_{2\chi_s}^*):&=\sum_{y^* \in M^{\mathrm{ab}}(F) \backslash Y_{P^*}(F)}f_{2\chi_s}^*(y^*).\end{split}
\end{align}
These sums converge absolutely for $\mathrm{Re}(s_0)$ sufficiently large (resp.~small) provided that $\mathrm{Re}(s_1),\dots,\mathrm{Re}(s_k)$ are sufficiently large.  Here $(s_0,\dots,s_k)$ are used to define $\chi_s$ as in \eqref{cha}.  To help motivate this definition we point out that Lemma \ref{lem:surj} implies that
\begin{align*}
P(F) \backslash Y(F) &\lto M^{\mathrm{ab}}(F) \backslash Y_{P}(F)
\end{align*}
is a bijection.

\begin{thm} \label{thm:ES:intro}
Let $f \in \mathcal{S}(Y_{P,P'}(\A_F))$.  Assume that 
\begin{enumerate}
    \item $F$ is a function field, or
    \item $F$ is a number field and Conjecture \ref{conj:poles:intro} is valid.
\end{enumerate}
Fix $s_1,\dots,s_k$ such that $\mathrm{Re}(s_i)$ is sufficiently large for $1 \leq i \leq k.$
Then $E_{Y_P}(f_{\chi_s})$ and $E_{Y_{P^*}}(\mathcal{F}_{P|P^*}(f)_{\chi_s}^*)$ are meromorphic in $s_0$.  Moreover one has 
$$
E_{Y_P}(f_{\chi_s})=E^*_{Y_{P^*}}(\mathcal{F}_{P|P^*}(f)_{\chi_s}^*).
$$
\end{thm}

\noindent Conjecture \ref{conj:poles:intro} below states that certain normalized 
degenerate Eisenstein series have only finitely many poles.  It is known if $G=\mathrm{SL}_n$ and in several other cases (see the remark after the statement of the conjectures below).

Let $w \in G(F)$ and let
 $P'$ be the stabilizer of $\overline{PwB}$ under the left action of $G$.
 Let $M'$ be the Levi subgroup of $P'$ containg $M.$
 Then one expects a family of functional equations for $E_{P^{\mathrm{der}} \backslash \overline{PwB}}(f_{\chi_s})$ analogous to the family of functional equations for the Eisenstein series formed from the induction of a quasi-character of $P \cap M'(\A_F)$ to $M'(\A_F)$.
 In the special case where $P<P'$ is maximal, only one functional equation is expected, and Theorem \ref{thm:ES:intro} provides this, answering question (b).  It also provides the meromorphic continuation of the Eisenstein series in $s_0,$ answering the question (a) as a function of $s_0$.  The proof  also shows that we can choose a linear combination of Schubert Eisenstein series that is entire in $s_0$ for fixed $s_1,\dots,s_k$ all with large real part, answering question (c) as a function of $s_0$.  With some effort, our methods should generalize to treat the case where $P$ is not necessarily maximal in $P'.$

Bump and the first author were able to obtain the meromorphic continuation of Schubert Eisenstein series in the $\mathrm{SL}_3$ case in all complex parameters and more than one functional equation.  It would be interesting to see if to what extent this continues to hold in general.

\begin{remarks}
\item We have already seen that one can choose the data so that $Y_{P,P'}$ is a lift to $X_P$  of a partial closure of a Schubert cell.  Schubert varieties admit analogues in any Kac-Moody group, and there is the intriguing possibility that Theorem \ref{thm:ES:intro} can be generalized to this setting.  There have been tantalizing hints of interactions between classical automorphic forms and Kac-Moody groups, a nice summary is contained in the introduction to \cite{Garland:Miller:Patnaik}.  The fact that such groups are infinite dimensional is a source of difficulty.  The key observation here is that the Schubert cells in the flag varieties of Kac-Moody groups are finite-dimensional. 

\item It is an important problem to investigate whether generalized Schubert Eisenstein series can be used to produce integral representations of automorphic $L$-functions.  See \S \ref{ssec:int:reps} below for more details.  For example, one could try to
generalize the famous doubling method of Piatetski-Shapiro and Rallis \cite{GPSR:LNM}.

\item For some information about the possible poles of $E_{Y_P}(f_{\chi_s}),$ see Corollary \ref{cor:poles}.
\end{remarks}

\subsection{The Poisson summation conjecture}

To prove Theorem \ref{thm:ES:intro} we prove new cases of a seminal conjecture due to Braverman and Kazhdan \cite{BK-lifting}.  The conjecture was later investigated by Lafforgue \cite{LafforgueJJM} and refined by Ng\^o \cite{NgoSums, Ngo:Hankel}.  It was partially set in the framework of spherical varieties by Sakellaridis \cite{SakellaridisSph}.
Here a spherical variety $X$ for a reductive group $G$ over $F$  is a normal integral separated $G$-scheme $X$ of finite type over $F$ such that  $X_{\bar{F}}$ admits an open orbit under a Borel subgroup of $G_{\bar{F}}.$

The conjecture can be roughly formulated as follows.
Assume that $X$ is an affine spherical variety with smooth locus $X^{\mathrm{sm}}$.  Then there should be a Schwartz space $ \mathcal{S}( X(\A_F)) < C^{\infty}(X^{\mathrm{sm}}(\A_F))$ and a Fourier transform
$$
\mathcal{F}_X:\mathcal{S}( X(\A_F)) \lto \mathcal{S}(X (\A_F))
$$
satisfying a certain twisted-equivariance property under $G(\A_F)$
such that for $f \in \mathcal{S}(X(\A_F))$ satisfying certain local conditions
$$
\sum_{x\in X^{\mathrm{sm}}(F)}f(x)=\sum_{x \in X^{\mathrm{sm}}(F)}\mathcal{F}_X(f)(x).
$$
We refer to this (somewhat vaguely stated) conjecture as the \textbf{Poisson summation conjecture}.  The original motivation, explored in \cite{BK-lifting,NgoSums,Ngo:Hankel}, is that  it implies the meromorphic continuation and functional equation of Langlands $L$-functions in great generality.  By the converse theorem  \cite{Cogdell:PS:ConverseII} this would imply Langlands functoriality in great generality.  

\begin{rem}We highlight the possibly confusing convention that functions in $\mathcal{S}(X(\A_F))$ need not be defined on all of $X(\A_F),$ only on $X^{\mathrm{sm}}(\A_F).$  One expects that for each place $v$ elements of $\mathcal{S}(X(F_v))$ are functions in $C^\infty(X^{\mathrm{sm}}(F_v))$ that are rapidly decreasing away from the singular locus $(X-X^{\mathrm{sm}})(F_v)$ and have particular asymptotic behavior as one approaches $(X-X^{\mathrm{sm}})(F_v).$  This was conjectured in \cite[\S 5]{Ngo:Hankel} in a special case.
\end{rem}

The only case of the Poisson summation conjecture that is completely understood is the case where $X$ is a vector space.  For the affine closures $X_P$  of the Braverman-Kazhdan space  $X_P^\circ$ much of the conjecture is known \cite{Getz:Liu:BK,Getz:Hsu,Getz:Hsu:Leslie,Jiang:Luo:Zhang,Shahidi:FT}.  There are some additional examples in \cite{Getz:Summ,Getz:Liu:Triple,Getz:Hsu,Gu}.  However the cases that are known are still very limited.  

In order to prove Theorem \ref{thm:ES:intro} we prove the Poisson summation conjecture for $Y_{P,P'}$.  We do not know if $Y_{P,P'}$ is affine, but it is clearly quasi-affine.  We also do not know whether it is always spherical under the action of a suitable reductive subgroup of $H$, but this is true in many cases \cite{Gaetz,Hodges}.  

In Theorem \ref{thm:PS:intro} we state our Poisson summation formula in an imprecise form.   Let $K_M$ be the maximal compact subgroup of $M^{\mathrm{ab}}(\A_F)$.

\begin{thm} \label{thm:PS:intro}
Let $f \in \mathcal{S}(Y_{P,P'}(\A_F))$.  Assume
\begin{enumerate}
    \item $F$ is a function field,
    \item $F$ is a number field, Conjecture \ref{conj:poles:intro} is valid, and $f$ is $K_M$-finite, or 
    \item $F$ is a number field and Conjecture \ref{conj:poles} is valid.
\end{enumerate}
One has
\begin{align*} 
\sum_{y \in Y_{P}(F)}f(y )+*
&=
\sum_{y^* \in Y_{P^*}(F)}\mathcal{F}_{P|P^*}(f)(y^* )+**.
\end{align*} The sums over $y$ and $y^*$ are absolutely convergent.  
\end{thm}
\noindent 
 The precise version of this theorem is given in Theorem \ref{thm:PS} below. In Theorem \ref{thm:PS:intro} the contributions marked $*$ and $**$ are certain boundary terms coming from residues of auxilliary degenerate Eisenstein series.  Using Lemma \ref{lem:Macts} one can choose many $f$ so that these contributions vanish.

Theorem \ref{thm:PS:intro} is our main theorem in the context of Poisson summation formulae.  It is a vast generalization of the Poisson summation formula for Braverman-Kazhdan spaces of maximal parabolic subgroups in reductive groups \cite{BK:normalized} which in turn is a vast generalization of the Poisson summation formula for a vector space.  The work \cite{BK:normalized} in fact treats arbitrary parabolic subgroups, and with more effort one could probably use it to generalize our work to the case where $P$ is not maximal in $P'.$  

\begin{rem}
In the degenerate case $H=G,$ Theorem \ref{thm:PS:intro} reduces to the Poisson summation formula for the Braverman-Kazhdan space $X_{P}.$  In this special case under suitable assumptions on $f$ the formula was proved in \cite{BK:normalized}.  When $G=\mathrm{Sp}_{2n}$ and $P$ is the Siegel parabolic it was proved for general test functions finite under a maximal compact subgroup of $\mathrm{Sp}_{2n}(\A_F)$ in \cite{Getz:Liu:BK}.
\end{rem}

\subsection{Conjectures \ref{conj:poles:intro} and \ref{conj:poles}} Let $M_{\beta_0}$ be the simple normal subgroup 
of the Levi subgroup $M'$ of $P'$ defined in \eqref{M:decomp}
below.
For any topological abelian group $A$ we denote by $\widehat{A}$ the set of quasi-characters of $A,$ that is, continuous homomorphisms $A \to \CC^\times.$

For $Q \in \{P,P^*\}$ let $\mathcal{S}(X_{Q \cap M_{\beta_0}}(\A_F))$ be the Schwartz space of \eqref{SspacesX}.
For any
\begin{align*}
    (m,f_1,f_2,\chi,s) \in M_{\beta_0}(\A_F) \times  \mathcal{S}(X_{P \cap M_{\beta_0}}(\A_F)) \times \mathcal{S}(X_{P^* \cap M_{\beta_0}}(\A_F)) \times \widehat{A_{\GG_m} F^\times \backslash \A_F^\times} \times \CC
\end{align*}
let $\chi_s:=\chi|\cdot|^s$ and  form the degenerate Eisenstein series 
\begin{align} \begin{split}
E(m,f_{1\chi_s})&=\sum_{x\in (P \cap M_{\beta_0}) \backslash M_{\beta_0}(F)}f_{1\chi_s}(xm)\\ E^*(m,f_{2\chi_s}^*)&=\sum_{x \in (P^* \cap M_{\beta_0}) \backslash M_{\beta_0}(F)}f_{2\chi_s}^*(xm).\end{split}
\end{align}
They converge for $\mathrm{Re}(s)$ large enough (resp.~$\mathrm{Re}(s)$ small enough).   Here $f_{1\chi_s}$ and $f_{2\chi_s}^*$ are the Mellin transforms of  \eqref{Mellin}  in the special case $P'=M_{\beta_0}.$   

  Let $K \leq M_{\beta_0}(\A_F)$ be a maximal compact subgroup.  The following conjecture appeared in the statements of theorems \ref{thm:ES:intro} and \ref{thm:PS:intro} above:

\begin{conj} \label{conj:poles:intro}
For each character $\chi\in \widehat{A_{\GG_m} F^\times \backslash \A_F^\times}$ there is a finite set $\Upsilon(\chi) \subset \CC$ such that if $
E(m,f_{\chi_s})
$
has a pole for any $K$-finite
$f \in 
 \mathcal{S}(X_{P \cap M_{\beta_0}}(\A_F))
 $
 then  $s \in \Upsilon(\chi).$  
\end{conj}
\noindent We point out that continuous homomorphisms $A_{\GG_m} F^\times \backslash \A_F^\times \to \CC^\times$ are automatically unitary because $A_{\GG_m} F^\times \backslash \A_F^\times$ is compact.  

In fact we expect the following stronger conjecture to be true:
\begin{conj} \label{conj:poles}
There is an integer $n$ and a finite set $\Upsilon \subset \CC$ depending only on $M_{\beta_0}$ such that if $
E(m,f_{\chi_s})
$
has a pole for any $K$-finite $f \in \mathcal{S}(X_{P \cap M_{\beta_0}}(F))$ and $\chi \in  \widehat{A_{\GG_m}F^\times \backslash \A_F^\times }$
 then $\chi^n=1$ and $s \in \Upsilon.$  
\end{conj}

  In \S \ref{sec:poles} we prove Conjecture \ref{conj:poles:intro} (or more accurately extract it from the literature) when $M_{\beta_0}$ is $\mathrm{SL}_n$ and Conjecture \ref{conj:poles}  when $P \cap M_{\beta_0}$ is a Siegel parabolic in the symplectic group $M_{\beta_0}.$  
The case of other simple groups remains open, but provided that $M_{\beta}$ is not of type $E$ or $F$ Hsu \cite{Hsu} has reduced Conjecture \ref{conj:poles} and \ref{conj:poles:intro} to a local archimedean statement.

We point out that the natural analogues of conjectures \ref{conj:poles:intro} and \ref{conj:poles} for general sections of   $I_P(\chi_s)$ are false:
\begin{example}
    Let $P$ be the Borel subgroup of upper triangular matrices in $\GL_2,$ let $K \leq \GL_2(\A_\QQ)$ be a maximal compact subgroup, and let $f_{\chi_s}$ be a standard section, that is, $f_{\chi_s}(k)$ is independent of $s=(s_1, s_2)$ for all $k \in K.$  
    Write $\chi_s\begin{psmatrix}a & \\ & b \end{psmatrix}=\chi_{1s_1}(a)\chi_{2s_2}(b),$ and assume for simplicity that $\chi_1\chi_2^{-1} \neq 1.$   Let $S$ be a finite set of places of $\QQ$ including all infinite places such that $\chi_1$ and $\chi_2$ are unramified outside of $S.$  Then $E(m,f_{\chi_s})$ has poles at every zero of $L^S(s_1-s_2+1,\chi_{1}\chi_{2}^{-1})$ by \cite[Theorem 3.7.1]{Bump:AFR}.  The corresponding normalized Eisenstein series defined in loc.~cit.~has at most one pole.
\end{example}

\noindent This is why the sections 
in conjectures \ref{conj:poles:intro} and \ref{conj:poles} are assumed to be Mellin transforms of elements of the Schwartz space.

\subsection{On integral representations of $L$-functions}
\label{ssec:int:reps}

In this subsection we make some comments on how one might try to use Theorem \ref{thm:ES:intro} to study integral representations of $L$-functions.  Though this discussion is purely speculative, we include it as motivation for several questions regarding algebraic homogeneous spaces.  The questions are posed below.

Let a subgroup $H \leq G$ act on $G$ via right multiplication, and assume that $Y$ is stable under the action of $H.$ For example, we could take $Y=P'\gamma H$ for some $\gamma \in G(F).$   

\begin{rem}
The subgroup of $G$ stabilizing  $P^{\mathrm{der}} \backslash \overline{PwB}$ on the right contains $B$ and hence is a parabolic subgroup of $G.$  On the other hand the stabilizer of  $P^{\mathrm{der}} \backslash P \gamma H$ may have reductive neutral component.  Indeed, one could take $H$ to be any maximal connected reductive subgroup of $G$ such that $P\gamma H \neq G;$ then the neutral component of the stabilizer of $P\gamma H$ on the right is $H.$    For example, one could take $H$ to be the neutral component of an orthogonal similitude group in $G=\GL_n.$   Let $P'$ be a maximal parabolic subgroup of $\GL_n$ with Levi quotient $\GL_{n-1} \times \GL_1.$  Then $P' \backslash \GL_n \cong \mathbb{P}^{n-1},$ and the natural action of $H$ on $\mathbb{P}^{n-1}$ is not transitive, hence for any parabolic subgroup $P$ maximal in $P'$ one has $P\gamma H \neq \GL_n$ for any $\gamma \in \GL_n(F).$    
\end{rem}

Assume that $Y$ admits an open $H$-orbit.   Let $\varphi$ be a cusp form in a cuspidal automorphic representation $\pi$ of $H(\A_F)$.
If one wanted to use Theorem \ref{thm:ES:intro} to study integral representations of $L$-functions one could try to investigate expressions of the form
\begin{align} \label{this:exp}
\int_{H(F) \backslash H(\A_F)}\varphi(h) E_{Y_P}(R(h)f_{\chi_s}) dh
\end{align}
where $R(h)f_{\chi_s}(y)=f_{\chi_s}(yh).$
If convergent, \eqref{this:exp} admits a functional equation because 
 $E_{Y_P}(R(h)f_{\chi_s})$ admits a functional equation by Theorem \ref{thm:ES:intro}.   The problem is deciding whether or not \eqref{this:exp} or some variant of it yields any new  $L$-functions.  Ignoring all questions of convergence, the integral above unfolds into a sum
\begin{align}
    \sum_{y_0 \in M^{\mathrm{ab}}(F) \backslash Y_{P}(F)/H(F)} \int_{H_{y_0}(F) \backslash H(\A_F)}\varphi(h)f_{\chi_s}(y_0h)dh.
\end{align}
where $H_{y_0}$ is the stabilizer of $y_0 \in (M^{\mathrm{ab}} \backslash Y_P)(F)$ in $H.$
Due to the conjectures of Sakellaridis and Venkatesh \cite{SakellaridisSph,SV} one expects this expression to unfold into a finite sum of Eulerian integrals when $M^{\mathrm{ab}} \backslash Y_{P}$ is spherical as an $H$-scheme.  

This motivates the following questions: 
\begin{enumerate}
    \item When is $M^{\mathrm{ab}} \backslash Y_P$ spherical as an $H$-scheme? 
    \item If $M^{\mathrm{ab}} \backslash Y_P$ is spherical, what is the stabilizer in $H$ of a point in the open $H$-orbit?
    \item If $H$ and a spherical subgroup $H' \leq H$ are fixed, can one classify the possible spherical embeddings of $H' \backslash H$ obtained from $M^{\mathrm{ab}} \backslash Y_P$ as $P,$ $P',$ $G,$ and $Y$ vary?
\end{enumerate}
We can also pose analogous questions when $H$ is not necessarily reductive by replacing $H$ by a Levi subgroup as in \cite{Gaetz,Hodges}. Finally, it is of interest to consider the questions above in the Kac-Moody setting mentioned in Remark 1 after Theorem \ref{thm:ES:intro}.


\subsection{Outline}
We outline the paper and give some indication of the proofs.  We begin with observations on the underlying geometry we are considering  in \S \ref{sec:GS}.  We then define the Schwartz space of $Y_{P,P'}$ locally and adelically in \S \ref{sec:Ss}.  We also construct a Fourier transform and prove that the Fourier transform preserves the Schwartz space.  To accomplish this we reduce the question to a similar statement on an auxilliary Braverman-Kazhdan space and then use the methods developed in \cite{Getz:Liu:BK,Getz:Hsu:Leslie}.  

We then prove a Poisson summation formula for the auxilliary Braverman-Kazhdan space in \S \ref{sec:PS:BK}.  This formula is proved in \cite{BK:normalized} for a different definition of the Schwartz space with some restrictions on the test functions involved.  It was obtained for arbitrary test functions in \cite{Getz:Liu:BK} in a special case.  In Theorem \ref{thm:PS:BK} below we deduce it for all Braverman-Kazhdan spaces attached to maximal parabolic subgroups.  We point out that it is most natural in our setting, and perhaps even necessary, to work with sections that are not finite under a maximal compact subgroup of $G(F_\infty)$ when $F$ is a number field (see \S \ref{sec:Ss}).  
Hence, some care is required in applying well-known results on Eisenstein series in this work.    Indeed, our sections need not even be standard, so we cannot use Lapid's results in \cite{LapidRem}.

In \S \ref{sec:proofs} we prove theorems \ref{thm:ES:intro} and \ref{thm:PS:intro}.  The procedure is to first prove Theorem \ref{thm:PS:intro} by reducing it to the summation formula of \S \ref{sec:PS:BK}.  We then use Theorem \ref{thm:PS:intro} to deduce Theorem \ref{thm:ES:intro}.
In \S \ref{sec:poles} we verify Conjecture \ref{conj:poles:intro} in certain cases.

To address questions of the referee, we have added Appendix \ref{App}.  In it we indicate possible generalizations of our results.  We also give some commentary on why it is convenient to work with Schwartz spaces instead of working solely with Eisenstein series and intertwining operators.

\subsection*{Acknowledgments} 
We appreciate the encouragement and questions of D.~Bump.  Answering his questions led to a generalization of our original main result and simplifications in exposition.
The authors thank S.~Leslie for help with parabolic subgroups and M.~Brion for answering several questions about Schubert cells (and in particular for observing Lemma \ref{lem:CP:smooth}).  We also thank D.~Ginzburg and F.~Shahidi for encouragement, S.~Kudla for answering a question on Eisenstein series and M.~Hanzer for explaining how to derive Theorem \ref{thm:GLn} from her paper \cite{Hanzer:Muic} with G.~Muic.
The second author thanks H.~Hahn for her constant support and encouragement and for her help with the structure of the paper. 
Finally, both authors thank the anonymous referee for suggestions that improved the exposition.

\section{Preliminaries} \label{sec:GS}

\subsection{Braverman-Kazhdan spaces}
\label{sec:ss:BK:spaces}

We work over a field $F.$  
Let $G$ be a connected split semisimple group over $F.$
We only consider parabolic subgroups of $G$ containing  a fixed split torus $T$; for such a subgroup $P$ we write $N_P$ for its unipotent radical.
We fix a Levi subgroup $M$ of $G$ containing $T$ and write
\begin{align}
M^{\mathrm{ab}}:=M/M^{\mathrm{der}}.
\end{align}
This is again an algebraic group \cite[\S 5.c]{Milne:AGbook}.
For all parabolic subgroups $P$ of $G$ with Levi subgroup $M$ we have a Braverman-Kazhdan space
\begin{align}
X_{P}^\circ:=P^{\mathrm{der}} \backslash G.
\end{align}
Here the quotient exists in the category of schemes and is constructed in the usual manner \cite[Appendix B]{Milne:AGbook}.
We observe that 
\begin{align} \label{der:P}
    P^{\mathrm{der}}=M^{\mathrm{der}}N_P
\end{align}
where $N_P$ is the unipotent radical of $P.$
The scheme $X_P^{\circ}$ is strongly quasi-affine, i.e.
\begin{align}
X_P:=\overline{X_P^{\circ}}^{\mathrm{aff}}:=\mathrm{Spec}(\Gamma(X_P^{\circ},\OO_{X_P^{\circ}}))
\end{align}
is an affine scheme of finite type over $F$ and the natural map $X_P^{\circ} \to X_P$ is an open immersion  \cite[Theorem 1.1.2]{Braverman:Gaitsgory}. Strictly speaking, in loc.~cit.~Braverman and Gaitsgory work over an algebraically closed field, but their results hold in our setting by fpqc descent along $\mathrm{Spec}(\bar{F}) \to \mathrm{Spec}(F).$ A convenient reference for fpqc descent is \cite[Theorem 4.3.7]{Poonen:Rational}.

\begin{lem} \label{lem:surj}
The torus $M^{\mathrm{ab}}$ is split.  The maps
\begin{align*}
M(F) \lto M^{\mathrm{ab}}(F) \quad \textrm{and} \quad 
G(F) \lto (P^{\mathrm{der}} \backslash G)(F)
\end{align*}
are surjective.  
\end{lem}
\begin{proof}
In \cite[Lemma 3.2 and Corollary  3.3]{Getz:Hsu:Leslie} this is proved in the special case where $P$ is a maximal parabolic.  The same proof implies the more general statement here.
\end{proof}

   We have a right action of $M^{\mathrm{ab}} \times G$ on $X_P^{\circ}$ given on points in an $F$-algebra $R$ by
\begin{align} \label{geo:action} \begin{split}
X^{\circ}_{P}(R) \times M^{\mathrm{ab}}(R) \times G(R) &\lto X^\circ_{P}(R)\\
(x,m,g) &\longmapsto m^{-1}xg. \end{split}
\end{align}
This action extends to an action of $M^{\mathrm{ab}} \times G$ on $X_P.$

We now discuss the Pl\"ucker embedding of $P^{\mathrm{der}} \backslash G$
as explained in a special case in \cite{Getz:Hsu:Leslie}.  We assume for simplicity that $G$ is simply connected.  Let  $T\leq B \leq P$ be a Borel subgroup, let $\Delta_G$ be the set of simple roots of $T$ in $G$ defined by $B$ and $\Delta_M$ the set of simple roots of $T$ in $M$ defined by $B \cap M.$
Let
\begin{align} \label{betak}
\Delta_P:=\Delta_G-\Delta_M
\end{align}
 be the set of simple roots in $\Delta_G$ attached to $P.$   For each $\beta \in \Delta_G$ we let 
 \begin{align} \label{Pbeta}
 B \leq P_{\beta} \leq G
 \end{align}
 be the unique maximal parabolic subgroup containing $B$ that does not contain the root group attached to $-\beta.$
 
As usual let $X^*(T)$ be the character group of $T.$  For each $\beta \in \Delta_P$ let $\omega_{\beta} \in X^*(T)_{\QQ}$ be the fundamental weight dual to the coroot $\beta^\vee,$ in other words,
\begin{align} \label{omegabeta}
\omega_{\beta}(\alpha^{\vee})=\begin{cases}1 &\textrm{ if }\alpha=\beta\\
0 &\textrm{ otherwise.}\end{cases}
\end{align}
Here $\alpha \in \Delta_G.$  Since $G$ is simply connected, the fundamental weight $\omega_{\beta}$ lies in $X^*(T).$

There is an irreducible representation
\begin{align} \label{Vbeta}
V_{\beta} \times G \lto V_{\beta}
\end{align}
of highest weight $-\omega_{\beta}$.  Here $G$ acts on the right.   
Let 
\begin{align} \label{Vcirc}
    V:=\prod_{\beta \in \Delta_P} V_\beta, \quad V^{\circ}=\prod_{\beta \in \Delta_P} (V_\beta-\{0\}).
\end{align}
Choose a highest weight vector $v_\beta \in V_\beta(F)$ for each $\beta.$    

\begin{lem} \label{lem:Plucker:stuff}
There is a closed immersion
\begin{align} \label{Pl0}
   \mathrm{Pl}_P:X_P^\circ \lto V^\circ
\end{align}
given on points by 
$$
 \mathrm{Pl}_P(g):=(v_\beta g).
$$
It extends to a $G$-equivariant map $\mathrm{Pl}_P:X_P \to V.$
The character $\omega_{\beta}$ extends to a character of $M$ for all $\beta$ and the map
\begin{align} \label{omegaP}
    \omega_P:=\prod_{\beta \in \Delta_P}\omega_{\beta}:M^{\mathrm{ab}} \lto \GG_m^{\Delta_P}
\end{align}
is an isomorphism.  If we let $M^{\mathrm{ab}}$ act via $\omega_P$ on $V$ then $\mathrm{Pl}_P$ is $M^{\mathrm{ab}}$-equivariant.  In particular for $(m,g) \in M^{\mathrm{ab}}(R) \times G(R)$ one has 
\begin{align}  \label{action}
\mathrm{Pl}_P(m^{-1}g)=(\omega_\beta(m)v_{\beta} g).
\end{align}
\end{lem}

\noindent Here 
$$
\GG_m^{\Delta_P}:=\prod_{\beta \in \Delta_P}\GG_m.
$$
We will use similar notation below without comment.  In the introduction, we implicitly chose a bijection between $\{0,1,\dots,k\}$ and $\Delta_P$ sending $0$ to $\beta_0,$ where $\beta_0$ is  defined as in \S \ref{ssec:YP} below.

To ease notation, for any subset $\Delta \subseteq \Delta_G$ let
\begin{align} \label{VXi}
    V_{\Delta}=\prod_{\beta \in \Delta}V_{\beta} \quad \textrm{ and }\quad V_{\Delta}^\circ:=\prod_{\beta \in \Delta}(V_{\beta}-\{0\}).
\end{align}

\begin{proof}   The stabilizer of the line spanned by $v_\beta$ in $V_\beta$ is the parabolic $P_\beta$ of \eqref{Pbeta}
 \cite[\S II.8.5]{Jantzen}.  Since $P\leq P_{\beta}$ for $\beta \in \Delta_P$ we deduce that $\mathrm{Pl}_P$ is well-defined, that $\omega_{\beta}$ extends to a character on $M,$ and that \eqref{action} holds.  Since $V$ is affine the map $\mathrm{Pl}_P$ tautologically extends to a $M^{\mathrm{ab}} \times G$-equivariant map $\mathrm{Pl}:X_P \to V.$
 We are left with checking that $\mathrm{Pl}_P$ is a closed immersion and that \eqref{omegaP} is an isomorphism.  

We first check that \eqref{omegaP} is an isomorphism.
 The group $M^{\mathrm{ab}}$ is isogenous to the center $Z_M$ of $M$ which is contained in the subgroup of $T$ on which all the $\beta \in \Delta_M=\Delta_G-\Delta_P$ vanish.  Thus $Z_M$ has at most dimension $|\Delta_P|$.
On the other hand the homomorphism $\GG^{\Delta_P}_m \to M^{\mathrm{ab}}$ given on points by $(x_\beta) \mapsto \prod_{\beta} \beta^\vee(x_\beta)$ is a section of $\omega_P$. 
Thus $\dim M^{\mathrm{ab}}=|\Delta_P|$ and $\omega_P$ is injective.  Using the usual contravariant equivalence of categories relating tori and their character groups we deduce that $\omega_P$ is also surjective.

To check that $\mathrm{Pl}_P$ is a closed immersion we first point out that we have a commutative diagram
\begin{equation} \label{Pl2}
\begin{tikzcd}
X_P^{\circ} \arrow[d,twoheadrightarrow] \arrow[r,"\mathrm{Pl}_P"] &V^\circ \arrow[d,twoheadrightarrow]\\
P \backslash G  \arrow[r,"\overline{\mathrm{Pl}}_P",hookrightarrow] & \prod_{\beta \in \Delta_P}\mathbb{P}V_{\beta}
\end{tikzcd}
\end{equation}
where $\overline{\mathrm{Pl}}_{P}$ is induced by $\mathrm{Pl}_P$.  We claim that $\overline{\mathrm{Pl}}_P$ is a closed immersion.  Since $P \backslash G$ is proper and $\prod_{\beta \in \Delta_P}\mathbb{P}V_{\beta}$ is separated the map $\overline{\mathrm{Pl}}_P$ has closed image.  It is an immersion provided that $P$ is the stabilizer of the image of $(v_{\beta})$ in $\prod_{\beta \in \Delta_P}\mathbb{P}V_{\beta}$ by
\cite[Proposition 7.17]{Milne:AGbook}. 
The stabilizer of the image of $(v_\beta)$ is 
 \begin{align} \label{other:para}
 \bigcap_{\beta \in \Delta_P}P_\beta.
 \end{align}
 But \eqref{other:para} is a parabolic subgroup containing $B$.  By considering the root groups contained in \eqref{other:para} we deduce that it  is $P$. This completes the proof of the claim.

Extend $\omega_P$ to a character of $P$ in the canonical manner.
Since the stabilizer of the image of $(v_\beta)$ in $\prod_{\beta \in \Delta_P}\mathbb{P}V_{\beta}$ is $P$ we deduce that the stabilizer of $(v_\beta)$ is contained in $P$ and hence equal to  $\ker (\omega_P:P \to \GG_m^{\Delta_P}).$  Since \eqref{omegaP} is an isomorphism
$\ker (\omega_P:P \to \GG_m^{\Delta_P})=P^{\mathrm{der}}.$ Thus the map $\mathrm{Pl}_P$ is an immersion of $X_P^{\circ}$ onto the orbit of $(v_\beta)$ \cite[Proposition 7.17]{Milne:AGbook}.

The left vertical arrow in \eqref{Pl2} is the geometric quotient by the action of $M^{\mathrm{ab}}$, and the right vertical arrow is the geometric quotient by $\GG_m^{\Delta_P}$.
In view of the equivariance property \eqref{action} and the fact that \eqref{omegaP} is an isomorphism,  we deduce that  $\mathrm{Pl}_P(X_P^\circ)$ is the inverse image of $\overline{\mathrm{Pl}}_P(P \backslash G)$ in $V^\circ.$  We have already proved that $\overline{\mathrm{Pl}}_P(P \backslash G)$ is closed, so we deduce that $\mathrm{Pl}_P(X_P^\circ)$ is closed.
\end{proof}

\begin{example}
Let $G=\SL_3$ and $P=B,$ the Borel of upper triangular matrices.    For a suitable ordering $\beta_1,\beta_2$ of the set of simple roots of $T$ in $B$ the representations $V_{\beta_1}$ and $V_{\beta_2}$ are just the standard representation $\GG_a^3$ and $\wedge^2 \GG_a^3.$  If we choose $(0,0,1)$ and $(0,1,0) \wedge (0,0,1)$ as our highest weight vectors then
$$
\mathrm{Pl}_B\begin{psmatrix} a \\ b\\ c \end{psmatrix} =\left(c,b \wedge c\right).
$$ 
Under the Pl\"ucker embedding the image of $X_B$ is the cone $C$ whose points in an $F$-algebra $R$ are given by 
$$
C(R)=\{(v_1,v_2) \in  R^3 \times \wedge^2 R^3: v_1\wedge v_2=0\}.
$$

\end{example}

\subsection{The $Y_{P}$ spaces} \label{ssec:YP}

We assume that we are in the situation of the introduction.  Thus $Y \subseteq G$ is a reduced subscheme whose stabilizer under the left action of $G$ contains a parabolic subgroup $P'> P.$  We assume \eqref{P'assump}, that is, $P$ is maximal in $P'.$ Moreover we let $P^*\leq P'$ be the unique parabolic subgroup with Levi subgroup $M$ that is not equal to $P.$
 There is a unique simple root $\beta_0$ in $\Delta_G-\Delta_M$ such that the root group attached to $-\beta_0$ is a subgroup of $P'$ but not $P.$
Let $M'$ be the unique Levi subgroup of $P'$ containing $M.$  The set of simple roots of $T$ in $M'$  with respect to the Borel $B$ is 
\begin{align} \label{beta0}
\Delta_{M'}=\{\beta_0\} \cup \Delta_M.
\end{align}
Since $G$ is simply connected, it follows from \cite[Proposition 4.3]{Borel:Tits:Compl} that the derived group $M'^{\mathrm{der}}$ is simply connected.
Hence it is a direct product of its simple factors.  There is a unique decomposition
\begin{align} \label{M:decomp}
M'^{\mathrm{der}}=M_{\beta_0} \times M^{\beta_0}
\end{align}
where $M_{\beta_0}$ is the unique simple factor containing the root group of $\beta_0.$  It then is the unique simple factor containing $\beta^\vee_0(\GG_m).$   Let $N_P$ denote the unipotent radical of $P.$  
\begin{lem} \label{lem:der}
The group $T_{\beta_0}:=T \cap M_{\beta_0}$ is a maximal torus of $M_{\beta_0}$ and $B \cap M_{\beta_0}$ is a Borel subgroup of $M_{\beta_0}$. The group $P \cap M_{\beta_0}$ is the unique maximal parabolic subgroup of $M_{\beta_0}$ containing $B \cap M_{\beta_0}$ that does not contain the root group attached to $-\beta_0,$ and $M \cap M_{\beta_0}$ is a Levi subgroup of $P \cap M_{\beta_0}.$  
We have 
\begin{align} \label{4Ids} \begin{split}
N_{P \cap M_{\beta_0}}&=N_{P} \cap M_{\beta_0},\\
M^{\mathrm{der}} \cap M_{\beta_0}&=(M \cap M_{\beta_0})^{\mathrm{der}},\\ P^{\mathrm{der}} \cap M_{\beta_0}&=
(P \cap M_{\beta_0})^{\mathrm{der}}, \textrm{ and }\\
P^{\mathrm{der}}&=(P^{\mathrm{der}} \cap M_{\beta_0}) M^{\beta_0}N_{P'}. \end{split}
\end{align}
\end{lem}

\begin{proof} The first two sentences follow from  \cite[Propositions 1.19 and 1.20, \S XXVI]{SGA3}. 

Since $N_{P \cap M_{\beta_0}}$ is the maximal connected normal unipotent subgroup of $P \cap M_{\beta_0}$ it is clearly contained in $N_P,$ the maximal connected normal unipotent subgroup of $P.$  Thus $N_{P \cap M_{\beta_0}} \leq N_P \cap M_{\beta_0}.$  On the other the image of $N_P \cap M_{\beta_0} <(M \cap M_{\beta_0})N_{P \cap M_{\beta_0}}=P \cap M_{\beta_0}$ under the canonical quotient map $P \cap M_{\beta} \to M \cap M_{\beta_0}$ is the identity.  It follows that $N_P \cap M_{\beta_0} \leq N_{P \cap M_{\beta_0}}.$

Let $Z_H$ denote the center of an algebraic group $H.$  Since $Z_{M'} \leq M$ and $M^{\beta_0}\leq M$ we have 
\begin{align} \label{before:der}
M=Z_{M'}(M \cap M_{\beta_0} \times M^{\beta_0})=Z_{M'}(Z_{M \cap M_{\beta_0}}(M \cap M_{\beta_0})^{\mathrm{der}} \times M^{\beta_0}).
\end{align}
Thus $M^{\beta_0}\leq M^{\mathrm{der}}$ and hence
\begin{align} \label{dir:prod}
    M^{\mathrm{der}}=(M^{\mathrm{der}} \cap M_{\beta_0}) \times M^{\beta_0}.
\end{align}
Taking derived groups of both sides of \eqref{before:der} we deduce that 
and $(M \cap M_{\beta_0})^{\mathrm{der}}=M^{\mathrm{der}} \cap M_{\beta_0}.$

We have 
\begin{align} \label{unip}
    N_P=N_{P \cap M'} N_{P'}=N_{P \cap M_{\beta_0}}N_{P'}
\end{align}
so 
\begin{align} \label{semidir}
P^{\mathrm{der}}=M^{\mathrm{der}}N_P= (M \cap M_{\beta_0})^{\mathrm{der}}N_{P \cap M_{\beta_0}} \ltimes M^{\beta_0}N_{P'}.
\end{align}
Thus
\begin{align} \label{Pder}
    P^{\mathrm{der}}\cap M_{\beta_0}=(M \cap M_{\beta_0})^{\mathrm{der}}N_{P \cap M_{\beta_0}} \leq (P \cap M_{\beta_0})^{\mathrm{der}}
\end{align}
Since $(P \cap M_{\beta_0})^{\mathrm{der}} \leq P^{\mathrm{der}} \cap M_{\beta_0}^{\mathrm{der}}=P^{\mathrm{der}} \cap M_{\beta_0}$ we deduce $(P \cap M_{\beta_0})^{\mathrm{der}} = P^{\mathrm{der}} \cap M_{\beta_0}.$  Combining \eqref{semidir} and \eqref{Pder} we also deduce $P^{\mathrm{der}}=(P^{\mathrm{der}} \cap M_{\beta_0})M^{\beta_0}N_{P'}.$
\end{proof}

Let
$$
\mathrm{pr}:G \lto X_P^\circ
$$
be the natural map.  In the following it is convenient to denote by $|X|$ the underlying topological space of a scheme $X.$  As usual, a subset of a topological space is locally closed if it is open in its closure. 

It is helpful to recall that if $X$ is a scheme then $Z \mapsto |Z|$ defines a bijection between the set of reduced subschemes of $X$ and the set of locally closed subsets of $|X|$ \cite[Proposition 3.52]{Gortz_Wedhorn}.
In particular, by definition, the schematic closure of a reduced subscheme $Z \subseteq X$ is the unique reduced subscheme $\overline{Z} \subseteq X$ satisfying $|\overline{Z}|=\overline{|Z|}.$

\begin{lem} \label{lem:locally:closed}
The set  $\mathrm{pr}(|Y|) \subset |X_P^\circ|$ is locally closed.
\end{lem}

\begin{proof}
 Let $\overline{Y}$ be the schematic closure of $Y$ in $G.$  Then $\overline{Y}$ is stable under left multiplication by $P^{\mathrm{der}}.$  The map $\mathrm{pr}$ is a geometric quotient \cite{Raynaud:Passage}.  In particular the underlying map of topological spaces is a quotient map.   It follows that  $\mathrm{pr}(|\overline{Y}|)$ is closed in $|X_{P}^\circ|$.
 The restriction of $\mathrm{pr}$ to $|\bar{Y}|$ is again a topological quotient map, so  $\mathrm{pr}(|Y|)$ is open in $\mathrm{pr}(|\overline{Y}|).$  We claim that the closure of  $\mathrm{pr}(|Y|)$ is $\mathrm{pr}(|\overline{Y}|).$  Indeed, if $\mathrm{pr}(|Y|)$ is contained in any closed set $Z,$ then $|Y| \subseteq \mathrm{pr}^{-1}(Z),$ which is closed, and hence $\overline{|Y|} \subseteq \mathrm{pr}^{-1}(Z).$  This  implies in turn that $\mathrm{pr}(\overline{|Y|}) \subseteq Z.$
\end{proof}

As in the introduction, $Y_{P}$ is the subscheme of $X_{P}^\circ$ with underlying topological space $\mathrm{pr}(|Y|),$ given the reduced induced subscheme structure. 

\begin{lem} \label{lem:CP:smooth}  If $Y=P'\gamma H$ for some subgroup $H \leq G$ and $\gamma \in G(F)$ then the subscheme $Y_{P} \subset X_P^{\circ}$ is smooth.
\end{lem}

\begin{proof}
The space $Y$ is smooth, and $\mathrm{pr}$ is a locally trivial fibration.  Since $Y$ is left $P^{\mathrm{der}}$-invariant $\mathrm{pr}$ restricts to a locally trivial fibration $\mathrm{pr}:Y \to Y_{P}.$  
\end{proof}

Lemma \ref{lem:surj} implies the following lemma:
\begin{lem} \label{lem:surj:2}
The map 
$$
Y(F) \lto Y_{P}(F)
$$
is surjective. \qed
\end{lem}

Using notation from \eqref{betak} and \eqref{VXi}
let
\begin{align} \label{partial:closures} \begin{split}
    X_{P,P'}:&=\mathrm{Pl}^{-1}_P(V_{\beta_0} \times V_{\Delta_{P'}}^{\circ}) \subseteq X_P\\
    X_{P^*,P'}:&=\mathrm{Pl}^{-1}_{P^*}(V_{\beta_0^{-1}} \times V_{\Delta_{P'}}^{\circ}) \subseteq X_{P^*}.  \end{split}
\end{align}
Thus $X_{P,P'}$ and $X_{P^*,P'}$ are subschemes of $X_{P}$ and $X_{P^*},$ respectively.  Let
\begin{align} \label{partial:closure}\begin{split}
    Y_{P,P'}:&=\overline{Y_{P}} \subseteq X_{P,P'}\\
    Y_{P^*,P'}:&=\overline{Y_{P^*}} \subseteq X_{P^*,P'}\end{split}
\end{align}
be the closures of $Y_{P}$ in $X_{P,P'}$ and $Y_{P^*}$ in $X_{P^*,P'},$ respectively.

Using \eqref{unip} we see that the natural map 
\begin{align} \label{unip:incl}
N_{P \cap M_{\beta_0}} \lto N_P/N_{P'}
\end{align}
is an isomorphism.
For all $y \in Y(F)$ we have a morphism 
\begin{align} \label{iotay0}
\iota_y:X_{P \cap M_{\beta_0}}^\circ \lto Y_{P}
\end{align}
characterized by the requirement that the diagram
\begin{equation} \label{iotay}
\begin{tikzcd}
M_{\beta_0} \arrow[r] \arrow[d,twoheadrightarrow] &  Y \arrow[d,twoheadrightarrow]\\
X^{\circ}_{P \cap M_{\beta_0}} \arrow[r,"\iota_y"] &Y_{P}
\end{tikzcd}
\end{equation}
commutes, where the top arrow is given on points by $m \mapsto my$ and the vertical arrows are the canonical surjections.

\begin{lem} \label{lem:cosets}
Let $\Xi \subset Y(F)$ be a set of representatives for $P'^{\mathrm{der}}(F) \backslash Y(F).$
One has
\begin{align*}
    Y_{P}(F)=\coprod_{y \in \Xi} \iota_{y}(X^\circ_{P \cap M_{\beta_0}}(F))\quad \textrm{and} \quad
     Y_{P^*}(F)=\coprod_{y \in \Xi} \iota_{y}(X^\circ_{P^* \cap M_{\beta_0}}(F)).
\end{align*}
\end{lem}

\begin{proof}
With notation as in \eqref{M:decomp} we have $M_{\beta_0}M^{\beta_0}N_{P'}=P'^{\mathrm{der}}.$
The fibers of the canonical projection 
$M^{\beta_0}(F)N_{P'}(F) \backslash Y(F) \to P'^{\mathrm{der}}(F) \backslash Y(F)$ are $M_{\beta_0}(F)$-torsors, so
$$
M_{\beta_0}(F)\Xi
$$
 is a set of representatives for $M^{\beta_0}(F)N_{P'}(F) \backslash Y(F).$  Moreover no two elements of $\Xi$ are in the same $M_{\beta_0}(F)$-orbit.  
 By Lemma \ref{lem:surj:2}
 $P^{\mathrm{der}}(F) \backslash Y(F)=Y_{P}(F).$
By Lemma \ref{lem:der} the 
 fibers of the natural projection
 $$
 M^{\beta_0}(F)N_{P'}(F) \backslash Y(F) \lto P^{\mathrm{der}}(F) \backslash Y(F)=Y_{P}(F)
 $$
are $(P^{\mathrm{der}} \cap M_{\beta_0})(F)$-torsors.  Again by Lemma \ref{lem:der} we have $P^{\mathrm{der}} \cap M_{\beta_0}=(P \cap M_{\beta_0})^{\mathrm{der}}$ and we deduce the first equality of the lemma.  

To obtain the second equality we replace $P$ by $P^*$ and argue by symmetry.
\end{proof}

\subsection{Examples} \label{sec:example} Let $G=\mathrm{SL}_n,$ $n>2.$   We generalize the example of \eqref{Bruhat:decomp} to higher rank in two different manners. Let $\sigma_{j} \in \SL_n(F),$ $1 \leq j \leq n-1,$ be the matrix that is the identity matrix with the rows $e_j$ and $e_{j+1}$ replaced by $-e_{j+1}$ and $e_j.$  Let
$$
\gamma_1=\sigma_{1}\sigma_{2}\cdots \sigma_{n-1}.
$$
This is a Coxeter element.  Let $R$ be an $F$-algebra.
Let $B_n \leq \SL_n$ be the Borel subgroup of upper triangular matrices and let
\begin{align*}
    P'(R):&=\{\begin{psmatrix} g & x \\ & b \end{psmatrix} \in \SL_n(R): (g,b) \in \GL_2(R) \times  B_{n-2}(R)\}\\
    H(R):&=\{\begin{psmatrix} b & x \\ & g \end{psmatrix} \in \SL_n(R): (g,b) \in \GL_2(R) \times  B_{n-2}(R)\}
\end{align*}
Then $B_n \gamma_1B_n \subseteq P'\gamma_1 H \subseteq \overline{B_n\gamma_1 B_n}$, and
\begin{align}
    \overline{B_n\gamma_1B_n}(R)=\left\{g \in \SL_n(R): g_{ij}=0 \textrm{ if }i>j+1 \right\}.
\end{align} 
If we take $P=B_n$, then $P$ is maximal in $P'$ and 
\begin{align}
    B_n^*(R):=\left\{\begin{psmatrix} a &  & x\\ b & c & z\\ & & d\end{psmatrix} \in \mathrm{SL}_n(R): a,c \in R^\times, d \in B_{n-2}(R) \right\}.
\end{align}

As another example, take
\begin{align}
   \gamma_2= \begin{psmatrix} & & &J \\ (-1)^\epsilon  & & & \\ &1  & &  \end{psmatrix}
\end{align}
where $J$ is the matrix with $1$'s on the antidiagonal and zeros elsewhere and $\epsilon \in \{0,1\}$ is chosen so the matrix has determinant $1$. 
Write
$$
P_{a,b}(R):=\left\{\begin{psmatrix} g & x \\ & h \end{psmatrix} \in \SL_n(R): (g,h) \in \GL_a(R) \times \GL_b(R) \right\}.
$$
Then
$$
B_n\gamma_2 B_n \subseteq P_{n-1,1}\gamma_2 P_{1,n-1} \subseteq \overline{B_n \gamma_2 B_n}
$$

Choose positive integers $a,b$ such that $a+b=n-1.$  Then we may take 
\begin{align*}
    P(R):=\left\{\begin{psmatrix} g_1 & x & y\\ & g_2 & z \\ & & c\end{psmatrix} \in \mathrm{SL}_n(R): (g_1,g_2,c) \in \GL_{a}(R) \times \GL_b(R) \times R^\times\right\}.
\end{align*}
In this case
\begin{align*}
    P^*(R):=\left\{\begin{psmatrix} g_1 &  & y\\ x & g_2 & z \\ & & c\end{psmatrix} \in \mathrm{SL}_n(R): (g_1,g_2,c) \in \GL_{a}(R) \times \GL_b(R) \times R^\times\right\}.
\end{align*}

\subsection{Function spaces}
Let $X$ be a quasi-projective scheme over a local field $F$. We denote by $C^0(X(F))$ the space of complex valued continuous functions on $X(F).$ Assume that $F$ is nonarchimedean.  In this case we denote by 
\begin{equation}\label{eqn: compact funtions notation}
    \mathcal{C}(X(F))
\end{equation}
the space of locally constant compactly supported functions on $X(F)$, also denoted by $C_c^\infty(X(F))$.  If $X$ is smooth, then we set 
$$
\mathcal{S}(X(F))=\mathcal{C}(X(F))=C_c^\infty(X(F)).
$$
For certain nonsmooth schemes $X$ we will define $\mathcal{S}(X(F)).$   

Now assume that $F$ is archimedean.  In this case we define $\mathcal{C}(X(F))$ as follows.  To contrast schemes with real algebraic varieties, we abuse notation and identify real algebraic varieties with their real points.  The set $X(F)=\mathrm{Res}_{F/\RR}X(\RR)$ is a real algebraic variety, and hence an affine real algebraic variety  
\cite[Proposition 3.2.10, Theorem 3.4.4]{BCR}.  In particular there is a closed embedding of real algebraic varieties
\begin{align}
    X(F) \hooklto \RR^n
\end{align}
for some $n$.  We define $\mathcal{C}(X(F))$ to be the space of restrictions of the usual Schwartz space $\mathcal{S}(\RR^n)$ to $X(F)$.  This space is independent of the choice of embedding and is naturally a Fr\'echet space \cite[\S 3]{Elazar:Shaviv}.    Thus $\mathcal{C}(X(F))\leq C^0(X(F)).$  We observe that if $X$ is not smooth then $C^\infty(X(F))$ is not defined when $F$ is archimedean.

If $X$ is smooth then we set 
$$
\mathcal{S}(X(F))=\mathcal{C}(X(F)) \leq C^\infty(X(F)).
$$
We have $\mathcal{S}(X(F)) \geq C_c^\infty(X(F))$ but the inclusion is strict in the archimedean case when $X(F)$ is noncompact.  We will also define $\mathcal{S}(X(F))$ for certain nonsmooth $X$.   

\subsection{Measures} \label{ssec:measures}

Let $F$ be a global field.  We fix a nontrivial character $\psi:F \backslash \A_F \to \CC^\times$ and a Haar measure $dx=\otimes_v' dx_v$ on $\A_F.$  We assume that $dx_v$  is self-dual with respect to $\psi_v.$  We let $d^\times x$ be the Haar measure on $\A_{F}^\times$ given by
$$
\otimes_v' \frac{\zeta_v(1)}{|x|_v}dx_v.
$$

 We fix, once and for all, a Chevalley basis of the Lie algebra of $G$ with respect to $T.$  For every root of $T$ in $G$ this provides us with a root vector $X_{\alpha}$ in each root space, and hence isomorphisms
$$
\GG_a \lto N_{\alpha}
$$
where $N_{\alpha}$ is the root group.  This in turn provides us with a Haar measure on $N_{\alpha}(\A_F)$ for all $\alpha.$  As a scheme (but not a group scheme) the unipotent radical of any parabolic subgroup $P$ with Levi subgroup $M$ is a product of various $N_{\alpha}.$  Thus we obtain a Haar measure on $N_{P}(\A_F).$  We use this normalization so that factorization of intertwining operators holds (otherwise it only holds up to a constant depending on the choice of Haar measures).

We fix a Haar measure on $M^{\mathrm{der}}(\A_F).$
 We give $M(\A_F)$ the unique Haar measure such that, upon endowing $M^{\mathrm{ab}}(\A_F)$ with the quotient measure, one has that
\begin{eqnarray}\label{is}
\omega_P:M^{\mathrm{ab}}(\A_F) \lto  (\A_F^\times)^{\Delta_P}
\end{eqnarray}
is measure preserving. This is independent of the parabolic $P$ with $M$ as its Levi, as different choices will just replace various $\omega_{\beta}$ with their inverses.  

Each of the measures we fixed above factors over the places of $F$ into a product of local measures, normalized so that the local analogue of \eqref{is} is measure preserving.  We use these local measures when working locally below.

\subsection{Quasi-characters} \label{ssec:q:char}
Our convention is that characters are unitary.  We refer to general continuous homomorphisms to $\CC^\times$ as quasi-characters. Let
$$
\chi:=\prod_{\beta \in \Delta_P}\chi_\beta: (A_{\GG_m} F^\times \backslash \A_F^\times)^{\Delta_P} \lto \CC^\times
$$
be a quasi-character.    If $s=(s_\beta) \in \CC^{\Delta_P}$ we let
$$
\chi_{s}=\prod_{\beta \in \Delta_P}\chi_{\beta}|\cdot|^{s_\beta}.
$$

We define a subgroup
\begin{align} \label{AGm}
    A_{\GG_m} \leq F_\infty^\times
\end{align}
in the following manner.  
The global field $F$ is a finite extension of $F_0,$ where $F_0=\QQ$ or $F_0=\mathbb{F}_p(t)$ for some prime $p.$  Let
$Z \leq \mathrm{Res}_{F/F_0}\GG_m$ be the maximal split subtorus.  When $F$ is a number field we take $A_{\GG_m}$ to be the neutral component of $Z(\RR).$  Thus $A_{\GG_m}$ is just $\RR_{>0}$ embedded diagonally in $F^\times_\infty.$   When $F$ is a function field we choose an isomorphism 
$$
Z \tilde{\lto} \GG_m
$$
 and let $A_{\GG_m}$ be the inverse image of $t^\ZZ.$

\subsection{The unramified setting} \label{ssec:unram}
For a nonarchimedean place $v$ of the global field $F$ let $\FF_v$ be the residue field of $F_v$ and $q_v:=|\FF_v|.$   Let $S$ be a 
 finite set of  places of $F$ including the infinite places.  Upon enlarging $S$ if necessary, we can choose a reductive group scheme $\mathcal{G}$ over $\OO_F^S$ (with connected fibers) and affine spaces
$\mathcal{V}_{\beta} \cong \GG_{a\OO_F^S}^{d_\beta}$ for some integer $d_\beta>0$ (isomorphism over $\OO_F^S$) equipped with homomorphisms
$$
\mathcal{G} \lto \mathrm{Aut}(\mathcal{V}_\beta)
$$
over $\OO_F^S$ whose generic fibers are the representations $G \to \mathrm{Aut}(V_\beta).$   By abuse of notation, write $G$ for $\mathcal{G}.$ For any of the subgroups $M,$ $P,$ $P',$ etc.~of $G_F$ we continue to use the same letter for their schematic closures in $G.$  Upon enlarging $S$ if necessary we may assume that these groups are all smooth.  We assume moreover that the groups whose generic fibers were reductive (resp.~parabolic, etc.)  extend to reductive group schemes (resp.~parabolic group schemes, etc.) over $\OO_F^S.$
We also assume that $\omega_{P}$ induces an isomorphism on $\OO_{F_v}$-points for all $v \not \in S,$ the highest weight vectors $v_\beta \in V(F)$ lie in $V_\beta(\OO_F^S),$ and their image in  $V_\beta(\FF_v)$ is again a highest weight vector for $G_{\FF_v}$ of weight $-\omega_\beta.$  Finally, we continue to denote by $Y$ the schematic closure of $Y$ in $G$.  It is a scheme over $\OO_F^S$, and the action of $P'_F$ on $Y_F$ extends to an action of $P'$ on $Y.$

Under the assumptions above (which are no loss of generality for $S$ large enough) and we are considering the local setting over $F_v$ for $v \not \in S$ we say that we are  \textbf{in the unramified setting}.

\section{The Schwartz space of $Y_{P,P'}(F)$} \label{sec:Ss}

For all but the last subsection of this section $F$ is a local field.  

\subsection{Induced representations} \label{ssec:Ind}

Consider the induced representations:
\begin{align} \label{induced}
    I_{P}(\chi_s):=\mathrm{Ind}_{P}^{G}(\chi_s \circ \omega_P) \quad \textrm{and} \quad 
    I^*_{P^*}(\chi_s):=\mathrm{Ind}_{P^*}^{G}(\chi_s \circ \omega_P).
\end{align}
By convention, sections of these induced representations are smooth functions on $G(F).$
Let $\Phi^{\chi_s}$ be a section of $I_{P}(\chi_s).$
Assume $F$ is archimedean. We say $\Phi^{\chi_s}$  is holomorphic (resp.~meromorphic) if $s \mapsto \Phi^{\chi_s}(g)$ is holomorphic as a function of $s \in \CC^{\Delta_P}$ for all $g \in G(F)$ and characters $\chi:(F^\times)^{\Delta_P} \to \CC^\times.$  If $F$ is  nonarchimedean with residue field of cardinality $q$ we say that  $\Phi^{\chi_s}$ is holomorphic if $\Phi^{\chi_s}(g) \in \CC[\{q^{s_\beta},q^{-s_\beta}:\beta \in \Delta_P\}]$ for all $g \in G(F)$ and characters $\chi:(F^\times)^{\Delta_P} \to \CC^\times.$  Similarly we say it is meromorphic if for all characters $\chi:(F^\times)^{\Delta_P} \to \CC^\times$ there is a $p_{\chi} \in \CC[\{q^{s_\beta},q^{-s_\beta}:\beta \in \Delta_P\}]$ such that 
$s \mapsto p_{\chi}(s)\Phi^{\chi_s}(g)$ is holomorphic for all $g \in G(F).$  Fix a maximal compact subgroup $K \leq G(F)$ such that the Iwasawa decomposition holds: $G(F)=P(F)K.$  We then say that $\Phi^{\chi_s}$ is \textbf{standard} if the restriction of the function $(s,g) \mapsto \Phi^{\chi_s}(g)$ to $\CC^{\Delta_P} \times K$ is independent of $s.$ 
We take the analogous conventions regarding sections of $I_{P^*}^*(\chi_s).$  Standard sections are allowed to be nontrivial on $K;$ in particular standard sections may be ramified.

Let $\mathcal{E}$ denote the ring of entire functions on $\CC^{\Delta_P}$ when $F$ is archimedean and $\CC[\{q^{s_\beta},q^{-s_\beta}:\beta \in \Delta_P\}]$ when $F$ is nonarchimedean.
For a fixed character $\chi:(F^\times)^{\Delta_P} \to \CC^\times$ there is an obvious action of $\mathcal{E}$ on the $\CC$-vector space of holomorphic sections of $I_P(\chi_s)$ preserving the subspace of $K$-finite sections.  As an $\mathcal{E}$-module, the subspace of holomorphic $K$-finite sections is generated by the subspace of standard $K$-finite sections. 
This allows us to apply results in the literature stated for $K$-finite standard sections to sections that are $K$-finite and merely holomorphic.  We will use this observation without further comment below.

\subsection{The Schwartz space} \label{sec:twist}

We define $P,P^*$ and $P'$ as in the introduction.  Thus $P \cap P^*$ is the Levi subgroup $M$ of the parabolic subgroups $P$ and $P^*.$  Moreover $P$ and $P^*$ are maximal (proper) parabolic subgroups of $P'.$  
To define the Fourier transform $\mathcal{F}_{P|P^*}$ we first apply an intertwining operator to certain functions on $Y_{P}(F)$ to arrive at functions on $Y_{P^*}(F).$  We then twist by certain operators that we recall in this section.  These operators were first introduced in \cite{BK:normalized}.

Suppose we are given $\lambda \in X_*(M^{\mathrm{ab}}).$ Let $s' \in \CC.$ We define
\begin{align}
    \lambda_{!}(s'):\mathcal{C}(Y_{P^*}(F)) \lto C^0(Y_{P^*}(F))
\end{align}
by 
\begin{align*}
    \lambda_{!}(s')(f)(x)=\int_{F^\times}\delta_{P^*}^{1/2}(\lambda(a))f(\lambda(a^{-1})x)\psi(a)|a|^{s'}da.
\end{align*}
Here $\psi:F \to \CC^\times$ is the local factor of the global additive character fixed in \S \ref{ssec:measures} and $da$ is the Haar measure on $F.$ 
In the special case where $P$ is maximal and $P'=Y=G$ this reduces to \cite[(4.2)]{Getz:Hsu:Leslie}, where it was denoted by $\lambda_{!}(\mu_s).$  The same operator is denoted by $\lambda_!(\eta_{\psi}^{s'})$ in \cite{BK:normalized,Shahidi:FT}.  To extend the domain of definition of $\lambda_!(s'),$ let $\Phi \in \mathcal{S}(F)$ satisfy $\Phi(0)=1$ and $\widehat{\Phi} \in C_c^\infty(F).$  Here
$\widehat{\Phi}(x):=\int_{F}\Phi(y)\psi\left(xy \right)dy.$  Define the regularized integral
\begin{align}
  \lambda_!(s')^{\mathrm{reg}}(f)(x):=\lim_{|b| \to \infty}\int_{F^\times} \Phi\left(\frac{a}{b} \right)\delta_{P^*}^{1/2}(\lambda(a))f(\lambda(a^{-1})x)\psi(a)|a|^{s'}da.
\end{align}
This limit is said to be well-defined if the integral
$$
\int_{F^\times} { \Phi \left(\frac{a}{b} \right)}\delta_{P^*}^{1/2}(\lambda(a))f(\lambda(a^{-1})x)\psi(a)|a|^{s'}da
$$
is convergent provided $|b|$ is large enough and the limit in the definition of $\lambda_!(s')^{\mathrm{reg}}(x)$  exists and is independent of $\Phi.$  Thus if the integral defining $\lambda_!(s')(f)$ is absolutely convergent then $\lambda_!(s')^{\mathrm{reg}}(f)=\lambda_!(s')(f).$  In particular this is the case if $f \in \mathcal{C}(Y_{P^*}(F)).$  To avoid more proliferation of notation we will drop the superscript $\mathrm{reg}.$

Let $L(s',\chi)$,  $\varepsilon(s',\chi,\psi)$, and
$$
\gamma(s',\chi,\psi)=\frac{\varepsilon(s',\chi,\psi)L(1-s',\chi^{-1})}{L(s',\chi)}
$$
be the usual Tate local zeta function, $\varepsilon$-factor, and $\gamma$-factor
attached to a quasi-character $\chi:F^\times \to \CC^\times$ and a complex number $s' \in \CC$.

We use a hat to denote the dual group of an $F$-group (more precisely the neutral component of the Langlands dual).  
Let $\mathcal{N}$ be a $1$-dimensional representation of $Z_{\widehat{M}}=\widehat{M}^{\mathrm{ab}}$  with $s' \in \tfrac{1}{2}\ZZ$ attached to it.  The action of $Z_{\widehat{M}}$ is given by a character $\lambda:Z_{\widehat{M}} \to \GG_m,$ which we may identify with a cocharacter $\lambda:\GG_m \to M^{\mathrm{ab}}.$  Let 
\begin{align*} 
a_{\mathcal{N}}(\chi_s):=L\left(-s',\chi_s
\circ \lambda \right)\quad \textrm{and} \quad
\mu_{\mathcal{N}}(\chi_s):=\gamma(-s',\chi_s \circ \lambda,\psi).
\end{align*}
We let 
\begin{align}
\widetilde{\mathcal{N}}
\end{align}
be $\mathcal{N}^\vee$ (on which $\widehat{M}^{\mathrm{ab}}$ acts via $-\lambda$) with the real number $-1-s'$ attached to it.

More generally, assume $\mathcal{N}=\oplus_{i=1}^\ell\mathcal{N}_i$ is a finite-dimensional representation of $M^{\mathrm{ab}}$ with each $\mathcal{N}_i$ $1$-dimensional.  We let the  $Z_{\widehat{M}}$-action on $\mathcal{N}_i$ be given by $\lambda_i$ and assume each $\mathcal{N}_i$ is equipped with a complex number $s_i$.  Define
\begin{align} \label{agamma}
a_{\mathcal{N}}(\chi_s):&=\prod_{i=1}^\ell a_{\mathcal{N}_i}(\chi_s),\\
\label{muNchi}
 \mu_{\mathcal{N}}(\chi_s):&=\prod_{i=1}^\ell\mu_{\mathcal{N}_i}(\chi_s).
\end{align}
We also define $\widetilde{\mathcal{N}}:=\oplus_{i=1}^\ell \widetilde{\mathcal{N}}_i.$ 
We will in fact only consider $\mathcal{N}=\oplus_{i=1}^\ell \mathcal{N}_i$ where the parameters  attached to $\mathcal{N}_i$ are all of the form $(s_i,\lambda_i)$ where $\lambda_i$ is an integer multiple of $\beta_0^\vee$.  Here $\beta_0$ is the simple root of \eqref{beta0}.  We will therefore abuse notation and allow ourselves to again write $\lambda_i$ for the integer $n_i$ such that $\lambda_i=n_i\beta_0^\vee$. 
Assume that $\lambda_i>0$ for all $i$.  Then we can order the $\mathcal{N}_i$ so that 
$$
\frac{s_{i+1}}{\lambda_{i+1}} \geq \frac{s_i}{\lambda_i}
$$
for all $i$.
Choosing such an ordering we define
\begin{align}
\label{muN}
    \mu_{\mathcal{N}}:&=\lambda_{1!}(s_1)\circ \dots \circ \lambda_{\ell !}(s_\ell).
\end{align}
Theorem \ref{thm:FT} below explains why it is reasonable to use the symbol $\mu_{\mathcal{N}}$ in both \eqref{muNchi} and \eqref{muN}.

Let $\mathfrak{n}_P$ be the Lie algebra of the unipotent radical $N_P$ of $P$ and let $\widehat{\mathfrak{n}}_P$ be its Langlands dual.    Let
\begin{align} \label{npq}
\widehat{\mathfrak{n}}_{P|P^*} :=\widehat{\mathfrak{n}}_{P}/\widehat{\mathfrak{n}}_P \cap \widehat{\mathfrak{n}}_{P^*}.
\end{align}  
Let $\{e,h,f\} \subset \widehat{\mathfrak{m}}$ be a principal $\mathfrak{sl}_2$-triple (here $\widehat{\mathfrak{m}}$ denotes the Langlands dual of $\mathfrak{m},$ the Lie algebra of $\mathfrak{m}$).  Consider the subspace $\widehat{\mathfrak{n}}_{P|P^*}^e \leq \widehat{\mathfrak{n}}_{P|P^*}$ annihilated by $e.$  
It admits a decomposition
\begin{align} \label{BK:decomp}
\widehat{\mathfrak{n}}_{P|P^*}^e=\oplus_{i} \mathcal{N}_i
\end{align}
where the $\mathcal{N}_i$ are $1$-dimensional eigenspaces for the action of $Z_{\widehat{M}}.$

We observe that $Z_{\widehat{M}}$ acts via a power of $\beta_0^\vee$ on $\mathcal{N}_i.$ 
We assign each $\mathcal{N}_i$ the real number $s_i$ that is half the eigenvalue of $h,$ and define
\begin{align} \label{ap}
a_{P|P}(\chi_s)=a_{\widetilde{\widehat{\mathfrak{n}}_{P|P^*}^e}}((\chi_s)^{-1}), \quad 
a_{P|P^*}(\chi_s)=
a_{\widehat{\mathfrak{n}}_{P|P^*}^e}(\chi_s).
\end{align}
These factors enjoy the symmetry property \begin{align}
    a_{P|P}(\chi_s)=a_{P^*|P^*}(\chi_s), \quad a_{P|P^*}(\chi_s)=a_{P^*|P}(\chi_s)
\end{align}
by the discussion on passing to the opposite parabolic contained in \cite[\S 4.2]{Getz:Hsu:Leslie}.

\begin{lem} \label{lem:holo} 
There is an $\varepsilon>0$ depending only on $P$ and $P'$ such that for any character $\chi:(F^\times)^{\Delta_P} \to \CC^\times$ the function $a_{P|P}(\chi_s)$ is holomorphic and nonzero for 
$$
\mathrm{Re}(s_{\beta_0}) \geq -\varepsilon.
$$
\end{lem}

\begin{proof}
Consider the set of parameters $\{(s_i,\lambda_{i}\beta_0^\vee)\}$ attached to $\widehat{\mathfrak{n}}_{P|P^*}.$
It suffices to check that for each $i$ one has $s_i \geq 0$ and  $\lambda_{i} >0 .$  Since $s_i$ is $\tfrac{1}{2}$ the highest weight of a certain $\mathfrak{sl}_2$ representation with respect to the usual Borel subalgebra $\langle h,e\rangle$ it is nonnegative.  It is also clear that $\lambda_i>0.$
\end{proof}

As above, let $M'$ be the unique Levi subgroup of $P'$ containing $M$ and define $M_{\beta_0}$ as in \eqref{M:decomp}.  We have $N_{P^* \cap M_{\beta_0}}=N_{P^*} \cap M_{\beta_0}$ by Lemma \ref{lem:der} (in which we can replace $P$ by $P^*$ by symmetry) and this in turn is equal to $N_{P^*} \cap M'$.  On the other hand $N_{P^*} \cap N_P=N_{P'}$ and $(N_{P^*} \cap M')N_{P'}=N_{P^*}.$  It follows that the closed immersion $ N_{P^*  \cap M_{\beta_0}} \to N_{P^*}$ induces a bijection
$$
N_{P^* \cap M_{\beta_0}}(F)  \tilde{\lto} N_{P^*}(F) \cap N_{P}(F) \backslash N_{P^*}(F).
$$
Thus the usual unnormalized intertwining operator restricts to define an operator
\begin{align}
    \mathcal{R}_{P|P^*}:C_c^\infty(Y_{P}(F)) \lto C^0(Y_{P^*}(F))
\end{align}
given by 
\begin{align} \label{Radon}
\mathcal{R}_{P|P^*}(f)(g):=\int_{N_{P^* \cap M_{\beta_0}}(F) } f(ug)du=\int_{N_{P^* \cap M'}(F) } f(ug)du
\end{align}
Here we have used the fact that $N_{P^* \cap M_{\beta_0}}=N_{P^* \cap M'}.$
We will use the same notation whenever $\mathcal{R}_{P|P^*}(f)$ is defined (e.g.~for more general smooth functions or via analytic continuation).
In \cite[\S 4]{Shahidi:FT} Shahidi proves that this agrees with the operator defined by Braverman and Kazhdan.

A section $\Phi^{\chi_s}$ of $I_P(\chi_s)$ is
\textbf{good} if it is meromorphic and if the section
$$
\frac{\mathcal{R}_{P|Q}\Phi^{\chi_s}(g)}{a_{P|Q}(\chi_s)}
$$
is holomorphic for all $g \in G(F)$ and $Q \in \{P,P^*\}$ (recall our conventions regarding meromorphic sections from \S \ref{ssec:Ind}).

We defined adelic Mellin transforms in \eqref{Mellin} above.  We use the obvious local analogues of this notation.   
We write  $f_{\chi_s}$ for the Mellin transform of any function $f:Y_{P}(F) \to \CC$ or $f:X_P^\circ(F) \to \CC$ such that the integral defining the Mellin transform is absolutely convergent or obtained by analytic continuation from some region of absolute convergence.

\textbf{Assume $F$ is nonarchimedean}. Let $K \leq M^{\mathrm{ab}}(F) \times G(F)$ be a compact open subgroup.
Let $\mathcal{C}_{\beta_0}(X_P(F))$ be the space of $K$-finite $f \in C^\infty(X_{P}^\circ(F))$ such that for $\mathrm{Re}(s_{\beta_0})$ sufficiently large the integral defining the Mellin transform
$f_{\chi_s}$ converges absolutely and defines a good section.
We define the \textbf{Schwartz space} of $Y_{P,P'}(F)$ to be the space of restrictions to $Y_{P}(F)$ of functions in $\mathcal{C}_{\beta_0}(X_P(F))$:
\begin{align} \label{image}
\mathcal{S}(Y_{P,P'}(F))=\mathrm{Im}(\mathcal{C}_{\beta_0}(X_P(F)) \lto C^0(Y_{P}(F))).
\end{align}

\begin{rem}
We expect that the space $C_c^\infty(Y_{P}(F))$ of compactly supported and locally constant functions  on $Y_P(F)$ is contained in $\mathcal{S}(Y_{P,P'}(F)).$  However this is not obvious.  If $M_{\beta_0}$ is not of type $E$ or $F$ then it follows from the proof of \cite[Theorem 1.3]{Hsu} that $C_c^\infty(Y_P(F)) <\mathcal{S}(Y_{P,P'}(F)).$  We expect an analogous statement holds in the archimedean case, but it may be more difficult to prove.
One can also obtain alternate characterizations of $\mathcal{S}(Y_{P,P'}(F))$ from the results of loc.~cit.
\end{rem}

Write 
$K_{\GG_m}$ for the maximal compact subgroup of $F^\times$.  Say that two quasi-characters $\chi_1,\chi_2:F^\times \to \CC^\times$ are equivalent if $\chi_1=\chi_{2}|\cdot|^s$ for some $s \in \CC$.  Then the set of equivalence classes of quasi-characters of $F^\times$ is in natural bijection with $\widehat{K}_{\GG_m}$.  Thus we sometimes write $\widehat{K}_{\GG_m}$ for a set of representatives of the quasi-characters of $F^\times$ modulo equivalence.  In the archimedean setting we fix the following sets of representatives:
$$
\widehat{K}_{\GG_m}:=\begin{cases}\{1,\mathrm{sgn}\} &\textrm{ if }F=\RR \\
\{z \mapsto \left(\tfrac{z}{(\bar{z}z)^{1/2}}\right)^m:m \in \ZZ \} & \textrm{ if }F=\CC.\end{cases}
$$

For extended real numbers $A,B \in \{ -\infty\} \cup \RR \cup \{\infty\}$ with $A<B$ let
\begin{align} \label{VAB}
V_{A,B}:=\{s \in \CC^{\Delta_P}:A<\mathrm{Re}(s_\beta) <B \textrm{ for }\beta \in \Delta_P\}.
\end{align}
For functions $\phi:\CC^{\Delta_P} \to \CC$ and polynomials $p$ on $\CC^{\Delta_P}$ let
\begin{align} \label{normVAB}
|\phi|_{A,B,p}:=\sup_{s \in V_{A,B}}|\phi(s)p(s)|
\end{align}
(which may be infinite).  \textbf{Assume $F$ is archimedean}.   The action \eqref{geo:action} induces an action of $U(\mathfrak{m}^{\mathrm{ab}} \oplus \mathfrak{g})$ on $C^\infty(X_{P}^\circ(F)).$  Here $U(\mathfrak{m}^{\mathrm{ab}} \oplus \mathfrak{g})$ is the universal enveloping algebra of the complexification of the Lie algebra $\mathfrak{m}^{\mathrm{ab}} \oplus \mathfrak{g}$ of $M^{\mathrm{ab}} \times G,$ viewed as a real Lie algebra.

Let $\mathcal{C}_{\beta_0}(X_{P}(F))$ be the set of all $f \in C^\infty(X_P^\circ(F))$ such that for all $D \in U(\mathfrak{m}^{\mathrm{ab}} \oplus \mathfrak{g})$ and each character $\chi:(F^\times)^{\Delta_P} \to \CC^\times$ the integral 
\eqref{Mellin} defining $(D.f)_{\chi_s}$ converges for $\mathrm{Re}(s_{\beta_0})$ large enough, is a good section, and satisfies the following condition:  For all real numbers $A<B,$ $Q \in \{P,P^*\},$ any polynomials $p_{P|Q}$ such that $p_{P|Q}(s)a_{P|Q}(\eta_s)$ has no poles  for all $(s,\eta) \in V_{A,B} \times \widehat{K}_{\GG_m}^{\Delta_P}$ and all compact subsets $\Omega \subset X_{P}^\circ(F)$ one has 
$$
|f|_{A,B,p_{P|Q},\Omega,D}:=\sum_{\eta \in \widehat{K}_{\GG_m}^{\Delta_P}}\mathrm{sup}_{g \in \Omega}|\mathcal{R}_{P|Q}(D.f)_{\eta_s}(g)|_{A,B,p_{P|Q}}<\infty.
$$
\noindent We observe that it is indeed possible to choose $p_{P|Q}(s)$ as above (independently of $\eta$).  This follows directly from the definition of $a_{P|Q}(\eta_s).$  Since we have defined $\mathcal{C}_{\beta_0}(X_{P}(F))$ for archimedean $F$ we can and do define $\mathcal{S}(Y_{P,P'}(F))$ as in \eqref{image}.

In the archimedean case the seminorms $|\cdot|_{A,B,p_{P|Q},\Omega,D}$ give $\mathcal{S}(Y_{P,P'}(F))$ the structure of a Fr\'echet space as we now explain.  The seminorms 
 $|\cdot|_{A,B,p_{P|Q},\Omega,D}$ give $\mathcal{C}_{\beta_0}(X_{P}(F))$ the structure of a Fr\'echet space via a standard argument.
See \cite[Lemma 3.2]{Getz:Hsu} for the proof in a special case, the proof in general is essentially the same.  As in the proof of \cite[Lemma 3.5]{Getz:Hsu}, using Mellin inversion, one checks that with respect to this Fr\'echet structure evaluating a function in $\mathcal{C}_{\beta_0}(X_P(F))$ at a point of $X_P^{\circ}(F)$ is a continuous linear functional on $\mathcal{C}_{\beta_0}(X_P(F)).$  Thus the $\CC$-linear subspace $I \leq \mathcal{C}_{\beta_0}(X_{P}(F))$ consisting of functions that vanish on $Y_{P}(F)$ is closed subspace.  
Restriction of functions to $Y_{P}(F)$ induces a $\CC$-linear isomorphism
\begin{align}
    \mathcal{C}_{\beta_0}(X_P(F))/I \lto \mathcal{S}(Y_{P,P'}(F))
\end{align}
and thus we obtain a Fr\'echet topology on $\mathcal{S}(Y_{P,P'}(F))$ by transport of structure.  This definition is inspired by \cite[\S 2.2]{Elazar:Shaviv}. 
The set of seminorms giving $\mathcal{S}(Y_{P,P'}(F))$ its topology are
$$
|f|_{A,B,p_{P|Q},\Omega,D}:=\mathrm{inf}\{|\widetilde{f}|_{A,B,p_{P|Q},\Omega,D}: \widetilde{f} \in \mathcal{C}_{\beta_0}(X_P(F)) \textrm{ and } \widetilde{f}|_{Y_{P}(F)}=f\}
$$
for $f \in \mathcal{S}(Y_{P,P'}(F))$. 
Let $H\leq G$ act on $G$ via right multiplication, and assume that $H$ stabilizes $Y.$ Then
 \eqref{geo:action} restricts to an action of $M^{\mathrm{ab}}(F) \times H(F)$ on $Y_{P}(F).$  This induces an action of $M^{\mathrm{ab}}(F) \times H(F)$ on $\mathcal{S}(Y_{P,P'}(F))$  that is continuous in the archimedean case.  

\begin{remarks} \item We have defined the Schwartz space in terms of restrictions of functions on $X_P^\circ(F)$ in order to take advantage of the transitive action of $G(F)$ on $X_P^{\circ}(F).$  

\item The space $Y_{P,P'}(F)$ plays no role in the definition of $\mathcal{S}(Y_{P,P'}(F)).$  However, at least in the nonarchimdean case, using the results of \cite{Hsu} it should be possible to give a characterization of $\mathcal{S}(Y_{P,P'}(F))$ as the space of smooth functions on $Y_{P}(F)$ that have particular germs as one approaches the boundary $Y_{P,P'}(F)-Y_{P}(F).$  
\end{remarks}

As in the special cases treated in \cite{Getz:Liu:BK,Getz:Hsu,Getz:Hsu:Leslie}, we have defined the Schwartz space to be the space of smooth functions with sufficiently well-behaved Mellin transforms.  This is reasonable because we can obtain information on  $f$ from its Mellin transforms via Mellin inversion as in the proof of Lemma \ref{lem:esti} below.

We now discuss the problem of bounding functions in $\mathcal{S}(Y_{P,P'}(F)).$  Since we have an embedding $\mathrm{Pl}_P:X_P \to V$ of $X_P$ into an affine space we will phrase our bounds in terms of this affine space.  Let $K \leq G(F)$ be a maximal compact subgroup such that the Iwasawa decomposition 
$$
G(F)=P(F)K
$$
holds. If we are in the unramified setting in the sense of \S \ref{ssec:unram} we take $K=G(\OO_F).$  The group $K$ does not act on $Y_{P}(F)$ in general.

The group $G(F)$ acts on each $V_\beta(F).$ 
For each $\beta \in \Delta_P$ choose a $K$-invariant norm $|\cdot|_\beta$ on the $F$-vector space $V_\beta(F).$  As a warning, for $F=\CC$ we have $|cv|_\beta=c\bar{c}|v|_\beta$ for $c \in F$ and $v \in V_{\beta}(F)$ (this is the ``number theorist's norm''). 
If we write $x=P^{\mathrm{der}}(F)mk$ with $(m,k) \in M^{\mathrm{ab}}(F) \times K$ then by Lemma \ref{lem:Plucker:stuff} one has
\begin{align} \label{normbeta}
|\mathrm{Pl}_{P_\beta}(x)|_\beta=|\mathrm{Pl}_{P_\beta}(m)|_\beta=|\omega_{\beta}(m)|^{-1}.
\end{align}
The inverse here appears because $G$ is acting on $V_{\beta}$ on the right. Let $r_{\beta} \in \RR_{>0}$ be the real numbers satisfying 
\begin{align} \label{ri}
\prod_{\beta \in \Delta_P}|\omega_\beta(m)|^{r_{\beta}}=\delta_P^{1/2}(m).
\end{align}
\quash{\begin{rem}
One can deduce a formula for the $r_\beta$ by reducing to the case of a maximal parabolic and applying \cite[Proposition 6.2]{Getz:Hsu:Leslie}.
\end{rem}}

Recall the definition of $V_{\beta_0}$ from \eqref{Vbeta} and $V_{\Delta_{P'}}^{\circ}$ from \eqref{VXi}.

\begin{lem} \label{lem:esti}
Let $f \in \mathcal{S}(Y_{P,P'}(F)).$ For any sufficiently small $\alpha \in \RR_{>0}$ there is a nonnegative Schwartz function $\Phi_{f,\alpha} \in \mathcal{S}(V_{\beta_0}(F) \times  V_{\Delta_{P'}}^\circ(F))$ such that
\begin{align*}
|f(x)| \leq  \Phi_{f,\alpha}(\mathrm{Pl}_P(x))\prod_{\beta \in \Delta_{P}}|\mathrm{Pl}_{P_{\beta}}(x)|_\beta^{\alpha-r_\beta}
\end{align*}
where $r_{\beta}$ is in \eqref{ri}.
If $F$ is archimedean, then $\Phi_{f,\alpha}$ can be chosen continuously as a function of $f.$
\end{lem}

\begin{proof}
Let
$$
I_F:=\begin{cases}i\left[ -\frac{\log q}{\pi},\frac{\log q}{\pi}\right] & \textrm{ if }F \textrm{ is non-archimedean}\\
i\RR & \textrm{ if }F \textrm{ is archimedean.}\end{cases}
$$ 
Let $c_\psi \in \RR_{>0}$ be chosen so that $c_\psi dx$ is the standard Haar measure on $F$, where $dx$ is normalized to be self-dual with respect to $\psi.$  Here the standard Haar measure is the Lebesgue measure if $F=\RR$, twice the Lebesgue measure if $F=\CC$, and the unique Haar measure giving $\OO_F$ measure $|\mathfrak{d}|^{1/2}$ if $F$ is nonarchimedean, where $\mathfrak{d}$ is a generator for the absolute different.  Then let
\begin{align}
    c_F:=\begin{cases} c_\psi \log q &\textrm{ if } F \textrm{ is nonarchimedean}\\
    \frac{c_\psi}{2} &\textrm{ if }F=\RR\\
    \frac{c_\psi}{2\pi} &\textrm{ if }F=\CC \end{cases}
\end{align}

For suitable continuous functions $f:Y_{P}(F) \to \CC$ the Mellin inversion formula states that 
\begin{align} \label{Mellin:inversion}
f(x)=\int_{\sigma+I_F^{\Delta_P}}\sum_{\eta \in \widehat{K}_{\GG_m}^{\Delta_P}}f_{\eta_s}(x) \frac{c_F^{|\Delta_P|}ds}{(2\pi i)^{|\Delta_P|}}
\end{align} 
for suitable $\sigma \in \RR^{\Delta_P}$.  By \cite[\S 2]{Blomer_Brumley_Ramanujan_Annals} this formula holds whenever the integral defining $f_{\eta_s}$ is absolutely convergent for all $\eta \in \widehat{K}_{\GG_m}^{\Delta_P}$ and $s$ such that $\mathrm{Re}(s)=\sigma$ and 
\begin{align}
    \int_{\sigma+I_F^{\Delta_P}}\sum_{\eta \in \widehat{K}_{\GG_m}^{\Delta_P}}|f_{\eta_s}(x)| ds<\infty.
\end{align}

By Lemma \ref{lem:holo} we have that $a_{P|P}(\chi)$ is holomorphic for $\mathrm{Re}(s_{\beta_0}) \geq -\varepsilon$ for some $\varepsilon>0$ independent of the character $\chi.$  It follows that, for $f \in \mathcal{S}(Y_{P,P'}(F))$,  \eqref{Mellin:inversion} holds for $\sigma=(-\varepsilon/2,\dots,-\varepsilon/2).$  Writing $x=P^{\mathrm{der}}(F)mk$ with $(m,k) \in M^{\mathrm{ab}}(F) \times K$ the above becomes
\begin{align}
    f(mk)=\int_{\sigma+I_F^{\Delta_P}}\sum_{\eta \in \widehat{K}_{\GG_m}^{\Delta_P}}\delta_P^{1/2}(m)\eta_s(\omega_P(m))f_{\eta_s}(k)  \frac{c_F^{|\Delta_P|}ds}{(2\pi i)^{|\Delta_P|}}.
\end{align}
To obtain the bound and the continuity statement from this expansion one uses the same argument as that proving \cite[Lemma 3.5]{Getz:Hsu}.  
\end{proof}

\subsection{The Fourier transform} 

To ease notation let
\begin{align} \label{muPs} \begin{split}
    \mu_P:&=\mu_{\widehat{\mathfrak{n}}^e_{P|P^*}}\\
    \mu_{P}(\chi_s):&=\mu_{\widehat{\mathfrak{n}}^e_{P|P^*}}(\chi_s) \end{split}
\end{align}
where the operator (resp.~function) on the right is defined as in \eqref{muN} (resp.~\eqref{muNchi}).

By the same argument proving \cite[Theorem 5.12]{Getz:Hsu:Leslie} we obtain the following theorem:
\begin{thm} \label{thm:FT}
The map
$$
\mathcal{F}_{P|P^*}:=\mu_P \circ \mathcal{R}_{P|P^*}:\mathcal{S}(Y_{P,P'}(F)) \lto \mathcal{S}(Y_{P^*,P'}(F))
$$
is a well-defined isomorphism, bicontinuous in the archimedean case. Moreover the diagram
\begin{equation*}
\begin{tikzcd}[column sep=huge]
\mathcal{S}(Y_{P,P'}(F)) \arrow[r,"\mathcal{F}_{P|P^*}"] \arrow[d,"(\cdot)_{\chi_s}"] & \mathcal{S}(Y_{P^*,P'}(F)) \arrow[d,"(\cdot)^*_{\chi_s}"]\\I_P(\chi_s)|_{Y_P(F)} \arrow[r,"\mu_{P}(\chi_s)\mathcal{R}_{P|P^*}"] &I_{P^*}^*(\chi_s)|_{Y_{P^*}(F)}
\end{tikzcd}
\end{equation*} 
commutes for all $\chi:F^\times \to \CC^\times$ and $s \in \CC^\times$.  
\qed
\end{thm}

\noindent Recall that $Y_P$ is the image of $Y \subset G$ in $X_P^\circ$ (see above \eqref{Cgamma}).  We point out that $\mathcal{R}_{P|P^*}$ is well-defined on the the space $I_P(\chi_s)|_{Y_P(F)}$ of restrictions of sections in $I_P(\chi_s)$ to $Y_P(F)$ since $Y$ is invariant under left translation by $P'.$

The commutativity of the diagram must be understood in the sense that one has an identity of meromorphic functions
$$
\mathcal{F}_{P|P^*}(f)_{\chi_s}^*=\mu_{P}(\chi_s)\mathcal{R}_{P|P^*}(f_{\chi_s}).
$$
Let $H \leq G$ be a subgroup, and consider its action on $G$ via right multiplication.
Assume that $Y$ is stable under the action of $H.$  
For 
$$
(m,h,x_1,x_2) \in M^{\mathrm{ab}}(F) \times H(F) \times Y_{P}(F) \times Y_{P^*}(F)
$$ 
and $(f_1,f_2) \in \mathcal{S}(Y_{P,P'}(F)) \times \mathcal{S}(Y_{P^*,P'}(F))$ let
\begin{align} \label{Sch:act}
L(m)R(h)f_1(x_1)=f_1(m^{-1}x_1h), \quad L(m)R(h)f_2(x_2)= f_2(m^{-1}x_2h)
\end{align}
be the left and right translation operators.  It is easy to see that $L(m)R(h)$ preserves $\mathcal{S}(Y_{P,P'}(F))$ and $\mathcal{S}(Y_{P^*,P'}(F)).$

\begin{lem} \label{lem:equiv}
One has
$$
\mathcal{F}_{P|P^*} \circ L(m)R(h)=\delta_{ P^* \cap M'}(m)L(m)R(h) \circ \mathcal{F}_{P|P^*}.
$$
\end{lem}
\begin{proof}
The operator $\mu_P$ is $M^{\mathrm{ab}}(F) \times H(F)$-equivariant.  Thus the lemma follows from the definition \eqref{Radon} of $\mathcal{R}_{P|P^*}.$
\end{proof}

 We have Schwartz spaces 
\begin{align} \label{SspacesX}
\mathcal{S}(X_{P \cap M_{\beta_0}}(F)) \quad \textrm{and} \quad \mathcal{S}(X_{P^* \cap M_{\beta_0}}(F))
\end{align}
defined as in \cite[\S 5.2]{Getz:Hsu:Leslie} and a Fourier transform
\begin{align}
\mathcal{F}_{P \cap M_{\beta_0}|P^* \cap M_{\beta_0}}:\mathcal{S}(X_{P \cap M_{\beta_0}}(F)) \lto \mathcal{S}(X_{P^* \cap M_{\beta_0}}(F))
\end{align}
defined as in \cite[Theorem 5.12]{Getz:Hsu:Leslie}.  
It is an isomorphism, bicontinuous in the archimedean case.  
These facts are a special case of our construction of the Schwartz space $\mathcal{S}(Y_{P,P'}(F))$ and the Fourier transform $\mathcal{F}_{P|P^*}.$  One simply replaces 
\begin{align} \label{repl1}
(G,P,P',Y)
\end{align}
by 
\begin{align} \label{repl2}
(M_{\beta_0},P \cap M_{\beta_0},M_{\beta_0},M_{\beta_0}).
\end{align}

Recall that for each $y \in Y(F)$ we have a map $\iota_y:X_{P \cap M_{\beta_0}}^\circ \to Y_{P}$ defined as in \eqref{iotay0}.

\begin{prop} \label{prop:iotay}
For each $y \in Y(F)$
one has a surjective map
\begin{align*}
\iota_y^*:\mathcal{S}(Y_{P,P'}(F)) &\lto \mathcal{S}(X_{P \cap M_{\beta_0}}(F))\\
f &\longmapsto f \circ \iota_y
\end{align*}
that fits into a commutative diagram
\begin{equation*}
    \begin{tikzcd}[column sep=huge]
    \mathcal{S}(Y_{P,P'}(F)) \arrow[r,"\mathcal{F}_{P|P^*}"] \arrow[d,"\iota_y^*"] & \mathcal{S}(Y_{P^*,P'}(F)) \arrow[d,"\iota_y^*"]\\
    \mathcal{S}(X_{P \cap M_{\beta_0}}(F)) \arrow[r,"\mathcal{F}_{P \cap M_{\beta_0}|P^* \cap M_{\beta_0}}"] &\mathcal{S}(X_{P^* \cap M_{\beta_0}}(F)).
    \end{tikzcd}
\end{equation*}
If $F$ is archimedean then $\iota_y^*$ is continuous.
\end{prop}

\noindent As observed by the referee, Proposition \ref{prop:iotay} asserts 
that $\mathcal{F}_{P|P'}$ is glued from its restrictions via $\iota_y.$ 

\begin{proof}
We recall that Langlands dual groups are contravariantly functorial with respect to morphisms of reductive algebraic groups $G \to H$ with normal image, and behave as expected with respect to Levi and parabolic subgroups.  For precise statements see \cite[Chapter I]{Borel:Corvallis}.  In particular in view of the commutative diagram
\begin{equation*}
\begin{tikzcd}
M \cap M_{\beta_0} \arrow[d,hookrightarrow]\arrow[r,hookrightarrow] & P \cap M_{\beta_0} \arrow[d,hookrightarrow] \arrow[r,hookrightarrow]  & M_{\beta_0} \arrow[d,hookrightarrow]\\
M \arrow[r,hookrightarrow]& P \cap M' \arrow[r,hookrightarrow]&  M' 
\end{tikzcd}
\end{equation*}
of inclusions of subgroups we can choose the dual groups (which are only defined canonically up to isomorphism) so that they fit in a commutative
diagram
\begin{equation*}
\begin{tikzcd}
\widehat{M \cap M_{\beta_0}} \arrow[r,hookrightarrow] & \widehat{P \cap M_{\beta_0} } \arrow[r,hookrightarrow]  & \widehat{M_{\beta_0}} \\
\widehat{M} \arrow[u,twoheadrightarrow]\arrow[r,hookrightarrow]& \widehat{P \cap M'} \arrow[r,hookrightarrow] \arrow[u,twoheadrightarrow]&  \widehat{M'} \arrow[u,twoheadrightarrow]
\end{tikzcd}
\end{equation*}
where the left horizontal arrows are inclusions of Levi subgroups and the right horizontal arrows are inclusions of parabolic subgroups.  Here the surjectivity of the right vertical arrow follows from the fact that $M_{\beta_0}$ is simple, and the surjectivity of the other vertical arrows follows from the surjectivity of the right vertical arrow.

Consider the representation 
$$
\widehat{\mathfrak{n}}_{P \cap M_{\beta_0}|P^* \cap M_{\beta_0}}=\widehat{\mathfrak{n}}_{P \cap M_{\beta_0}}
$$
of $\widehat{M \cap M_{\beta_0}}.$  
We regard it as a representation of $\widehat{M}$ via the quotient map $\widehat{M} \to \widehat{M \cap M}_{\beta_0}$. Choose a principal $\mathfrak{sl}_2$-triple in $\widehat{\mathfrak{m}}.$  Its image under the quotient map
$$
\widehat{\mathfrak{m}} \lto \widehat{\mathfrak{m} \cap \mathfrak{m}_{\beta_0}}
$$
is a principal $\mathfrak{sl}_2$-triple in $\widehat{\mathfrak{m} \cap \mathfrak{m}_{\beta_0}}.$  Using our comments on dual groups at the beginning of the proof one checks that the quotient map
$$
 \widehat{\mathfrak{n}}_{P|P^*} \lto \widehat{\mathfrak{n}}_{P \cap M_{\beta_0}|P^* \cap M_{\beta_0}} 
$$
is an isomorphism of $\widehat{M}$-representations  that restricts to a bijection
$$
\widehat{\mathfrak{n}}_{P|P^*}^e \lto \widehat{\mathfrak{n}}_{P \cap M_{\beta_0}|P^* \cap M_{\beta_0}}^e.
$$
In view of these observations it is easy to check that the map $\iota_y^*$ is well-defined and surjective, the diagram is commutative, and $\iota_y^*$ is continuous when $F$ is archimedean.
\end{proof}

\begin{cor} \label{cor:Fourier:inverse}
One has 
$$
\mathcal{F}_{P|P^*} \circ \mathcal{F}_{P^*|P}=\mathrm{Id}.
$$
\end{cor}

\begin{proof}
In \cite{BK:normalized} Braverman and Kazhdan prove that $
\mathcal{F}_{P \cap M_{\beta_0}|P^* \cap M_{\beta_0}} \circ \mathcal{F}_{P^* \cap M_{\beta_0}|P \cap M_{\beta_0} }$ is the identity.
Thus the corollary follows from Proposition \ref{prop:iotay}.
\end{proof}

\subsection{The unramified setting}

We assume that $F$ is nonarchimedean and is unramifed over its prime field, that $\psi$ is unramified, and that we are in the unramified setting in the sense of \S \ref{ssec:unram}. 
Let
$$
\one_0 \in C_c^\infty(Y_{P}(F)).
$$
be the characteristic function of the image of $Y(\OO_F)$ in $Y_{P}(F)$. 
  We define the \textbf{basic function}
\begin{align} \label{basic:func}
    b_{Y_{P,P'}}:Y_{P}(F) \lto \CC
\end{align}
to be the unique function in $C^\infty(Y_{P}(F))$ that is finite under a compact open subgroup of $M^{\mathrm{ab}}(F)$ such that 
$$
(b_{Y_{P,P'}})_{\chi_s}=a_{P|P}(\chi_s)(\one_{0})_{\chi_s}
$$
for $\mathrm{Re}(s_{\beta_0})$ sufficiently large.  As explained in \eqref{repl1} and \eqref{repl2}, the spaces $X_{P \cap M_{\beta_0}}$ are special cases of $Y_{P,P'},$ so $b_{X_{P \cap M_{\beta_0}}}$ is defined.

\begin{lem} \label{lem:iotay:basic}
Assume that $y \in P^{\mathrm{der}}(F)\beta_0^\vee(F^\times) Y(\OO_F).$  Then $\iota_y^*(b_{Y_{P,P'}})=b_{X_{P \cap M_{\beta_0}}}.$
\end{lem}
\begin{proof}
We have already explained the relation between $\widehat{\mathfrak{n}}_{P \cap M_{\beta_0}|(P \cap M_{\beta_0})^*}$ and $\widehat{\mathfrak{n}}_{P|P^*}$ as representations of $\widehat{M}$ and $\widehat{M \cap M_{\beta_0}}$ in the proof of Proposition \ref{prop:iotay}.  
This relationship implies that $a_{P|P}(\chi_s)=a_{P \cap M_{\beta_0}|P \cap M_{\beta_0}}((\chi_{\beta_0})_{s_{\beta_0}}).$
\end{proof}

Define $r_\beta$ as in \eqref{ri}.  Arguing as in the proof of Lemma \ref{lem:esti} we obtain the following lemma:

\begin{lem} \label{lem:unram:esti} 
There are constants $\alpha',c>0$ independent of the cardinality of the residue field $q$  such that if $q>c$ and $\alpha'>\alpha>0$ then 
\begin{align*}
|b_{Y_{P,P'}}(x)| \leq  \one_{V_{\beta_0}(\OO_F) \times V_{\Delta_{P'}}^\circ(\OO_F)}(\mathrm{Pl}_{P}(x))\prod_{\beta \in \Delta_P}|\mathrm{Pl}_{P_\beta}(x)|^{\alpha-r_\beta}.
\end{align*} \qed
\end{lem}

\begin{rem} The claim on the independence of the residual characteristic is important because we will require this result for all but finitely many places of a global field.  
\end{rem}

\begin{prop} \label{prop:basic:fixed}
One has $b_{Y_{P,P'}} \in \mathcal{S}(Y_{P,P'}(F)).$  Moreover $\mathcal{F}_{P|P^*}(b_{Y_{P,P'}})=b_{Y_{P^*,P'}}.$
\end{prop}

\begin{proof}
The characters of $Z_{\widehat{M}}$ that appear in $\widehat{\mathfrak{n}}_{P|P^*}$ are all of the form $\lambda\beta_{0}^\vee$ with $\lambda \in \ZZ_{>0}$. Let
$$
V_{\lambda}=\widehat{\mathfrak{n}}_{P|P^*}(\lambda)=\widehat{\mathfrak{n}}_{M_{\beta_0} \cap P}(\lambda)
$$
be the $\lambda\beta_0^\vee$-isotypic space and let $$
r_{\lambda}:\widehat{M}_{\beta_0} \lto \mathrm{Aut}(V_{\lambda})
$$
be the corresponding representation.  It is irreducible \cite[Proposition 4.1]{Shahidi:Ramanujan} \cite{Langlands:EP}.  Let $\mathrm{triv}:(F^\times)^{\Delta_P} \to \CC^\times$ be the trivial character.  Recall
the comments on the relationship between $\widehat{\mathfrak{n}}_{P|P^*}$ and $\widehat{\mathfrak{n}}_{P \cap M_{\beta_0} }$ from the proof of Proposition \ref{prop:iotay}. 
Let $\pi$ be the trivial representation of $M_{\beta_0}(F)$.  The Gindikin-Karpelevic formula implies that 
\begin{align}
    \mathcal{R}_{P|P^*}((\one_{0})_{\mathrm{triv}_s})=(\one_{0})_{\mathrm{triv}_s}^*\prod_{\lambda} \frac{L(\lambda s_{\beta_0}, \pi,r^\vee_\lambda)}{L(1+\lambda s_{\beta_0},\pi,r^\vee_\lambda)}
\end{align}
where the product is over all $\lambda \in \ZZ_{\geq 1}$ such that $V_{\lambda} \neq 0$ \cite[(2.7)]{Shahidi:Ramanujan} \cite[Proposition 4.6]{Lai:Tama}.
Here the $L$-functions are Langlands $L$-functions and $r_\lambda^\vee$ is the dual of $r_\lambda.$ In more detail, $\pi$ determines a Langlands class $c \in \widehat{M}_{\beta_0}(\CC)$ by the Satake isomorphism, and 
$$
L(s,\pi,r_{\lambda}^\vee)=\det\left(I_{V_\lambda}-\frac{r_{\lambda}^\vee(c)}{q^{s}} \right)^{-1}.
$$
In fact, if $\sigma:\mathrm{SL}_2 \to \widehat{M}_{\beta_0}$ is a principal $\mathrm{SL}_2$ then
$c=\sigma\begin{psmatrix} q^{1/2} & \\ & q^{-1/2} \end{psmatrix}$ \cite[\S 7]{Gross:Sat}.

Consider  $\widehat{\mathfrak{n}}_{P|P^*}(\lambda).$  As a representation of a principal $\mathfrak{sl}_2$-triple in $\widehat{\mathfrak{m}}$ it decomposes into a direct sum of irreducible representations in natural bijection with the $\mathcal{N}_i$ in \eqref{BK:decomp} that appear in $\widehat{\mathfrak{n}}^e_{P|P^*}(\lambda).$  The dimension of the corresponding irreducible representation is $2s_i+1,$ where $2s_i$ is the $h$-eigenvalue on $\mathcal{N}_i$ as above.  We conclude that 
\begin{align*}
    \frac{L(\lambda s_{\beta_0},\pi,r_\lambda^\vee)}{L(1+\lambda s_{\beta_0},\pi,r_\lambda^\vee)}&=\prod_{i} \left(\frac{1-q^{-1-s_i-\lambda s_{\beta_0}}}{1-q^{-s_i-\lambda s_{\beta_0}}}\frac{1-q^{-s_i-\lambda s_{\beta_0}}}{1-q^{1-s_i-\lambda s_{\beta_0}}} \cdots \frac{1-q^{-1+s_i-\lambda s_{\beta_0}})}{1-q^{s_i-\lambda s_{\beta_0}}}\right)\\
    &=\prod_{i} \frac{1-q^{-1-s_i-\lambda s_{\beta_0}}}{1-q^{s_i-\lambda s_{\beta_0}}}
\end{align*}
where the product is over $\mathcal{N}_i$ in \eqref{BK:decomp} that appear in $\widehat{\mathfrak{n}}^e_{P|P^*}(\lambda).$ 
Thus
\begin{align} \label{relate:0}
\prod_{\lambda} \frac{L(\lambda s_{\beta_0},\pi,r_\lambda^\vee)}{L(1+\lambda s_{\beta_0},\pi,r_\lambda^\vee)}=\frac{a_{P|P^*}(\mathrm{triv}_s)}{a_{P|P}(\mathrm{triv}_s)}.
\end{align}
We deduce for all unramified $\chi:F^\times \to \CC^\times$ that 
\begin{align}
    \mathcal{R}_{P|P^*}((\one_{0})_{\chi_s})=\frac{a_{P|P^*}(\chi_s)}{a_{P|P}(\chi_s)}(\one_{0})_{\chi_s}^*.
\end{align}
It follows immediately that $b_{Y_{P,P'}} \in \mathcal{S}(Y_{P,P'}(F)).$  
For unramified $\chi$ 
\begin{align}
    \mu_P(\chi_s)=\frac{a_{P^*|P^*}((\chi_s)^{-1})}{a_{P|P^*}(\chi_s)}  
\end{align}
where $\mu_P(\chi_s)$ is defined as in \eqref{muPs}.
Hence 
\begin{align}
    \mu_P(\chi_s) \mathcal{R}_{P|P^*}((\one_{0})_{\chi_s})=\frac{a_{P^*|P^*}((\chi_s)^{-1})}{a_{P|P}(\chi_s)}(\one_{0})_{\chi_s}^*.
\end{align}
Applying Theorem \ref{thm:FT} we have
\begin{align}
    \mathcal{F}_{P|P^*}(b_{Y_{P,P'}})_{\chi_s}^*=a_{P|P}(\chi_s)\mu_P(\chi_s) \mathcal{R}_{P|P^*}((\one_{0})_{\chi_s})=a_{P^*|P^*}((\chi_s)^{-1})(\one_{0})^*_{\chi_s}
\end{align}
Combining this with Mellin inversion we have $\mathcal{F}_{P|P^*}(b_{Y_{P,P'}})=b_{Y_{P^*,P'}}.$
\end{proof}

Let $\varpi$ be a uniformizer for $F$.  Recall the left translation operator $L$ from \eqref{Sch:act}.  For our use in the proof of Theorem \ref{thm:ES:intro} below we require the following lemma:

\begin{lem} \label{lem:Macts}
Let $(\alpha,\lambda) \in \CC \times \ZZ_{\neq 0}.$  For any $f \in \mathcal{S}(Y_{P,P'}(F))^{M(\OO_F) \times H(\OO_F)}$ 
$$
\left(\mathrm{Id}-\frac{\alpha}{\delta_P^{1/2}(\beta_0^{\vee}(\varpi^{\lambda}))}L(\beta_0^{\vee}(\varpi^{-\lambda}))\right)f \in \mathcal{S}(Y_{P,P'}(F))^{M(\OO_F) \times H(\OO_F)}
$$
and 
$$
\left(\left(\mathrm{Id}-\frac{\alpha}{\delta_P^{1/2}(\beta_0^{\vee}(\varpi^{\lambda}))}L(\beta_0^{\vee}(\varpi^{-\lambda}))\right)f\right)_{\chi_s}=(1-\alpha\chi_{\beta_0}(\varpi^{\lambda}))f_{\chi_s}.
$$\qed
\end{lem}

\subsection{The adelic setting}
Now let $F$ be a global field.  
 We let
\begin{align}
\mathcal{S}(Y_{P,P'}(\A_F))=\widehat{\otimes}_{v|\infty}\mathcal{S}(Y_{P,P'}(F_v)) \otimes \otimes'_{v \nmid \infty}\mathcal{S}(Y_{P,P'}(F_v))
\end{align}
where the restricted direct product is taken with respect to the basic functions 
of \eqref{basic:func}.  Here when $F$ is a number field the hat denotes the projective topological tensor product and when $F$ is a function field it is the algebraic tensor product.  The tensor product of the local Fourier transforms induces an isomorphism
\begin{align}
    \mathcal{F}_{P|P^*}:\mathcal{S}(Y_{P,P'}(\A_F)) \lto \mathcal{S}(Y_{P^*,P'}(\A_F)).  
\end{align}
Here we are using Proposition \ref{prop:basic:fixed}.

The following is the global analogue of Proposition \ref{prop:iotay}:
\begin{prop} \label{prop:iotay:glob}
For each $y \in Y(\A_F)$
one has a map
\begin{align*}
\iota_y^*:\mathcal{S}(Y_{P,P'}(\A_F)) &\lto \mathcal{S}(X_{P \cap M_{0}}(\A_F))\\
f &\longmapsto f \circ \iota_y
\end{align*}
that fits into a commutative diagram
\begin{equation*}
    \begin{tikzcd}[column sep=huge]
    \mathcal{S}(Y_{P,P'}(\A_F)) \arrow[r,"\mathcal{F}_{P|P^*}"] \arrow[d,"\iota_y^*"] & \mathcal{S}(Y_{P^*,P'}(\A_F)) \arrow[d,"\iota_y^*"]\\
    \mathcal{S}(X_{P \cap M_{\beta_0}}(\A_F)) \arrow[r,"\mathcal{F}_{P \cap M_{\beta_0}|P^* \cap M_{\beta_0}}"] &\mathcal{S}(X_{P^* \cap M_{\beta_0}}(\A_F)).
    \end{tikzcd}
\end{equation*}
\end{prop}

\begin{proof}
Let $K=\prod_vK_v \leq G(\A_F)$ be a maximal compact subgroup.
The element $y=(y_v) \in Y_{P}(\A_F)$ has the property that $y_v \in K_v$ for almost all $v.$  Thus the proposition follows from its local analogue Proposition \ref{prop:iotay} and the corresponding statement for basic functions in Lemma \ref{lem:iotay:basic}.
\end{proof}

\begin{lem} \label{lem:Mellin:conv}
For $f \in \mathcal{S}(Y_{P,P'}(\A_F))$ (resp.~$f \in \mathcal{S}(Y_{P^*,P'}(\A_F))$ the integrals defining $f_{\chi_s}$ (resp.~$f_{\chi_s}^*$) converge absolutely for $\mathrm{Re}(s_{\beta_0})$ sufficiently large (resp.~sufficiently small). 
\end{lem}

\begin{proof}
This follows from the estimates in lemmas \ref{lem:esti} and \ref{lem:unram:esti}.
\end{proof}

Lemma \ref{lem:Mellin:conv} implies that the Mellin transforms \eqref{Mellin} define maps
\begin{align*}
(\cdot)_{\chi_s}:=(\cdot)_{\chi_s,P}:\mathcal{S}(Y_{P,P'}(\A_F)) &\lto I_P(\chi_s)|_{Y_{P}(\A_F)} \\
 (\cdot)_{\chi_s}^*:=(\cdot)_{\chi_s,P^*}^*:\mathcal{S}(Y_{P^*,P'}(\A_F)) &\lto I_{P^*}^*(\chi_s)|_{Y_{P^*}(\A_F)}
\end{align*}
for $\mathrm{Re}(s_\beta)$ sufficiently large (resp.~sufficiently small).
These Mellin transforms will be used in the following sections.

\section{The Poisson summation formula on $X_{P \cap M_{\beta_0}}(F)$} \label{sec:PS:BK}

The Poisson summation formula on $X_{P \cap M_{\beta_0}}(F)$ was established under some local assumptions on the test functions involved in \cite{BK:normalized} with a slightly different definition of the Schwartz space.  In this section we establish it in general following the arguments of \cite{Getz:Liu:BK}. 

To ease notation, for this section only we assume that $P \leq M_{\beta_0},$ which implies $M_{\beta_0}=G.$  This amounts to assuming that $G$ is simple and $P$ is a maximal parabolic subgroup.  Thus, letting $Y=G$, we have
$$
X_{P \cap M_{\beta_0}}=X_P=Y_{P,G}.
$$
The construction of the Schwartz space and the Fourier transform given in the previous section reduces to the construction of \cite{Getz:Hsu:Leslie} in this case.  
We observe that $\Delta_P=\{\beta_0\}$ in the setting of this section.  Thus we drop $\Delta_P$ and $\beta_0$ from notation when no confusion is likely.

For a quasi-character $\chi:A_{\GG_m} F^\times \backslash \A_F^\times \to \CC^\times,$ $s \in \CC$, 
$f_1 \in \mathcal{S}(X_P(\A_F))$ and $f_2 \in \mathcal{S}(X_{P^*}(\A_F))$ we have  degenerate Eisenstein series
\begin{align} \label{ES} \begin{split}
    E(g,f_{1\chi_s}):&=\sum_{\gamma \in P(F) \backslash G(F)}f_{1\chi_s}(\gamma g)\\
    E^*(g,f^*_{2\chi_s}):&=\sum_{\gamma \in P^*(F) \backslash G(F)}f^*_{2\chi_s}(\gamma g). \end{split}
\end{align}
Here $\chi_s=\chi|\cdot|^s$.
It is well-known that these converge absolutely for $\mathrm{Re}(s)$ large enough (resp.~small enough).  For the proof of absolute convergence in a special case see \cite[Lemma 6.5]{Getz:Liu:BK}; the proof generalizes to our setting.  

Let
\begin{align} \label{aPP}
a_{P|P}(\chi_s):=\prod_v a_{P|P}(\chi_{vs}).
\end{align}  

\begin{lem} \label{lem:holo2} Let $\chi \in \widehat{A_{\GG_m} F^\times \backslash \A_F^\times}$ .
The function $a_{P|P}(\chi_s)$ is holomorphic and nonzero for $\mathrm{Re}(s)> 0.$  It admits a meromorphic continuation to the plane. Moreover there is an integer $n$ depending only on $G$ such that $a_{P|P}(\chi_s)$ is holomorphic if $\chi^n \neq 1.$  
\end{lem}
\begin{proof}
The first claim follows from the same remarks proving Lemma \ref{lem:holo}.
Since $a_{P|P}(\chi_s)$ is a product of (completed) Hecke $L$-functions the second two assertions are clear. 
\end{proof}

\begin{thm}[Langlands] \label{thm:ES:Langlands}
Let $f \in \mathcal{S}(X_P(\A_F))$ be finite under a maximal compact subgroup of $G(\A_F).$  
For a character $\chi:A_{\GG_m} F^\times \backslash \A_F^\times \to \CC^\times$
the Eisenstein series $E(g,f_{\chi_s})$ has a meromorphic continuation to the $s$-plane and admits a functional equation
\begin{align}
    E(g,f_{\chi_s})=E^*(g,\mathcal{F}_{P|P^*}(f)_{\chi_s}^*).
\end{align}
If $\mathrm{Re}(z)=0$ then the order of the pole of $E(g,f_{\chi_s})$  at $s=z$ is bounded by the order of the pole of $a_{P|P}(\chi_s)$ at $s=z.$
\end{thm}

\begin{proof}
To ease translation with the manner the theory is usually phrased, let $Q$ be the unique parabolic subgroup containing $B$ that is conjugate to $P^*,$ and let $M_Q$ be the Levi subgroup of $Q$ containing $T.$  Choose $w \in G(F)$ so that $w P^*w^{-1}=Q$ and  $wMw^{-1}=M_Q.$ There is an isomorphism
\begin{align*} 
    j:C^\infty(X_{P^*}^\circ(F)) &\tilde{\lto} C^\infty(X_Q^\circ(F))\\
    f &\tilde{\longmapsto} (x \mapsto f(w^{-1}x)). 
\end{align*}  For suitable sections $\Phi^{\chi_s} \in  I_Q(\chi_s)$ and $g \in G(\A_F)$ we can then form the Eisenstein series
\begin{align}
E_Q(g,\Phi^{\chi_s})=\sum_{\gamma  \in Q(F) \backslash G(F)}\Phi^{\chi_s}(\gamma g).
\end{align}  
Let $K$ be a maximal compact subgroup of $G(\A_F)$.  If $\Phi^{\chi_s}$ is $K$-finite and holomorphic for $\mathrm{Re}(s)$ sufficiently large, then  $E_Q(g,\Phi^{\chi_s})$ is absolutely convergent for $\mathrm{Re}(s)$ sufficiently large.  We observe that for $f \in \mathcal{S}(X_P(\A_F))$ one has 
\begin{align} \label{relate}
E^*(g,f_{\chi_{s}}^*)=E_Q(g,j (f_{\chi_s}^*))
\end{align}
for $\mathrm{Re}(s)$ sufficiently small.  

By definition of the Schwartz space, for any $f \in \mathcal{S}(X_P(\A_F))$ the section $f_{\chi_s}$ is a holomorphic multiple of the product of completed Hecke $L$-functions $a_{P|P}(\chi_s).$  With this in mind, the meromorphy assertion is proven in \cite{Bernstein:Lapid,Langlands:FE} for $K$-finite functions $f.$ 
 These references also contain a proof of the functional equation
\begin{align*} 
E(g,f_{\chi_s})=E_Q(g,j\left(\mathcal{R}_{P|P^*}(f_{\chi_s})\right))
\end{align*}
which implies by \eqref{relate} that
\begin{align}\label{FE:unnorm}
E(g,f_{\chi_s})=E^*(g, \mathcal{R}_{P|P^*}(f_{\chi_s})).
\end{align}

  We have 
\begin{align} \label{relate:sections}
\mathcal{R}_{P|P^*}(f_{\chi_s})=
(\mathcal{F}_{P|P^*}(f))_{\chi_s}^*
\end{align}
by Theorem \ref{thm:FT} and the argument of \cite[Lemma 6.2]{Getz:Liu:BK}.  Thus the functional equation stated in the theorem follows from \eqref{FE:unnorm}.  

As mentioned above for any $f \in \mathcal{S}(X_P(\A_F))$ the Mellin transform $f_{\chi_s}$ is a holomorphic multiple of $a_{P|P}(\chi_s)$.  Thus the last assertion of the theorem follows from the fact that Eisenstein series attached to $K$-finite holomorphic sections are themselves holomorphic on the unitary axis \cite[Theorem 7.2]{ArthurIntro}, \cite{Langlands:FE}.
\end{proof}

Let $C(\chi_s)$ be the analytic conductor of $\chi_s,$ normalized as in \cite[(5.7)]{Getz:Liu:BK}.  Using notation as in \eqref{VAB} and \eqref{normVAB} one has the following estimate:
\begin{thm} \label{thm:ES:estimate} 
Assume that $F$ is a number field, that Conjecture \ref{conj:poles:intro} is true, and $f \in \mathcal{S}(X_P(\A_F))$ is $K$-finite.  
Let $A<B$ be real numbers and let $p\in \CC[x]$ be a polynomial such that $p(s)E(g,\chi_s)$ is holomorphic in the strip $V_{A,B}.$  Then for all $N \geq 0$ one has
$$
|E(g,f_{\chi_s})|_{A,B,p} \ll_{N,f} C(\chi_s)^{-N}.
$$
\end{thm}
\begin{proof}
The argument is the same as that proving \cite[Theorem 6.3]{Getz:Liu:BK}.  
\end{proof}

\begin{rem}
Lapid has proved a version of Theorem \ref{thm:ES:Langlands} without the assumption of $K$-finiteness \cite{LapidRem}.  However, the proofs we know of Theorem \ref{thm:ES:estimate} still require the Eisenstein series to be $K$-finite.  Indeed, Eisenstein series that are $K$-finite are known to be quotients of functions of finite order (in the sense of complex analysis).  The main input into the proof of Theorem \ref{thm:ES:estimate} is the Phragmen-Lindel\"of principle, which requires the functions involved to be of finite order. 
\end{rem}

Let $K_M <M^{\mathrm{ab}}(\A_F)$ be the maximal compact subgroup and let 
\begin{align} \label{kappaF}
\kappa_F:=\begin{cases}\frac{2^{r_1+r_2}h_F R_F}{d_F^{1/2}e_F}& \textrm{ if }F \textrm{ is a number field}\\
\frac{h_F\log q}{d_F^{1/2}(q-1)}&\textrm{ if }F \textrm{ is a function field}\end{cases}
\end{align}
where $r_1$ and $r_2$ are the number of real (resp.~complex) places, $h_F$ is the class number,  $R_F$ is the regulator, $d_F$ is the absolute discriminant, $e_F$ is the number of roots of unity in $F$,  and in the function field case $F$ has field of constants $\mathbb{F}_q$.  
For complex numbers $s_0$ let
\begin{align} \label{wz}
w(s_0)=\begin{cases} \tfrac{1}{2} \textrm{ if }\mathrm{Re}(s_0)=0,\\
1 \textrm{ otherwise.}\end{cases}
\end{align}
We now prove the following special case of Theorem \ref{thm:PS:intro}:

\begin{thm}\label{thm:PS:BK} 
Let $f \in \mathcal{S}(X_P(\A_F))$.  Assume that
\begin{enumerate}
    \item $F$ is a function field,
    \item $F$ is a number field, Conjecture \ref{conj:poles:intro} is valid, and $f$ is $K_M$-finite, or 
    \item $F$ is a number field and Conjecture \ref{conj:poles} is valid.
\end{enumerate}
One has
\begin{align*} 
\sum_{x \in X_{P}^\circ(F)}&f(x)+\sum_{\chi} \sum_{\substack{s_0 \in \CC\\ \mathrm{Re}(s_0)\geq 0}} \frac{w(s_0)}{\kappa_F} \mathrm{Res}_{s=s_0}E^*(I,\mathcal{F}_{P|P^*}(f)^*_{\chi_{-s}})
\\&=\sum_{x^* \in X_{P^*}^{\circ}(F)}\mathcal{F}_{P|P^*}(f)(x^*)+\sum_{\chi} \sum_{\substack{s_0 \in \CC\\ \mathrm{Re}(s_0)\geq 0}} \frac{w(s_0)}{\kappa_F} \mathrm{Res}_{s=s_0}E(I,f_{\chi_s}).
\end{align*}
Here the sum on $\chi$ is over characters of $A_{\GG_m} F^\times \backslash \A_F^\times.$  The sums over $x$ and $x^*$ are absolutely convergent.  
\end{thm}

\noindent It is clear that if $f$ is $K_M$-finite and Conjecture \ref{conj:poles:intro} is valid then the double sum over $\chi$ and $s_0$ has finite support. The same is true if $F$ is a function field or Conjecture \ref{conj:poles} is valid.  This is how we are using these assumptions here and below.

\begin{rem}Before proving Theorem \ref{thm:PS:BK} we  clarify its meaning.
Assume $F$ is a number field and Conjecture \ref{conj:poles:intro} is valid.  
For any $s_0 \in \CC$ and $f^\infty \in \mathcal{S}(X_P(\A_F^\infty))$ with the finite adels $ \A_F^\infty$ consider the linear functional
\begin{align} \label{residue:func}
f_\infty \longmapsto \mathrm{Res}_{s=s_0}E(I,(f_\infty f^\infty)_{\chi_s}).
\end{align}
It is defined on the dense subspace of $\mathcal{S}(X_P(F_\infty))$ consisting of $K_\infty$-finite functions.  The proof of the theorem shows that this linear functional is continuous with respect to the Fr\'echet topology on $\mathcal{S}(X_P(F_\infty)).$  Hence it extends to all of $\mathcal{S}(X_P(F_\infty)).$  
 For $f_\infty$ that are not $K_\infty$-finite this is the manner the expression $\mathrm{Res}_{s=s_0}E(I,f_{\chi_s})$ is to be interpreted in this paper.   We take the obvious analogous conventions regarding the meaning of $\mathrm{Res}_{s=s_0}E^*(I,f'^*_{\chi_s})$ for $f' \in \mathcal{S}(X_{P^*}(\A_F))$ that are not $K_\infty$-finite. 
 \end{rem}

\begin{proof}[Proof of Theorem \ref{thm:PS:BK}]
Let $K_\infty \leq G(F_\infty)$ be a maximal compact subgroup.  Assume first that $f=f_\infty f^\infty$ where $f_\infty \in \mathcal{S}(X_{P}(F_\infty))$ is $K_\infty$-finite.  
Let 
$$
I_F:=\begin{cases} 
i\left[-\frac{\pi}{\log q},\frac{\pi}{\log q} \right] & \textrm{ if }F \textrm{ is a function field, and }\\
i\RR &\textrm{ if }F \textrm{ is a number field.} \end{cases}
$$
By Mellin inversion and Theorem \ref{thm:ES:estimate} in the number field case
\begin{align}
    \sum_{x \in X_{P}^\circ(F)}f(x)=\sum_{\chi} \int_{\sigma+I_F }E(I,f_{\chi_s}) \frac{ds}{c 2\pi i}
\end{align}
for $\sigma$ sufficiently large and a suitable constant $c.$  
Here the sum over $\chi$ is as in the statement of the theorem.  Since $f$ is $K_\infty$-finite, the support of the sum over $\chi$ is finite.  The constant $c$ may be computed as follows.  It depends on our choice of Haar measure on $[M^{\mathrm{ab}}],$ which is induced by the Haar measure on $M^{\mathrm{ab}}(\A_F)$ fixed in \S \ref{ssec:measures}. This is constructed as explained in \S \ref{ssec:measures} from a  measure on $\A_F$ that is self-dual with respect to $\psi.$  The induced measure on $F \backslash \A_F$ is independent of the choice of $\psi.$  As in \cite{Weil:Basic:NT}, give $A_{\GG_m}$ the Haar measure that corresponds to $\frac{dt}{t}$ under the isomorphism $|\cdot|:A_{\GG_m} \tilde{\to}\RR_{>0}$ if $F$ is a number field and the counting measure if $F$ is a function field.
Thus using \cite[\S VII.6, Proposition 12]{Weil:Basic:NT} and a choice of self-dual measure one obtains
\begin{align} \label{meas}
\mathrm{meas}(A_{\GG_m} F^\times \backslash \A_F^\times)=\begin{cases} \frac{2^{r_1+r_2}h_F R_F}{d_{F}^{1/2}e_F} &\textrm{ if }F \textrm{ is a number field}\\
\frac{h_F }{d_{F}^{1/2}(q-1)} & \textrm{ if }F \textrm{ is a function field.}
\end{cases}
\end{align}
This implies $c=\kappa_F.$  Notice that in the function field case  $\kappa_F=\mathrm{meas}(A_{\GG_m} F^\times \backslash \A_F^\times)\log q$ because of the measure of $I_F.$

We shift contours to $\mathrm{Re}(\sigma)$ very small to see that this is 
\begin{align}
    \sum_{x \in X_{P}^\circ(F)}f(x)=\sum_{\chi} \int_{\sigma'+I_F}E(I,f_{\chi_s}) \frac{ds}{\kappa_F2\pi i}+\frac{1}{\kappa_F}\sum_{\chi}\sum_{s_0 \in \CC} \mathrm{Res}_{s=s_0}E(I,f_{\chi_s})
\end{align}
where now $\sigma'$ is sufficiently small.  Here the support of the sum over $\chi$ and $s_0$ is finite by assumptions (2) or (3) in the number field case and the fact that $E(I,f_{\chi_s})$ is rational in the sense of \cite[IV.1.5]{MW:Spectral:Decomp:ES} in the function field case \cite[Proposition IV.1.12]{MW:Spectral:Decomp:ES}.  In the number field case the bound required to justify the contour shift is provided by Theorem \ref{thm:ES:estimate}.   We now apply the functional equation of Theorem \ref{thm:ES:Langlands} and Mellin inversion to deduce the identity 
\begin{align}  \label{before:fe}
\sum_{x \in X_{P}^\circ(F)}f(x)=\sum_{x^* \in X_{P^*}^{\circ}(F)}\mathcal{F}_{P|P^*}(f)(x^*)+\frac{1}{\kappa_F}\sum_{\chi}\sum_{s_0 \in \CC} \mathrm{Res}_{s=s_0}E(I,f_{\chi_s}).
\end{align}

Since all elements of the Schwartz space $\mathcal{S}(X_P(\A_F))$ are $K_\infty$-finite in the function field case assume for the moment that $F$ is a number field.
Write $K_{M\infty}=K_{M} \cap M^{\mathrm{ab}}(F_\infty).$ We assume without loss of generality that $f=f_\infty f^\infty$ where $f_\infty \in \mathcal{S}(X_P(F_\infty))$, $f^\infty \in \mathcal{S}(X_P(\A_F^\infty))$.  We additionally either assume (2) and that $f_\infty$ transforms according to a particular character $\eta$ under $K_{M\infty}$, or we assume (3).  
Since $K_\infty$-finite functions are dense in $\mathcal{S}(X_{P}(F_\infty))$ by the standard argument \cite[\S 4.4.3.1]{Warner:semisimple:I} we can choose a sequence $\{f_i:i \in \ZZ_{\geq 1}\} \subset \mathcal{S}(X_P(F_\infty))$ of $K_\infty$-finite $f_i$ such that $f_i \to f$ in the Frech\'et topology on $\mathcal{S}(X_P(F_\infty)).$  Under assumption (2) we additionally assume that  the $f_i$ transform under $K_{M\infty}$ by $\eta$.  
We observe that the support of the sum over $\chi$ and $s_0$ in \eqref{before:fe} and its analogues when $f$ is replaced by $f_if^\infty$ are contained in a finite set independent of $i$ under either assumption (2) or (3).  In fact, the finite set can be taken to depend only on the $K_M^\infty$-type of $f^\infty$.  

It is clear from Lemma \ref{lem:esti} that for each fixed $f^\infty \in \mathcal{S}(X_P(\A_F^\infty))$ the map
\begin{align*}
\Lambda_{1,f^\infty}:\mathcal{S}(X_P(F_\infty)) &\lto \CC\\
f_\infty &\longmapsto \sum_{x \in X_P^\circ(F)}f_\infty(x)f^\infty(x)
\end{align*}
is continuous on $\mathcal{S}(X_P(F_\infty)).$  The same is true of 
\begin{align*}
\Lambda_{2,f^\infty}:\mathcal{S}(X_P(F_\infty)) &\lto \CC\\
f_\infty &\longmapsto \sum_{x^* \in X_{P^*}^\circ(F)}\mathcal{F}_{P|P^*}(f_\infty)(x^*)\mathcal{F}_{P|P^*}(f^\infty)(x^*)
\end{align*}
since the Fourier transform is continuous.  Thus $\Lambda_{j,f^\infty}(f_i) \to \Lambda_{j,f^\infty}(f)$ for as $i \to \infty$ for $j \in \{1,2\}.$  Finally consider
\begin{align*}
\Lambda_{f^\infty,s_0}:\mathcal{S}(X_P(F_\infty)) &\lto \CC\\
f_\infty &\longmapsto  \mathrm{Res}_{s=s_0}E(I,f_{\chi_s}).
\end{align*}
Using Lemma \ref{lem:Macts} we can choose an $f'^\infty \in \mathcal{S}(X_{P}(\A_F^\infty))$ such that 
$$
\Lambda_{f^\infty,s_0}=\Lambda_{1,f'^\infty}-\Lambda_{2,f'^\infty}.
$$
In more detail one uses Lemma \ref{lem:Macts} to choose $f'^\infty$ so that the contribution of $\mathrm{Res}_{s=s_0}E(I,f_{\chi_s})$ to \eqref{before:fe} is unchanged, but the contribution of the other residues vanish.  

Thus $\Lambda_{f^\infty,s_0}$ is continuous. We deduce that \eqref{before:fe} is valid for all $K_M$-finite $f$  under assumption (2) and for all $f$ under assumption (3).

We now return to the case of a general global field.
To obtain the expression in the theorem from \eqref{before:fe} we use the functional equation and holomorphy assertion of Theorem \ref{thm:ES:Langlands} and the observeration that for any meromorphic function $f$ on $\CC$ and any $s_0 \in \CC$ one has
\begin{align}
    \mathrm{Res}_{s=s_0}f(s)=-\mathrm{Res}_{s=-s_0}f(-s).
\end{align}
\end{proof}

\begin{lem} \label{lem:lin:func} Let $\chi \in \widehat{A_{\GG_m} F^\times \backslash \A_F^\times}.$  Let $n$ be the maximal order of the pole at $s_0$ of $E(I,f_{\chi_s})$
as $f$ ranges over $K_\infty$-finite elements of $\mathcal{S}(X_P(F_\infty)).$  
For each $f^\infty \in \mathcal{S}(X_P(\A_F^\infty))$ there are continuous linear functionals
$$
\Lambda_i((\cdot)f^\infty):\mathcal{S}(X_P(F_\infty)) \lto \CC
$$
such that if $f_\infty$ is $K_\infty$-finite then
\begin{align} \label{when:K:finite}
    \delta_P^{-1/2}(\beta^\vee_0(t))\sum_{i=0}^{n-1} |t|^{-s_0}(\log|t|)^i \overline{\chi}(t)\Lambda_i(f_\infty f^\infty)=\mathrm{Res}_{s=s_0}E(I,(L(\beta_0^\vee(t))f_\infty f^\infty)_{\chi_s}).
\end{align}
\end{lem}
As usual, when $F$ is a function field we give $\mathcal{S}(X_P(F_\infty))$ the discrete topology.
\begin{proof}
We have $\tfrac{1}{|t|^{s}}=\sum_{j=0}^\infty\frac{(-(s-s_0)\log|t|)^j}{ j! |t|^{s_0}}$ for $s$ in a neighborhood of $s_0$.
On the other hand if $f_\infty$ is $K_\infty$-finite then we can write 
$$
E(I,(f_{\infty}f^\infty)_{\chi_s})=\sum_{i=0}^{n-1}\frac{(-1)^{i}i! \Lambda_i(f_\infty f^\infty)}{(s-s_0)^{i+1}} +g(s)
$$
where $g(s)$ is holomorphic in a neighborhood of $s_0$ and $\Lambda_i(f_\infty f^\infty) \in \CC$.  Then the expression \eqref{when:K:finite} is valid for $K_\infty$-finite $f_\infty.$

The fact that $\Lambda_i((\cdot)f^\infty)$
extends to a continuous functional on all of $\mathcal{S}(X_P(F_\infty))$ is tautological in the function field case.  In the number field case it follows from a variant of the proof of the continuity of the functional $\Lambda_{f^\infty,s_0}$ in the proof of Theorem \ref{thm:PS:BK}.
\end{proof}

\section{Proofs of Theorem \ref{thm:ES:intro} and Theorem \ref{thm:PS:intro}} \label{sec:proofs}

For $f \in \mathcal{S}(Y_{P,P'}(\A_F))$ consider the sum 
\begin{align} \label{sum:xi}
\sum_{y \in Y_{P}(F)}f(y ).
\end{align}

\begin{lem} \label{lem:abs:conv}
The sum \eqref{sum:xi} is absolutely convergent. 
\end{lem}
\begin{proof}
Let $|\cdot|_{\beta}=\prod_{v}|\cdot|_{\beta,v}$, where $|\cdot|_{\beta,v}$ is the norm on $V_{\beta}(F_v)$ fixed above \eqref{normbeta}.
Using lemmas \ref{lem:esti} and \ref{lem:unram:esti} and the notation in these lemmas we see that for any sufficiently small $\alpha>0$ there is a Schwartz function on $\Phi \in \mathcal{S}(V_{\beta_0}(\A_F) \times V_{\Delta_{P'}}^\circ(\A_F))$ such that the sum is dominated by 
\begin{align}
    \sum_{\xi \in V^{\circ}(F)}\Phi(\xi)\prod_{\beta \in \Delta_P}|\xi|^{\alpha-r_\beta}_\beta <\infty
\end{align}
where $V^\circ$ is defined as in \eqref{Vcirc}.
\end{proof}

Recall the function $w(s_0)$ from \eqref{wz}.  The following theorem is the precise version of Theorem \ref{thm:PS:intro}:
\begin{thm} \label{thm:PS}
Let $f \in \mathcal{S}(Y_{P,P'}(\A_F))$.  Assume
\begin{enumerate}
    \item $F$ is a function field,
    \item $F$ is a number field, Conjecture \ref{conj:poles:intro} is valid, and $f$ is $K_M$-finite, or 
    \item $F$ is a number field and Conjecture \ref{conj:poles} is valid.
\end{enumerate}
One has
\begin{align*} 
&\sum_{x \in Y_{P}(F)}f(x )+\sum_{y \in P'^{\mathrm{der}}(F) \backslash Y(F)}\sum_{\chi} \sum_{\substack{s_0 \in \CC\\\mathrm{Re}(s_0)\geq 0}}\frac{w(s_0)}{\kappa_F} \mathrm{Res}_{s=s_0}E^*(I,\iota_y^*(\mathcal{F}_{P|P^*}(f))^*_{\chi_{-s}})\\
&=
\sum_{x^* \in Y_{P^*}(F)}\mathcal{F}_{P|P^*}(f)(x ^*)+\sum_{y \in P'^{\mathrm{der}}(F) \backslash Y(F)}\sum_{\chi} \sum_{\substack{s_0 \in \CC\\\mathrm{Re}(s_0)\geq 0}} \frac{w(s_0)}{\kappa_F}\mathrm{Res}_{s=s_0}E(I,\iota_y^*(f)_{\chi_s})
\end{align*}
where the sums on $\chi$ are over all characters of $A_{\GG_m} F^\times \backslash \A_F^\times.$  Moreover the sums over $x$ and $x^*$ and the triple sums over $(y,\chi,s_0)$ are absolutely convergent. 
\end{thm}

\begin{rem}
We point out that conjectures \ref{conj:poles:intro} and \ref{conj:poles} are assertions about functions in $\mathcal{S}(X_{P \cap M_{\beta_0}}(\A_F)),$ whereas $f$ lies in $\mathcal{S}(Y_{P,P'}(\A_F)).$  
\end{rem}

\begin{proof}
We use Lemma \ref{lem:cosets} to write
\begin{align} \label{before:PS} \begin{split}
\sum_{x\in Y_{P}(F)}f(x )       &=\sum_{ y }
\sum_{\gamma \in X_{P \cap M_{\beta_0}}^{\circ}(F)}
    f(\iota_y(\gamma))
    =\sum_{ y}
\sum_{\gamma \in X_{P \cap M_{\beta_0}}^{\circ}(F)}
    \iota_y^*(f)(\gamma). 
    \end{split}
\end{align}
Here and throughout the proof all sums over $y$ are over $P'^{\mathrm{der}}(F) \backslash Y(F)$.
By Proposition \ref{prop:iotay:glob} we have $\iota_y^*(f) \in \mathcal{S}(X_{P \cap M_{\beta_0}}(\A_F)).$  
There is a natural map $(M \cap M_{\beta_0})^{\mathrm{ab}} \to M^{\mathrm{ab}}$ and hence $(M \cap M_{\beta_0})^{\mathrm{ab}}(\A_F)$ acts on $\mathcal{S}(Y_{P,P'}(\A_F)).$  With respect to this action $\iota_y^*$ is 
 $(M \cap M_{\beta_0})^{\mathrm{ab}}(\A_F)$-equivariant.  In particular if (2) is valid the assumption that $f \in \mathcal{S}(Y_{P,P'}(\A_F))$ is $K_M$-finite implies $\iota_y^*(f)$ is finite under the maximal compact subgroup of $(M \cap M_{\beta_0})^{\mathrm{ab}}(\A_F).$  Hence applying Theorem \ref{thm:PS:BK}
we see that the above is 
\begin{align*}
&-\sum_{y }\Bigg(\sum_{\chi} \sum_{\substack{s_0 \in \CC\\\mathrm{Re}(s_0)\geq 0}} \frac{w(s_0)}{\kappa_F}\mathrm{Res}_{s=s_0}E^*(I,(\mathcal{F}_{P \cap M_{\beta_0}|P^* \cap M_{\beta_0}}(\iota_y^*(f)))^*_{\chi_{-s}})\Bigg)\\&+\sum_{ y }\Bigg(
\sum_{\gamma \in X_{P^* \cap M_{\beta_0}}^{\circ}(F)}
    \mathcal{F}_{P \cap M_{\beta_0}|P^* \cap M_{\beta_0}}(\iota_y^*(f))(\gamma)    +\sum_{\chi}\sum_{\substack{s_0 \in \CC\\\mathrm{Re}(s_0)\geq 0}} \frac{w(s_0)}{\kappa_F}\mathrm{Res}_{s=s_0}E(I,\iota_y^*(f)_{\chi_s})\Bigg).
\end{align*}
By Proposition \ref{prop:iotay:glob} this is 
\begin{align*}
&-\sum_{y }\Bigg(\sum_{\chi} \sum_{\substack{s_0 \in \CC\\\mathrm{Re}(s_0)\geq 0}} \frac{w(s_0)}{\kappa_F}\mathrm{Res}_{s=s_0}E^*(I,\iota_y^*(\mathcal{F}_{P |P^*}(f))^*_{\chi_{-s}})\Bigg)\\&+\sum_{ y }\Bigg(
\sum_{\gamma \in X_{P^* \cap M_{\beta_0}}^{\circ}(F)}
    \iota_y^*(\mathcal{F}_{P|P^*}(f))(\gamma)    +\sum_{\chi}\sum_{\substack{s_0 \in \CC\\\mathrm{Re}(s_0)\geq 0}} \frac{w(s_0)}{\kappa_F}\mathrm{Res}_{s=s_0}E(I,\iota_y^*(f)_{\chi_s})\Bigg).
\end{align*}
By Lemma \ref{lem:cosets}
\begin{align}
\sum_{ y }
\sum_{\gamma \in X_{P^* \cap M_{\beta_0}}^{\circ}(F)}
    \iota_y^*(\mathcal{F}_{P |P^*}(f))( \gamma )
 \label{after:PS}
=\sum_{x^* \in Y_{P^*}(F)}
    \mathcal{F}_{P |P^*}(f)( x^*).
\end{align}
This completes the proof of the identity in the theorem.

For the absolute convergence statement in the theorem, we observe that \eqref{before:PS} and \eqref{after:PS} are absolutely convergent by Lemma \ref{lem:abs:conv}.  By the argument above and the functional equation in Theorem \ref{thm:ES:Langlands} we see that 
\begin{align} \label{for:AC} \begin{split}
\sum_{\gamma \in X_{P \cap M_{\beta_0}}^{\circ}(F)}&
    f(\iota_y(\gamma))-\sum_{\gamma \in X_{P^* \cap M_{\beta_0}}^{\circ}(F)}\iota_y^*(
    \mathcal{F}_{P|P^* }(f))(\gamma)\\
 &=\sum_{\gamma \in X_{P \cap M_{\beta_0}}^{\circ}(F)}
    f(\iota_y(\gamma))-\sum_{\gamma \in X_{P^* \cap M_{\beta_0}}^{\circ}(F)}
    \mathcal{F}_{P \cap M_{\beta_0}|P^* \cap M_{\beta_0}}(\iota_y^*(f))(\gamma)\\   &=\frac{1}{\kappa_F}\sum_{\chi} \sum_{s_0 \in \CC} \mathrm{Res}_{s=s_0}E(I,\iota_y^*(f)_{\chi_s}) \end{split}
\end{align} 
for all $y.$ 
Arguing as above, the sum over $y$ of the first line of \eqref{for:AC} converges absolutely, and hence the sum over $y$ in
\begin{align*}
   \frac{1}{\kappa_F}\sum_{y} \left(\sum_{\chi} \sum_{s_0 \in \CC} \mathrm{Res}_{s=s_0}E(I,\iota_y^*(f)_{\chi_s})\right)
\end{align*}
is absolutely convergent as well.  

The support of the sum over $\chi$ is finite and under assumptions (1) or (2) it depends only on the $K_M$-type of $f.$  Under assumption (3) it depends only on the $K_M \cap M(\A_F^\infty)$-type of $f.$  With this in mind, for any fixed $s_0 \in \CC$ we can use Lemma \ref{lem:Macts} to choose an $f' \in \mathcal{S}(Y_{P,P'}(\A_F))$ with the same $K_M$-type as $f$ so that 
\begin{align} 
   &\sum_{y } \left(\sum_{\chi} \sum_{s_0 \in \CC} \mathrm{Res}_{s=s_0}E(I,\iota_y^*(f')_{\chi_s})\right)\label{one:term}=\sum_{y}   \mathrm{Res}_{s=s_0}E(I,\iota_y^*(f)_{\chi_s}). 
\end{align}
We deduce that the right hand side of \eqref{one:term} is absolutely convergent, and the same is true if we replace $E(I,\iota_y^*(f)_{\chi_s})$ by $E^*(I,\iota_y^*(\mathcal{F}_{P|P^*}(f))^*_{\chi_s})$  by Proposition \ref{prop:iotay} and Theorem \ref{thm:ES:Langlands}.
\end{proof}

Let $(f_1,f_2) \in \mathcal{S}(Y_{P,P'}(\A_F)) \times \mathcal{S}(Y_{P^*,P'}(\A_F)).$
  Recall the definition of the generalized Schubert Eisenstein series $E_{Y_P}(f_{1\chi_s})$ and $E_{Y_{P^*}}^*(f_{2\chi_s}^*)$ of  \eqref{SEis}.  Using
 lemmas \ref{lem:esti} and \ref{lem:unram:esti} and standard arguments one obtains the following lemma:
\begin{lem} \label{lem:convergence}
For $f_1 \in \mathcal{S}(Y_{P,P'}(\A_F))$ one has 
\begin{align*}
\sum_{x \in M^{\mathrm{ab}}(F) \backslash Y_{P}(F)}\int_{M^{\mathrm{ab}}(\A_F)}  \delta_{P}^{1/2}(m)|\chi_s(\omega_P(m))f_1(m^{-1}x)|dm<\infty
\end{align*}
provided that $\mathrm{Re}(s_{\beta})$ is sufficiently large for all $\beta \in \Delta_P.$  In particular,
$E_{Y_P}(f_{1\chi_s})$ converges absolutely provided that $\mathrm{Re}(s_{\beta})$ is sufficiently large for all $\beta \in \Delta_P.$
 Similarly, for $f_2 \in \mathcal{S}(Y_{P^*,P'}(\A_F))$ the  series $E_{Y_{P^*}}^*(f_{2\chi_s}^*)$ converges absolutely provided that $\mathrm{Re}(s_{\beta_0})$ is sufficiently small and $\mathrm{Re}(s_{\beta})$ is sufficiently large for $\beta \neq \beta_0.$ \qed
\end{lem}

With notation as in Lemma \ref{lem:Plucker:stuff}, let
\begin{align}
    A^{\beta_0}:=\ker \omega_{\beta_0}:M^{\mathrm{ab}} \lto \GG_m
\end{align}
Thus $M^{\mathrm{ab}}$ is the internal direct sum of $\beta_0^\vee(\GG_m)$ and $A^{\beta_0}$ by Lemma \ref{lem:Plucker:stuff}. 
Recall that $M'$ is the unique Levi subgroup of $P'$ containing $M$.  Applying Lemma \ref{lem:Plucker:stuff} with $P$ replaced by $P'$ we see that the canonical map $M^{\mathrm{ab}} \to M'^{\mathrm{ab}}$ restricts to induce an isomorphism 
\begin{align} \label{Abeta0:isom}
A^{\beta_0} \tilde{\lto} M'^{\mathrm{ab}}.
\end{align}

We have a Haar measure on $\A_F^\times$ and we endow $A^{\beta_0}(\A_F)$ with a Haar measure by declaring that the product measure on $A^{\beta_0}(\A_F)\beta_0^\vee(\A_F^\times)$ is our Haar measure on $M^{\mathrm{ab}}(\A_F).$  
As usual let
\begin{align*}
    (\A_F^\times)^1:=\{t \in \A_F^\times:|t|=1\}, \quad 
    (\A_F^\times)^+:=\{t \in \A_F^\times:|t|>1\}, \quad 
    (\A_F^\times)^-:=\{t \in \A_F^\times:|t|<1\}.
\end{align*}
For algebraic $F$-groups $G$ we set
$$
[G]:=G(F) \backslash G(\A_F)
$$
and define
\begin{align*}
    [\GG_m]^1:=F^\times \backslash (\A_F^\times)^1, \quad [\GG_m]^ \pm:=F^\times \backslash (\A_F^\times)^\pm.
\end{align*}
In the proof of Theorem \ref{thm:PS:BK} we endowed $A_{\GG_m}$ and hence $A_{\GG_m} \backslash \A_F^\times$ with Haar measures.  We endow $(\A_F^\times)^1$ with the pull-back of this measure along the canonical isomorphism $(\A_F^\times)^1 \tilde{\to} A_{\GG_m} \backslash \A_F^\times.$

The following technical lemma will be used in the proof of Theorem \ref{thm:ES:intro}:
\begin{lem} \label{lem:mero}
Let $f \in \mathcal{S}(Y_{P,P'}(\A_F)).$  
Assume 
\begin{enumerate}
\item $F$ is a function field,
    \item $F$ is a number field, Conjecture \ref{conj:poles:intro} is valid and $f$ is $K_M$-finite, or
    \item $F$ is a number field and Conjecture \ref{conj:poles} is valid.
\end{enumerate}
Fix $(s_{\beta}) \in \CC^{\Delta_{P'}}$ such that  $\mathrm{Re}(s_\beta)$ is sufficiently large for all $\beta \neq \beta_0$.
For every $\chi \in (\widehat{A_{\GG_m} F^\times \backslash \A_F^\times})^{\Delta_P},$ $\chi' \in \widehat{A_{\GG_m} F^\times \backslash \A_F^\times}$ and $s_0 \in \CC$ one has that
\begin{align} \label{is:conv}
 &\sum_{y }  \int_{[\GG_m]^1 \times A^{\beta_0}(\A_F)}\delta_{P}^{1/2}(\beta_0^\vee(t)m)\Big|\chi_s(t\omega_P(m))\mathrm{Res}_{z=s_0}E(I,\iota_y^*(L(\beta_0^\vee(t)m)f)_{\chi'_{z}})\Big|dmd^\times t
 \end{align}
is finite.  Moreover for $\mathrm{Re}(s_{\beta_0})$ sufficiently large
\begin{align} \label{+}
    &\sum_{y}\int_{[\GG_m]^- \times A^{\beta_0}(\A_F)}\delta_{P}^{1/2}(\beta_0^\vee(t)m)\Big|\chi_s(t\omega_P(m)) \mathrm{Res}_{z=s_0}E^*(I,\iota_y^*(L(\beta_0^\vee(t)m)f)^*_{\chi'_{ z}})\Big|dm d^\times t
\end{align}
converges.  The expression 
\begin{align} \label{+2}
 &\sum_{y }  \int_{ [\GG_m]^- \times A^{\beta_0}(\A_F)}\delta_{P}^{1/2}(\beta_0^\vee(t)m)\chi_s(t\omega_P(m))\mathrm{Res}_{z=s_0}E(I,\iota_y^*(L(\beta_0^\vee(t)m)f)_{\chi'_{z}})dm d^\times t
\end{align}
originally defined for $\mathrm{Re}(s_{\beta_0})$ large, admits a meromorphic continuation as a function of $s_{\beta_0}.$   Here all sums over $y$ are over $ P'(F) \backslash Y(F)$.
\end{lem}
\begin{proof}  The group $[\GG_m]^1$ is compact.  Using this observation and a variant of Lemma \ref{lem:convergence} we see that for all $f_1 \in \mathcal{S}(Y_{P,P'}(\A_F))$ and $f_2 \in \mathcal{S}(Y_{P^*,P'}(\A_F))$ the integrals
\begin{align} \label{are:convergent} \begin{split}
&\int_{[\GG_m]^1 \times [A^{\beta_0}]}\delta_P^{1/2}(\beta_0^\vee(t)m) |\chi_s(t\omega_P(m))|\sum_{x \in Y_{P}(F)}|f_1((\beta_0^\vee(t)m)^{-1}x)|dm d^\times t\\
& \int_{[\GG_m]^1 \times [A^{\beta_0}]}\delta_{P}^{1/2}(\beta_0^\vee(t)m) |\chi_s(t\omega_P(m))|\sum_{x^* \in Y_{P^*}(F)}|f_2((\beta_0^\vee(t)m)^{-1}x^*)|dmd^\times t, \end{split}
\end{align}
are convergent.  Here and throughout the proof we assume that $\mathrm{Re}(s_{\beta})$ is sufficiently large for $\beta \neq \beta_0.$

Let $(f_\infty,f^\infty) \in \mathcal{S}(Y_{P,P'}(F_\infty)) \times \mathcal{S}(Y_{P,P'}(\A_F^\infty))$, let $(t,m) \in \A_F^\times \times A^{\beta_0}(\A_F)$, and let $y \in Y(F).$
Arguing as in the proof of Theorem \ref{thm:PS}  for each $\chi' \in \widehat{A_{\GG_m} F^\times \backslash \A_F^\times}$ and $s_0 \in \CC$
there is an $f'^\infty \in \mathcal{S}(Y_{P,P'}(\A_F^\infty))$ such that
\begin{align}\label{the:above} \begin{split} 
&\sum_{x \in X_{P \cap M_{\beta_0}(F)}}
    \iota_y(f_\infty f'^\infty)((\beta_0^\vee(t)m)^{-1} x)-  \sum_{x^*\in X_{P^* \cap M_{\beta_0}}(F)}
    \mathcal{F}_{P \cap M_{\beta_0} |P^* \cap M_{\beta_0}}\iota_y^*(L(\beta_0^{\vee}(t)m)(f_\infty f'^\infty))(x^*)\\
    &=\mathrm{Res}_{z=s_0}E(I,\iota_y^*(L(\beta_0^\vee(t)m)(f_\infty f^\infty))_{\chi'_{z }}). \end{split}
\end{align}
Multiplying both sides of \eqref{the:above} by $\delta_P^{1/2}(\beta_0^\vee(t)m)\chi_s(t\omega_P(m))$, summing over $y \in P'^{\mathrm{der}}(F) \backslash G(F)$ and then integrating over $[\GG_m]^1 \times [A^{\beta_0}]$ we arrive at 
\begin{align}\label{the:above2} \begin{split} 
\int_{[\GG_m]^1 \times [A^{\beta_0}]}\sum_{y}\delta_P^{1/2}(\beta_0^\vee(t)m)&\chi_s(t\omega_P(m))\Bigg(\sum_{x \in X_{P \cap M_{\beta_0}(F)}}
    \iota_y(f_\infty f'^\infty)((\beta_0^\vee(t)m)^{-1} x)\\&-  \sum_{x^*\in X_{P^* \cap M_{\beta_0}}(F)}
    \mathcal{F}_{P \cap M_{\beta_0} |P^* \cap M_{\beta_0}}\iota_y^*(L(\beta_0^{\vee}(t)m)(f_\infty f'^\infty))(x^*)\Bigg) dm d^\times t\\
    =\int_{[\GG_m]^1 \times [A^{\beta_0}]}\delta_P^{1/2}(\beta_0^\vee(t)m))&\chi_s(t\omega_P(m))\sum_{y} \mathrm{Res}_{z=s_0}E(I,\iota_y^*(L(\beta_0^\vee(t)m)(f_\infty f^\infty))_{\chi'_{ z }})dmd^\times t. \end{split}
\end{align}
To prove that the bottom sum and integral in \eqref{the:above2} converge absolutely it suffices to prove that the top sum and integral converge absolutely.  
Using Lemma \ref{lem:cosets} and Proposition \ref{prop:iotay:glob} this reduces to the assertion that the integrals in \eqref{are:convergent} are convergent.  
Hence
\begin{align}  \label{before:unfold}
\int\sum_{y \in P'^{\mathrm{der}}(F) \backslash Y(F)}\delta_P^{1/2}(\beta_0^\vee(t)m))\Big|\chi_s(t\omega_P(m)) \mathrm{Res}_{z=s_0}E(I,\iota_y^*(L(\beta_0^\vee(t)m)(f_\infty f^\infty))_{\chi'_{z }})\Big|dmd^\times t
\end{align}
converges, where the integral is over 
$[\GG_m]^1 \times [A^{\beta_0}].$  Using \eqref{Abeta0:isom} we see that  \eqref{before:unfold} unfolds to
\eqref{is:conv}.
This completes the proof of the first assertion of the lemma.  

Let $n$ be the maximal order of the pole of $E(I,f'_{\chi'_{z}})$ at $z=s_0$ as $f' \in \mathcal{S}(X_{P \cap M_{\beta_0}}(\A_F))$ varies.
Recall the continuity assertion for $\iota_y^*$ contained in Proposition \ref{prop:iotay}.  Using the first absolute convergence assertion of the lemma and  Lemma \ref{lem:lin:func}  there are continuous linear functionals
\begin{align}
 \Lambda_{i,y}((\cdot)f^\infty):\mathcal{S}(Y_{P,P'}(F_\infty)) \lto \CC
\end{align}
for $0 \leq i \leq n-1$ 
such that  
\begin{align} \label{lambdas} \begin{split}
    &\sum_{y \in P'(F) \backslash Y(F)}\int_{ A^{\beta_0}(\A_F)}\delta_{P}^{1/2}(\beta_0^\vee(t)m)\chi_s(t\omega_P(m)) \mathrm{Res}_{z=s_0}E^*(I,\iota_y^*(L(\beta_0^\vee(t)m)f)^*_{\chi'_{ z}})dm \\
&=  \sum_{y \in P'(F) \backslash Y(F) } \int_{A^{\beta_0}(\A_F)} \left( \sum_{i=0}^{n-1}\chi_s(t\omega_P(m))|t|^{-s_{0}}(\log|t|)^i\overline{\chi}(t) \Lambda_{i,y}(L(m)(f_\infty f^\infty))\right) dm
\end{split}
\end{align}
for all $t \in \A_F^\times.$
Here when $F$ is a function field we give $\mathcal{S}(Y_{P,P'}(F_\infty))$ the discrete topology.  Varying $f^\infty$ and using Lemma \ref{lem:Macts} we deduce that 
\begin{align}
    \sum_{y \in P'(F) \backslash Y(F)}  \int_{A^{\beta_0}(\A_F)}|\chi_s(\omega_P(m)\Lambda_{i,y}(L(m)(f_\infty f^\infty))|dm
\end{align}
is convergent for each $i.$ Using this fact we see that for $\mathrm{Re}(s_{\beta_0})$ sufficiently large the integral over $t \in [\GG_m]^-$ of \eqref{lambdas} is 
\begin{align}
    \sum_{y \in P'(F) \backslash Y(F) } \int_{A^{\beta_0}(\A_F)} \left( \int_{[\GG_m]^-}\sum_{i=0}^{n-1}\chi_s(t\omega_P(m))|t|^{-s_{0}}(\log|t|)^i\overline{\chi}(t)d^\times t \Lambda_{i,y}(L(m)(f_\infty f^\infty))\right) dm;
\end{align}
in other words, it is permissible to bring the integral over $t$ inside the other integral and the sum.  The integrals over $t$ may be evaluated in an elementary manner, yielding the remaining assertions of the lemma.
\end{proof}

\begin{proof}[Proof of Theorem \ref{thm:ES:intro}]  
Assume for the moment that $f$ is $K_M$-finite and that $\mathrm{Re}(s_{\beta})$ is sufficiently large for all $\beta \in \Delta_P$.  Then, using Lemma \ref{lem:convergence} to justify convergence, we have
\begin{align*}
    E_{Y_P}(f_{\chi_s})=&\int_{[M^{\mathrm{ab}}]}\delta_P^{1/2}(m)\chi_s(\omega_P(m))\sum_{x \in Y_{P}(F)}f(m^{-1} x )dm\\
    =&\int_{[\GG_m]^+ \times [A^{\beta_0}]}\delta_P^{1/2}(\beta_0^\vee(t)m)\chi_s(t\omega_P(m))\sum_{x \in Y_{P}(F)}f((\beta_0^\vee(t)m)^{-1}x )dmd^\times t\\
    &+\int_{[\GG_m]^1 \times [A^{\beta_0}]}\delta_P^{1/2}(\beta_0^\vee(t)m)\chi_s(t\omega_P(m))\sum_{x\in Y_{P}(F)}f((\beta_0^\vee(t)m)^{-1}x )dmd^\times t\\
    &+\int_{[\GG_m]^- \times [A^{\beta_0}]}\delta_P^{1/2}(\beta_0^\vee(t)m)\chi_s(t\omega_P(m))\sum_{x \in Y_{P}(F)}f((\beta_0^\vee(t)m)^{-1}x )dmd^\times t.
\end{align*}
Here $d^\times t$ is the measure on $[\GG_m].$
The subgroup $ [\GG_m]^1<[\GG_m]$ has nonzero measure with respect to $d^\times t$ if and only if $F$ is a function field.

By Lemma \ref{lem:equiv} one has 
\begin{align} \label{equi:latter}
\mathcal{F}_{P|P^*} \circ L(m)=\delta_{P^* \cap M'}(m) L(m) \circ \mathcal{F}_{P|P^*}.
\end{align}
Moreover $\delta_{P}^{1/2}(m)\delta_{P^* \cap M'}(m)=\delta_{P^*}^{1/2}(m).$
With this in mind we apply the Poisson summation formula of Theorem \ref{thm:PS} to $\tfrac{1}{2}$ the second integral and the third integral above to see that $E_{Y_P}(f_{\chi_s})$ is the sum of \eqref{1} and \eqref{2} below:
\begin{align}
\label{1}
&\int_{[\GG_m]^+ \times [A^{\beta_0}]}\delta_P^{1/2}(\beta_0^\vee(t)m)\chi_s(t\omega_P(m))\sum_{x \in Y_{P}(F)}f((\beta_0^\vee(t)m)^{-1}x )dm d^\times t\\ \nonumber 
    &+\frac{1}{2}\int_{[\GG_m]^1 \times [A^{\beta_0}]
    }\delta_P^{1/2}(\beta_0^\vee(t)m)\chi_s(t\omega_P(m))\sum_{x \in Y_{P}(F)}f((\beta_0^\vee(t)m)^{-1}x )dmd^\times t\\ \nonumber 
    &+\frac{1}{2}\int_{[\GG_m]^1 \times [A^{\beta_0}]}\delta_{P^*}^{1/2}(\beta_0^\vee(t)m)\chi_s(t\omega_P(m))\sum_{x \in Y_{P^*}(F)}\mathcal{F}_{P|P^*}(f)((\beta_0^\vee(t)m)^{-1} x )dm d^\times t\\ \nonumber
    &+\int_{[\GG_m]^- \times [A^{\beta_0}]}\delta_{P^*}^{1/2}(\beta_0^\vee(t)m)\chi_s(t\omega_P(m))\sum_{x \in Y_{P^*}(F)}\mathcal{F}_{P|P^*}(f)((\beta_0^\vee(t)m)^{-1}x )dm d^\times t
\end{align}
and
\begin{align}
\label{2}
\frac{1}{2}\int_{[\GG_m]^1 \times [A^{\beta_0}]
}\chi_s(t\omega_P(m))&\sum_{\substack{s_0 \in \CC\\ \mathrm{Re}(s_0) \geq  0}}\frac{w(s_0)}{\kappa_F}\Bigg( \sum_{y,\chi'} \delta^{1/2}_P(\beta_0^\vee(t)m) \mathrm{Res}_{z=s_0}E(I,\iota_y^*(L(\beta_0^\vee(t)m)f)_{\chi_z'})\\&-\sum_{y,\chi'}  \delta^{1/2}_{P^*}(\beta_0^\vee(t)m) \mathrm{Res}_{z=s_0}E^*(I,\iota_y^*(L(\beta_0^\vee(t)m)\mathcal{F}_{P|P^*}(f))^*_{\chi'_{-z}})\Bigg)dmd^\times t \nonumber \\
    +\int_{[\GG_m]^- \times [A^{\beta_0}]}\chi_s(t\omega_{P}(m))&\sum_{\substack{s_0 \in \CC\\ \mathrm{Re}(s_0)\geq  0}}\frac{w(s_0)}{\kappa_F}\Bigg(\sum_{y,\chi'}  \delta^{1/2}_P(\beta_0^\vee(t)m) \mathrm{Res}_{z=s_0}E(I,\iota_y^*(L(\beta_0^\vee(t)m)f)_{\chi'_z})\nonumber \\&-\sum_{y,\chi'}  \delta_{P^*}^{1/2}(\beta_0^\vee(t)m) \mathrm{Res}_{z=s_0}E^*(I,\iota_y^*(L(\beta_0^\vee(t)m)\mathcal{F}_{P|P^*}(f))^*_{\chi_{-z}'}) \nonumber \Bigg)dm d^\times t 
\end{align}
where the sums over $y$ are over $P'^{\mathrm{der}}(F) \backslash Y(F)$ and the sums over $\chi'$ are over $\widehat{A_{\GG_m} F^\times \backslash \A_F^\times}.$ 
We recall that the sums over $s_0$ have finite support in a set depending only on the $K_M$-type of $f$ by Conjecture \ref{conj:poles:intro} in the number field case and the fact that $E(I,f_{\chi_s})$ is rational in the sense of \cite[IV.1.5]{MW:Spectral:Decomp:ES} in the function field case \cite[Proposition IV.1.12]{MW:Spectral:Decomp:ES}.  Moreover, using Lemma \ref{lem:equiv} we see that the sums over $\chi'$ have support in a finite set depending only on the $K_M$-type of $f.$

For the remainder of the proof we continue to assume that $\mathrm{Re}(s_{\beta})$ is sufficiently large for $\beta \in \Delta_{P'}.$  Our goal is to understand the behavior of $E_{Y_P}(f_{\chi_s})$ as a function of $s_{\beta_0}.$
Lemma \ref{lem:convergence} implies that the upper two integrals in \eqref{1} converge absolutely for $\mathrm{Re}(s_{\beta_0})$ large enough, and hence they converge for all $s_{\beta_0}.$
Lemma \ref{lem:convergence} also implies that the lower two integrals in \eqref{1} converge absolutely for $\mathrm{Re}(s_{\beta_0})$ small enough.  Thus they converge for all $s_{\beta_0}$.   We deduce that \eqref{1} is a holomorphic function of $s_{\beta_0}.$

Now consider \eqref{2}.  We have the functional equation
\begin{align} \label{fe}
E(I,\iota_y^*(L(m)f)_{\chi_{s_{\beta_0 }}})=\delta_{P^* \cap M'}(m)E^*(I,\iota_y^*(L(m)\mathcal{F}_{P|P^*}(f))_{\chi_{s_{\beta_0}}})
\end{align}
by Proposition \ref{prop:iotay:glob}, Theorem \ref{thm:ES:Langlands} and \eqref{equi:latter}.  
 Using Lemma \ref{lem:mero} and \eqref{fe} to justify switching sums and integrals, we deduce that for $\mathrm{Re}(s_{\beta_0})$ sufficiently large \eqref{2} is equal to \begin{align}
\label{3}
\frac{1}{2}\sum_{\substack{s_0 \in \CC\\ \mathrm{Re}(s_0)\geq 0}}&\frac{w(s_0)}{\kappa_F}\sum_{y}   \int_{[\GG_m]^1 \times A^{\beta_0}(\A_F)}\chi_s(t\omega_P(m))\Bigg( \delta_P^{1/2}(t\omega_P(m))\mathrm{Res}_{z=s_0}E(I,\iota_y^*(L(\beta_0^\vee(t)m)f)_{\chi_{z_{\beta_0}}})\\
&- \delta_{P^*}^{1/2}(t\omega_P(m))\mathrm{Res}_{s=s_0}E^*(I,\iota_y^*(L(\beta_0^\vee(t)m)\mathcal{F}_{P|P^*}(f))^*_{(\chi_{-z_{\beta_0}})}
\Bigg)dm d^\times t \nonumber \\
    +\sum_{\substack{s_0 \in \CC\\ \mathrm{Re}(s_0) \geq 0}}&\frac{w(s_0)}{\kappa_F}\sum_y 
    \int_{[\GG_m]^- \times A^{\beta_0}(\A_F)}\chi_s(t\omega_{P}(m))\Bigg(\delta_{P}^{1/2}(\beta_0^\vee(t)m)\mathrm{Res}_{z=s_0}E(I,\iota_y^*(L(\beta_0^\vee(t)m)f)_{\chi_{z_{\beta_0}}}) \nonumber \\
    &-    \delta_{P^*}^{1/2}(\beta_0^\vee(t)m)\mathrm{Res}_{z=s_0}E^*(I,\iota_y^*(L(\beta_0^\vee(t)m)\mathcal{F}_{P|P^*}(f))^*_{(\chi_{-z_{\beta_0}}) })\Bigg)dmd^\times t  \nonumber
\end{align}
where the sums on $y$ are now over $P'(F) \backslash G(F).$  Here we are using \eqref{Abeta0:isom} to unfold the integral.
By Lemma \ref{lem:mero} and \eqref{fe} the expression \eqref{3}
is holomorphic for $\mathrm{Re}(s)$ large, and admits a meromorphic continuation to the plane.  Thus under the assumption that $f$ is $K_M$-finite, we have proved the meromorphic continuation statement in the theorem.

By a symmetric argument 
we deduce that the sum of \eqref{1} and \eqref{2} is 
 $E^*_{Y_{P^*}}(\mathcal{F}_{P|P^*}(f)_{\chi_s}^*).$   This proves the functional equation 
$$
E_{Y_P}(f_{\chi_s})=E^*_{Y_{P^*}} (\mathcal{F}_{P|P^*}(f)_{\chi_s}^*).
$$
Thus far we have assumed that $f$ is $K_M$-finite.  To remove this assumption we note that for any $f \in \mathcal{S}(X_{P}(\A_F))$ and any $\chi \in \widehat{A_{\GG_m} F^\times \backslash \A_F^\times}$ there is a $K_M$-finite $f' \in \mathcal{S}(X_P(\A_F))$ such that $f_{\chi_s}=f_{\chi_s}'$.  It will then satisfy $\mathcal{F}_{P|P^*}(f)_{\chi_s}^*=\mathcal{F}_{P|P^*}(f')_{\chi_s}^*$ by Theorem \ref{thm:FT}. This allows us to remove the $K_M$-finiteness assumption.
\end{proof}

The proof of Theorem \ref{thm:ES:intro} shows that the poles of $E_{Y_P}(f_{\chi_s})$ are controlled by the poles of $E(I,f'_{\chi_{s_{\beta_0 }}})$ for $f' \in \mathcal{S}(X_{P \cap M_{\beta_0}}(\A_F)).$  In particular it is easy to deduce the following corollary from the proof:
\begin{cor} \label{cor:poles}
Under the hypotheses of Theorem \ref{thm:ES:intro}, for fixed $(s_{\beta}) \in \CC^{\Delta_{P'}}$ with $\mathrm{Re}(s_{\beta})$ sufficiently large (for $\beta \neq \beta_0$) the order of the pole of $E_{Y_P}(f_{\chi_s})$ at $s_{\beta_0}=s_0$ is no greater than the maximum of the orders of the pole of
 $E(I,f'_{\chi_{z_{\beta_0}}})$ at $z=s_0$ as $f'$ ranges over  $\mathcal{S}(X_{P \cap M_{\beta_0}}(\A_F)).$  \qed
\end{cor}

\section{On the poles of degenerate Eisenstein series} \label{sec:poles}

We assume for this section that $F$ is a number field.  Our goal here is to verify conjectures \ref{conj:poles:intro} and \ref{conj:poles} in some cases.  Without loss of generality we take $G=M_{\beta_0}$, thus $P \leq G$ is now a maximal parabolic subgroup containing our fixed Borel $B$ and $P^*$ is the opposite parabolic.  Let $Q$ be the unique maximal parabolic subgroup containing $B$ that is conjugate to $P^*.$

Let $K \leq G(\A_F)$ be a maximal compact subgroup and let 
$$
\CC_+:=\{z \in \CC:\mathrm{Re}(z)>0\}.
$$

\begin{lem} \label{lem:enough}
 To prove Conjecture \ref{conj:poles:intro} it suffices to show 
 that for each character $ \chi \in \widehat{A_{\GG_m} F^\times \backslash \A_F^\times}$ there is a finite set $\Upsilon(\chi) \subset \CC_+$ such that if 
 $
E(g,\Phi^{\chi_s}_P)
$ or $E(g,\Phi^{\chi_s}_Q)$
has a pole at $s=s_0$ with $\mathrm{Re}(s_0)>0$ for any holomorphic $K$-finite section $\Phi^{\chi_s}_P \in I_P(\chi_s)$ or $\Phi^{\chi_s}_Q \in I_Q(\chi_s)$  then $s \in \Upsilon(\chi).$
\end{lem}

 \begin{proof}
 By the observations in the proof of Lemma \ref{lem:holo} if $a_{P|P}(\chi_s)$ has a pole at $s=s_0$ for $\mathrm{Re}(s_0) \geq 0$ then $s_0=0.$  Moreover, the order of the pole is bounded by an integer depending only on $P$ and $G.$  Thus for any $f \in \mathcal{S}(X_P(\A_F))$ Theorem \ref{thm:ES:Langlands} implies that the only possible pole of $E(g,f_{\chi_s})$ on the line $\mathrm{Re}(s)=0$ is at $s=0$ and its order is bounded in a sense depending only on $P$ and $G.$
 
 Now consider poles of $E(g,f_{\chi_s})$ for $\mathrm{Re}(s)<0.$  Using notation from the proof of Theorem \ref{thm:ES:Langlands} and arguing as in that proof we have
 $$
 E(g,f_{\chi_s})=E^*(g,\mathcal{F}_{P|P^*}(f)^*_{\chi_s})=E_Q(g,j(\mathcal{F}_{P|P^*}(f)^*_{\chi_s})).
 $$
Thus the order of the pole of $E(g,f_{\chi_s})$ at $s_0$ is equal to the order of the pole of $E_Q(g,j(\mathcal{F}_{P|P^*}(f)^*_{\chi_{s}}))$
at $s_0.$
 Since $\mathcal{F}_{P|P^*}(f) \in \mathcal{S}(X_{P^*}(\A_F))$  and $a_{P^*|P^*}((\chi_s)^{-1})$ is holomorphic for $\mathrm{Re}(s)<0$ by Lemma \ref{lem:holo2} we have that $\mathcal{F}_{P|P^*}(f)^*_{\chi_s} \in I_{P^*}^*(\chi_s)$ is a holomorphic section for $\mathrm{Re}(s)<0$. This implies  $j(\mathcal{F}_{P|P}(f)^*_{\chi_{s}})$ is a holomorphic section of $I_Q((\chi^{-1})_{-s})$ for $\mathrm{Re}(s)<0.$  The lemma follows.
\end{proof}
 
The proof of the following lemma is analogous:
 
\begin{lem} \label{lem:enough2}
 To prove Conjecture \ref{conj:poles} it suffices to show 
 that there is a finite set $\Upsilon \subset \CC_+$ and an integer $n>0$ such that if 
 $
E(g,\Phi^{\chi_s}_P)
$ or $E(g,\Phi^{\chi_s}_Q)$
has a pole at $s=s_0$ with $\mathrm{Re}(s_0)>0$ for any $\chi \in \widehat{A_{\GG_m} F^\times \backslash \A_F^\times}$ and holomorphic $K$-finite sections $\Phi^{\chi_s}_P \in I_P(\chi_s)$ or $\Phi^{\chi_s}_Q \in I_Q(\chi_s)$ then  $s_0 \in \Upsilon$ and $\chi^n=1.$ \qed
\end{lem}

\begin{thm}
Suppose that $G=\mathrm{Sp}_{2n}$ and that $ P$ is the Siegel parabolic, that is, the parabolic subgroup with Levi subgroup isomorphic to $\GL_n.$  Then Conjecture \ref{conj:poles} is true.
\end{thm}

\begin{proof}
We use the results of \cite{Ikeda:poles:triple}.  We point out that
$a_{P|P}(\chi_s)=a_{I_{2n}}(\chi_s)$ in the notation of loc.~cit.~by \cite[\S 4.3]{Getz:Hsu:Leslie}.  
Thus the theorem follows from Lemma \ref{lem:enough2}, the fact that $a_{P|P}(\chi_s)$ is holomorphic and nonzero for $\mathrm{Re}(s) > 0$ by Lemma \ref{lem:holo2}, and the Corollary of 
\cite[Proposition 1.6]{Ikeda:poles:triple}.
\end{proof}

\begin{thm} \label{thm:GLn} Suppose that $G=\SL_n$.  Then Conjecture \ref{conj:poles:intro} is true. 
\end{thm}

\begin{proof}
This can be derived using Lemma \ref{lem:enough} and the same arguments as those proving \cite[Theorem 1.4]{Hanzer:Muic}.  The proof comes down to two statements explained in \cite[\S 6]{Hanzer:Muic}.  The first is the statement that a certain quotient of completed $L$-functions depending on $\chi$ has only finitely many poles for $\mathrm{Re}(s)>0$.   The second is that a certain normalized intertwining operator has only finitely many poles.  These are proven exactly as in loc.~cit.~The required facts on induced representations of $\GL_2$ may be found in \cite{Jacquet:Langlands}.
\end{proof}

Much of the work towards proving Conjecture \ref{conj:poles} for arbitrary parabolic subgroups of  symplectic groups is contained in \cite{Hanzer:NS}, and much of the work towards proving Conjecture \ref{conj:poles} for general linear groups is contained in \cite{Hanzer:Muic}. The additional steps necessary seem to require careful analysis of the reducibility points of local principal series representations at the archimedean places.  More generally, the nonarchimedean work in  \cite{Hsu} reduces Conjecture \ref{conj:poles} to a related conjecture involving only the archimedean theory  for any simple group not of type $E$ or $F.$

\appendix

\section{Possible generalizations} \label{App}

In this paper we have focused on series related to degenerate Eisenstein series as opposed to cuspidal Eisenstein series.  One of our reasons for this decision is that the the Poisson summation conjecture is not yet known for the spherical varieties attached to cuspidal Eisenstein series.  

This leads one to two interconnected questions:  \begin{enumerate}
\item Can one prove analogous results for series related to general Eisenstein series (e.g.~cuspidal Eisenstein series)?
\item Can one work directly with Eisenstein series without explicit reference to the Poisson summation conjecture?
\end{enumerate}
In this appendix we explain our expected answers  to these questions.  

We maintain the notation adopted earlier in the paper.  In particular $P < P' \leq G$ are a pair of parabolic subgroups with $P$ maximal in $P',$ $Y \subset G$ is a subscheme stable under left multiplication by $P',$ etc.
Let $\sigma$ be an irreducible $A_M \backslash M(\A_F)$-subrepresentation of $L^2(A_MM(F) \backslash M(\A_F)).$  We can then form the global induced representations 
$$
I_P(\sigma):=\mathrm{Ind}_P^G(\sigma) \quad \textrm{ and }I_{P^*}^*(\sigma):=\mathrm{Ind}_{P^*}^G(\sigma)
$$
of $G(\A_F)$ and the subspaces \begin{align*}
I_P(\sigma)^0:=\mathrm{Ind}_P^G(\sigma)^0<I_{P}(\sigma),\\
I_{P^*}^*(\sigma)^0:=\mathrm{Ind}_{P^*}^G(\sigma)^0<I_{P^*}^*(\sigma)
\end{align*}
in the usual manner \cite[\S 10.1-10.3]{GetzHahn:Book}.  For $\varphi \in I_P(\sigma)^0$ and $\varphi^* \in I_{P^*}^*(\sigma)^0$ we form the Eisenstein series
\begin{align}
E(x,\varphi,s):&=\sum_{\gamma \in P(F) \backslash G(F)}\varphi(\gamma x)e^{\langle H_P(\gamma x),s+\rho_P\rangle}\\
E(x,\varphi^*,s):&=\sum_{\gamma \in P^*(F) \backslash G(F)}\varphi^*(\gamma x)e^{\langle H_{P}(\gamma x),s\rangle}e^{\langle H_{P^*}(\gamma x),\rho_{P^*}\rangle}.
\end{align}
Here we are identifying $\CC^{k+1}$ with the vector space $\mathfrak{a}_{M\CC}^*$ of loc.~cit. 

In this context generalized Schubert Eisenstein series are the sums
\begin{align}
    E_{Y_P}(x,\varphi,s):&=\sum_{y \in M^{\mathrm{ab}}(F) \backslash Y_P(F)} \varphi(yx)e^{\langle H_P(\gamma x),s+\rho_P\rangle}\\
    E_{Y_{P^*}}^*(x,\varphi^*,s):&=\sum_{y \in M^{\mathrm{ab}}(F) \backslash Y_{P^*}(F)} \varphi^*(yx)e^{\langle H_{P}(\gamma x),s\rangle}e^{\langle H_{P^*}(\gamma x),\rho_{P^*}\rangle}
\end{align}
Assume $\mathrm{Re}(s_i)$  is sufficiently large for $1 \leq i \leq k.$  If $\mathrm{Re}(s_0)$ is sufficiently large (resp.~small) then $E_{Y_P}(x,\varphi,s)$ (resp.~$E^*_{Y_{P^*}}(x,\varphi^*,s)$) converges absolutely by comparison with $E(x,\varphi,s),$
(resp.~$E^*(x,\varphi^*,s)$) \cite[Proposition 10.3.1]{GetzHahn:Book}.   We expect that $E_{Y_P}(x,\varphi,s)$ and $E_{Y_{P^*}}(x,\varphi^*,s)$ admit meromorphic continuations in $s_0$ and a functional equation relating one to the other.

In order to prove this one could probably roughly argue as we have in this paper.  One would want to break $E_{Y_P}(x,\varphi,s)$ into an infinite (but convergent) sum of Eisenstein series on $M_{\beta_0}$ as before.  One would then try to prove that the well-known properties of Eisenstein series under intertwining operators are inherited by these sums.  In other words, one would want to study intertwining operators in families.  The key difficulty to overcome to complete the argument is 
to understand this family of intertwining operators.  This may require normalization.
The argument would also be complicated by the necessity of understanding infinite sums of Eisenstein series outside their ranges of absolute convergence.

To relate this to the philosophy of the current paper we observe that  the  Fourier transform is essentially a lift to the Schwartz space of a normalized intertwining operator.  Moreover it is easier (at least for the authors) to understand Poisson summation formulae in families as opposed to infinite linear combinations of meromorphically continued functions. In the former case one is always working with absolutely convergent sums, and not functions only defined via meromorphic continuation.
This gives a conceptual advantage of working with 
Schwartz spaces.


\bibliography{refs}

\newcommand{\etalchar}[1]{$^{#1}$}
\def\polhk#1{\setbox0=\hbox{#1}{\ooalign{\hidewidth
  \lower1.5ex\hbox{`}\hidewidth\crcr\unhbox0}}}
\begin{thebibliography}{ABD{\etalchar{+}}66}

\bibitem[ABD{\etalchar{+}}66]{SGA3}
M.~Artin, J-E. Bertin, M.~Demazure, A.~Grothendieck, P.~Gabriel, M.~Raynaud,
  and J-P. Serre.
\newblock {\em Sch\'emas en groupes}.
\newblock S\'eminaire de G\'eom\'etrie Alg\'ebrique de l'Institut des Hautes
  \'Etudes Scientifiques. Institut des Hautes \'Etudes Scientifiques, Paris,
  1963/1966.

\bibitem[Art05]{ArthurIntro}
J.~Arthur.
\newblock An introduction to the trace formula.
\newblock In {\em Harmonic analysis, the trace formula, and {S}himura
  varieties}, volume~4 of {\em Clay Math. Proc.}, pages 1--263. Amer. Math.
  Soc., Providence, RI, 2005.

\bibitem[BB11]{Blomer_Brumley_Ramanujan_Annals}
V.~Blomer and F.~Brumley.
\newblock On the {R}amanujan conjecture over number fields.
\newblock {\em Ann. of Math. (2)}, 174(1):581--605, 2011.

\bibitem[BC14]{Bump:Choie}
D.~Bump and Y.~Choie.
\newblock Schubert {E}isenstein series.
\newblock {\em Amer. J. Math.}, 136(6):1581--1608, 2014.

\bibitem[BCR98]{BCR}
J.~Bochnak, M.~Coste, and M-F. Roy.
\newblock {\em Real algebraic geometry}, volume~36 of {\em Ergebnisse der
  Mathematik und ihrer Grenzgebiete (3) [Results in Mathematics and Related
  Areas (3)]}.
\newblock Springer-Verlag, Berlin, 1998.
\newblock Translated from the 1987 French original, Revised by the authors.

\bibitem[BG02]{Braverman:Gaitsgory}
A.~Braverman and D.~Gaitsgory.
\newblock Geometric {E}isenstein series.
\newblock {\em Invent. Math.}, 150(2):287--384, 2002.

\bibitem[BK00]{BK-lifting}
A.~Braverman and D.~Kazhdan.
\newblock {$\gamma$}-functions of representations and lifting.
\newblock {\em Geom. Funct. Anal.}, (Special Volume, Part I):237--278, 2000.
\newblock With an appendix by V. Vologodsky, GAFA 2000 (Tel Aviv, 1999).

\bibitem[BK02]{BK:normalized}
A.~Braverman and D.~Kazhdan.
\newblock Normalized intertwining operators and nilpotent elements in the
  {L}anglands dual group.
\newblock {\em Mosc. Math. J.}, 2(3):533--553, 2002.
\newblock Dedicated to Yuri I. Manin on the occasion of his 65th birthday.

\bibitem[BL24]{Bernstein:Lapid}
J.~Bernstein and E.~Lapid.
\newblock On the meromorphic continuation of {E}isenstein series.
\newblock {\em J. Amer. Math. Soc.}, 37(1):187--234, 2024.

\bibitem[Bor79]{Borel:Corvallis}
A.~Borel.
\newblock Automorphic {$L$}-functions.
\newblock In {\em Automorphic forms, representations and {$L$}-functions
  ({P}roc. {S}ympos. {P}ure {M}ath., {O}regon {S}tate {U}niv., {C}orvallis,
  {O}re., 1977), {P}art 2}, Proc. Sympos. Pure Math., XXXIII, pages 27--61.
  Amer. Math. Soc., Providence, R.I., 1979.

\bibitem[BT72]{Borel:Tits:Compl}
A.~Borel and J.~Tits.
\newblock Compl\'{e}ments \`a l'article: ``{G}roupes r\'{e}ductifs''.
\newblock {\em Inst. Hautes \'{E}tudes Sci. Publ. Math.}, (41):253--276, 1972.

\bibitem[Bum97]{Bump:AFR}
D.~Bump.
\newblock {\em Automorphic forms and representations}, volume~55 of {\em
  Cambridge Studies in Advanced Mathematics}.
\newblock Cambridge University Press, Cambridge, 1997.

\bibitem[CPS99]{Cogdell:PS:ConverseII}
J.~W. Cogdell and I.~I. Piatetski-Shapiro.
\newblock Converse theorems for {${\rm GL}_n$}. {II}.
\newblock {\em J. Reine Angew. Math.}, 507:165--188, 1999.

\bibitem[ES18]{Elazar:Shaviv}
B.~Elazar and A.~Shaviv.
\newblock Schwartz functions on real algebraic varieties.
\newblock {\em Canad. J. Math.}, 70(5):1008--1037, 2018.

\bibitem[Gae22]{Gaetz}
C.~Gaetz.
\newblock Spherical {S}chubert varieties and pattern avoidance.
\newblock {\em Selecta Math. (N.S.)}, 28(2):Paper No. 44, 9, 2022.

\bibitem[{Get}22]{Getz:Summ}
J.~R. {Getz}.
\newblock {Summation formulae for quadrics}.
\newblock {\em arXiv e-prints}, page arXiv:2201.02583, January 2022.

\bibitem[GH20]{Getz:Hsu}
J.~R. {Getz} and C-H. {Hsu}.
\newblock {The Fourier transform for triples of quadratic spaces}.
\newblock {\em arXiv e-prints}, page arXiv:2009.11490, September 2020.

\bibitem[GH24]{GetzHahn:Book}
J.~R. Getz and H.~Hahn.
\newblock {\em An introduction to automorphic representations with a view
  toward trace formulae}, volume 300 of {\em Graduate Texts in Mathematics}.
\newblock Springer, Cham, 2024.

\bibitem[GHL23]{Getz:Hsu:Leslie}
J.~R. {Getz}, C-H. {Hsu}, and S.~{Leslie}.
\newblock {Harmonic analysis on certain spherical varieties}.
\newblock {\em J. Eur. Math. Soc.}, 2023.

\bibitem[GL19]{Getz:Liu:Triple}
J.~R. Getz and B.~Liu.
\newblock A summation formula for triples of quadratic spaces.
\newblock {\em Adv. Math.}, 347:150--191, 2019.

\bibitem[GL21]{Getz:Liu:BK}
J.~R. Getz and B.~Liu.
\newblock A refined {P}oisson summation formula for certain
  {B}raverman-{K}azhdan spaces.
\newblock {\em Sci. China Math.}, 64(6):1127--1156, 2021.

\bibitem[GMP17]{Garland:Miller:Patnaik}
H.~Garland, S.~D. Miller, and M.~M. Patnaik.
\newblock Entirety of cuspidal {E}isenstein series on loop groups.
\newblock {\em Amer. J. Math.}, 139(2):461--512, 2017.

\bibitem[GPSR87]{GPSR:LNM}
S.~Gelbart, I.~Piatetski-Shapiro, and S.~Rallis.
\newblock {\em Explicit constructions of automorphic {$L$}-functions}, volume
  1254 of {\em Lecture Notes in Mathematics}.
\newblock Springer-Verlag, Berlin, 1987.

\bibitem[Gro98]{Gross:Sat}
B.~H. Gross.
\newblock On the {S}atake isomorphism.
\newblock In {\em Galois representations in arithmetic algebraic geometry
  ({D}urham, 1996)}, volume 254 of {\em London Math. Soc. Lecture Note Ser.},
  pages 223--237. Cambridge Univ. Press, Cambridge, 1998.

\bibitem[{Gu}21]{Gu}
M.~P. {Gu}.
\newblock {Automorphic-twisted summation formulae for pairs of quadratic
  spaces}.
\newblock {\em arXiv e-prints}, page arXiv:2102.06113, February 2021.

\bibitem[GW10]{Gortz_Wedhorn}
U.~G{\"o}rtz and T.~Wedhorn.
\newblock {\em Algebraic geometry {I}}.
\newblock Advanced Lectures in Mathematics. Vieweg + Teubner, Wiesbaden, 2010.
\newblock Schemes with examples and exercises.

\bibitem[Han18]{Hanzer:NS}
M.~Hanzer.
\newblock Non-{S}iegel {E}isenstein series for symplectic groups.
\newblock {\em Manuscripta Math.}, 155(1-2):229--302, 2018.

\bibitem[HM15]{Hanzer:Muic}
M.~Hanzer and G.~Mui\'{c}.
\newblock On the images and poles of degenerate {E}isenstein series for {${\rm
  GL}(n,\Bbb{A_Q})$} and {${\rm GL}(n,\Bbb R)$}.
\newblock {\em Amer. J. Math.}, 137(4):907--951, 2015.

\bibitem[{Hsu}21]{Hsu}
C-H. {Hsu}.
\newblock {Asymptotics of Schwartz functions: nonarchimedean}.
\newblock {\em arXiv e-prints}, page arXiv:2112.02403, December 2021.

\bibitem[HY22]{Hodges}
R.~Hodges and A.~Yong.
\newblock Coxeter combinatorics and spherical {S}chubert geometry.
\newblock {\em J. Lie Theory}, 32(2):447--474, 2022.

\bibitem[Ike92]{Ikeda:poles:triple}
T.~Ikeda.
\newblock On the location of poles of the triple {$L$}-functions.
\newblock {\em Compositio Math.}, 83(2):187--237, 1992.

\bibitem[Jan03]{Jantzen}
J.~C. Jantzen.
\newblock {\em Representations of algebraic groups}, volume 107 of {\em
  Mathematical Surveys and Monographs}.
\newblock American Mathematical Society, Providence, RI, second edition, 2003.

\bibitem[JL70]{Jacquet:Langlands}
H.~Jacquet and R.~P. Langlands.
\newblock {\em Automorphic forms on {${\rm GL}(2)$}}.
\newblock Lecture Notes in Mathematics, Vol. 114. Springer-Verlag, Berlin-New
  York, 1970.

\bibitem[JLZ24]{Jiang:Luo:Zhang}
D.~Jiang, Z.~Luo, and L.~Zhang.
\newblock Harmonic analysis and gamma functions on symplectic groups.
\newblock {\em Mem. Amer. Math. Soc.}, 295(1473):vii+89, 2024.

\bibitem[Laf14]{LafforgueJJM}
L.~Lafforgue.
\newblock Noyaux du transfert automorphe de {L}anglands et formules de
  {P}oisson non lin\'eaires.
\newblock {\em Jpn. J. Math.}, 9(1):1--68, 2014.

\bibitem[Lai80]{Lai:Tama}
K.~F. Lai.
\newblock Tamagawa number of reductive algebraic groups.
\newblock {\em Compositio Math.}, 41(2):153--188, 1980.

\bibitem[Lan71]{Langlands:EP}
R.~P. Langlands.
\newblock {\em Euler products}.
\newblock Yale University Press, New Haven, Conn.-London, 1971.
\newblock A J. K. Whittemore Lecture in Mathematics given at Yale University,
  1967, Yale Mathematical Monographs, 1.

\bibitem[Lan76]{Langlands:FE}
R.~P. Langlands.
\newblock {\em On the functional equations satisfied by {E}isenstein series}.
\newblock Lecture Notes in Mathematics, Vol. 544. Springer-Verlag, Berlin-New
  York, 1976.

\bibitem[Lap08]{LapidRem}
E.~M. Lapid.
\newblock A remark on {E}isenstein series.
\newblock In {\em Eisenstein series and applications}, volume 258 of {\em
  Progr. Math.}, pages 239--249. Birkh\"{a}user Boston, Boston, MA, 2008.

\bibitem[Mil17]{Milne:AGbook}
J.~S. Milne.
\newblock {\em Algebraic groups}, volume 170 of {\em Cambridge Studies in
  Advanced Mathematics}.
\newblock Cambridge University Press, Cambridge, 2017.
\newblock The theory of group schemes of finite type over a field.

\bibitem[MW95]{MW:Spectral:Decomp:ES}
C.~M{\oe}glin and J.-L. Waldspurger.
\newblock {\em Spectral decomposition and {E}isenstein series}, volume 113 of
  {\em Cambridge Tracts in Mathematics}.
\newblock Cambridge University Press, Cambridge, 1995.
\newblock Une paraphrase de l'{\'E}criture [A paraphrase of Scripture].

\bibitem[Ng{\^o}14]{NgoSums}
B.~C. Ng{\^o}.
\newblock On a certain sum of automorphic {$L$}-functions.
\newblock In {\em Automorphic forms and related geometry: assessing the legacy
  of {I}. {I}. {P}iatetski-{S}hapiro}, volume 614 of {\em Contemp. Math.},
  pages 337--343. Amer. Math. Soc., Providence, RI, 2014.

\bibitem[Ng{\^o}20]{Ngo:Hankel}
B.~C. Ng{\^o}.
\newblock Hankel transform, {L}anglands functoriality and functional equation
  of automorphic {$L$}-functions.
\newblock {\em Jpn. J. Math.}, 15(1):121--167, 2020.

\bibitem[Poo17]{Poonen:Rational}
B.~Poonen.
\newblock {\em Rational points on varieties}, volume 186 of {\em Graduate
  Studies in Mathematics}.
\newblock American Mathematical Society, Providence, RI, 2017.

\bibitem[Ray67]{Raynaud:Passage}
M.~Raynaud.
\newblock Passage au quotient par une relation d'\'{e}quivalence plate.
\newblock In {\em Proc. {C}onf. {L}ocal {F}ields ({D}riebergen, 1966)}, pages
  78--85. Springer, Berlin, 1967.

\bibitem[Sak12]{SakellaridisSph}
Y.~Sakellaridis.
\newblock Spherical varieties and integral representations of {$L$}-functions.
\newblock {\em Algebra Number Theory}, 6(4):611--667, 2012.

\bibitem[Sha88]{Shahidi:Ramanujan}
F.~Shahidi.
\newblock On the {R}amanujan conjecture and finiteness of poles for certain
  {$L$}-functions.
\newblock {\em Ann. of Math. (2)}, 127(3):547--584, 1988.

\bibitem[Sha18]{Shahidi:FT}
F.~Shahidi.
\newblock On generalized {F}ourier transforms for standard {$L$}-functions.
\newblock In {\em Geometric aspects of the trace formula}, Simons Symp., pages
  351--404. Springer, Cham, 2018.

\bibitem[Spr98]{Springer:LAG}
T.~A. Springer.
\newblock {\em Linear algebraic groups}, volume~9 of {\em Progress in
  Mathematics}.
\newblock Birkh\"{a}user Boston, Inc., Boston, MA, second edition, 1998.

\bibitem[SV17]{SV}
Y.~Sakellaridis and A.~Venkatesh.
\newblock Periods and harmonic analysis on spherical varieties.
\newblock {\em Ast\'{e}risque}, (396):viii+360, 2017.

\bibitem[War72]{Warner:semisimple:I}
G.~Warner.
\newblock {\em Harmonic analysis on semi-simple {L}ie groups. {I}}.
\newblock Springer-Verlag, New York-Heidelberg, 1972.
\newblock Die Grundlehren der mathematischen Wissenschaften, Band 188.

\bibitem[Wei74]{Weil:Basic:NT}
A.~Weil.
\newblock {\em Basic number theory}.
\newblock Springer-Verlag, New York-Berlin, third edition, 1974.
\newblock Die Grundlehren der Mathematischen Wissenschaften, Band 144.

\end{thebibliography}
\bibliographystyle{alpha}

\end{document}